\documentclass[11pt]{article}
\usepackage{amsmath, amssymb, amsthm} %DIF >
\usepackage[T1]{fontenc}
\usepackage{dsfont}
\usepackage[colorlinks=true, allcolors=blue,backref]{hyperref}
\usepackage[capitalize,nameinlink]{cleveref}
\usepackage{url}
\usepackage[utf8]{inputenc}

\usepackage{tikz}
\usepackage{comment}

\allowdisplaybreaks
\expandafter\let\expandafter\oldproof\csname\string\proof\endcsname
\let\oldendproof\endproof
\renewenvironment{proof}[1][\proofname]{%
	\oldproof[\bf #1]%
}{\oldendproof}

\allowdisplaybreaks

\theoremstyle{plain}

\newtheorem{lemma}{Lemma}[section]
\newtheorem{theorem}[lemma]{Theorem}
\newtheorem{claim}[lemma]{Claim}
\newtheorem{proposition}[lemma]{Proposition}

\newtheorem{definition}[lemma]{Definition}

\newcommand{\Bin}{\ensuremath{\textrm{Bin}}}

\newcommand{\mc}{\mathcal}

\RequirePackage[normalem]{ulem} %DIF PREAMBLE
\RequirePackage{color}\definecolor{RED}{rgb}{1,0,0}\definecolor{BLUE}{rgb}{0,0,1} %DIF PREAMBLE
\newcommand{\abs}[1]{\left\vert {#1} \right\vert}
\newcommand{\floor}[1]{\left\lfloor {#1} \right\rfloor}
\newcommand{\ceil}[1]{\left\lceil {#1} \right\rceil}

\global\long\def\zj#1{\textcolor{cyan}{\textbf{[ZJ comments:} #1\textbf{]}}}

\title{ \vspace{-0.8cm}Difference-Isomorphic Graph Families
}

\author{
Lior Gishboliner\thanks{Department of Mathematics, ETH, Z\"urich, Switzerland. Research supported in part by SNSF grant 200021\_196965. Email: \textbf{\{lior.gishboliner, zhihan.jin, benjamin.sudakov\}@math.ethz.ch}.}
\and 
Zhihan Jin\footnotemark[1]
\and Benny Sudakov\footnotemark[1]
}

\date{}

\begin{document}

\maketitle

\begin{abstract}
    Many well-studied problems in extremal combinatorics deal with the maximum possible size of a family of objects in which every pair of objects satisfies a given restriction. One problem of this type was recently raised by Alon, Gujgiczer, K\"orner, Milojevi\'c and Simonyi. They asked to determine the maximum size of a family $\mathcal{G}$ of graphs on $[n]$, such that for every two $G_1,G_2 \in \mathcal{G}$, the graphs $G_1 \setminus G_2$ and $G_2 \setminus G_1$ are isomorphic. We completely resolve this problem by showing that this maximum is exactly $2^{\frac{1}{2}\big(\binom{n}{2} - \floor{\frac{n}{2}}\big)}$ and characterizing all the extremal constructions. We also prove an analogous result for $r$-uniform hypergraphs.
\end{abstract}

\section{Introduction}

% \subsection{Related work}
% For $n \ge r \ge 2$, let $\sim$ be a symmetric binary relation between $r$-graphs on $[n]$.
% A natural question is to study the maximum size of a family $\cal G$ of $r$-graphs on $[n]$ where $G_1\sim G_2$ for every two distinct $r$-graphs $G_1,G_2 \in \mc{G}$.
% There are some examples.
% \begin{itemize}
%     \item $G_1 \sim G_2$ if $G_1\cap G_2$ contains a triangle (or $K_t$).
%     \item $G_1 \sim G_2$ if $G_1 \cap G_2$ contains a Hamilton cycle. See ''A Note on Hamiltonian-Intersecting Families of Graphs'' for reference.
%     \item $G_1 \sim G_2$ if $G_1 \triangle G_2$ contains triangle or a Hamilton cycle (graph-codes).
%     \item $G_1 \sim G_2$ if $G_1\setminus G_2 \cong G_2\setminus G_1$ (this work).
% \end{itemize}

% \subsection{Results}

Some of the most well-studied problems in extremal combinatorics ask for the maximum size of a collection of objects where every pair of objects satisfies some given property. For example, the celebrated Erd\H{o}s--Ko--Rado theorem \cite{EKR} determines the maximum size of a family of subsets of size $k$ of a set of size $n$ in which every two subsets intersect. 
This is one of the fundamental results of extremal set theory and has many extensions and generalizations.
Another example is the famous theorems of Ray-Chaudhuri and Wilson \cite{RC_W} and of Frankl and Wilson \cite{FW}, bounding the size of a family of subsets of $[n]$ where the intersection of any two sets belongs to a given set $L$. In addition to their intrinsic interest, these theorems have surprising applications to Ramsey theory and geometry, see \cite{BF_Book}. We also refer the reader to \cite{FT_Book} for an overview of extremal set theory. 

An interesting variant of the above problems is obtained by equipping the ground set with some combinatorial structure, such as being the edge set of a complete graph/hypergraph (see~\cite{Ellis} for many examples).
This direction of research was initiated by Simonovits and S\'os \cite{SS} and received considerable attention in recent decades. 
For example, Ellis, Filmus and Friedgut \cite{EFF} proved a conjecture of Simonovits and S\'os (and improved an earlier bound in \cite{CGFS}) by showing that if $\mathcal{G}$ is a family of graphs on $[n]$ where the intersection of any two graphs contains a triangle, then $|\mathcal{G}| \leq 2^{\binom{n}{3}-3}$. Note that this bound is tight by taking the family of all graphs containing a given triangle. See also \cite{Berger_et_al,BZ_K4,IRT} for related results. It is worth mentioning that the proofs of this theorem and other results mentioned in the above two paragraphs use various interesting and important techniques, such as probabilistic arguments, linear-algebraic methods, entropy, and discrete Fourier analysis. 

Another noteworthy problem on graph families is the following question of Gowers: Is it true that if $\mathcal{G}$ is a family of graphs on $[n]$ with no two graphs $G_1,G_2 \in \mathcal{G}$ satisfying that $G_1 \subseteq G_2$ and $G_2 \setminus G_1$ is a clique, then $|\mathcal{G}| \leq o\left(2^{\binom{n}{2}}\right)$. As explained by Gowers in his blog post \cite{Gowers}, this problem is related to the first unknown case of the polynomial density Hales-Jewett theorem. 

Recently, Alon, Gujgiczer, K\"orner, Milojevi\'c and Simonyi \cite{AGKMS} initiated the study of so-called {\em graph codes}, which also falls into the general framework of studying graph families with pairwise restrictions. 
Let $\mathcal{F}$ be a family of graphs on $[n]$ which is closed under isomorphism. Using the terminology from \cite{AGKMS}, a set $\mathcal{G}$ of graphs on $[n]$ is called {\em $\mathcal{F}$-good} if $G_1 \oplus G_2 \in \mathcal{F}$ for every $G_1,G_2 \in \mathcal{G}$, where $G_1 \oplus G_2$ is the symmetric difference of the (edge sets of) $G_1,G_2$. 
Note that if $\mathcal{F}$ is the set of all graphs with at least $d$ edges, then an $\mathcal{F}$-good graph family is simply a binary code (in the sense of coding theory) with distance at least $d$. Thus, the study of $\mathcal{F}$-good graph families is a variant of the study of codes where the ground set is given the structure of $E(K_n)$.
In \cite{AGKMS} and in a subsequent work \cite{Alon_Codes}, the maximum size of $\mathcal{F}$-good families was estimated for various choices of $\mathcal{F}$, such as the set of all connected graphs, all Hamiltonian graphs, and all $\mathcal{H}$-free graphs for certain graph-families $\mathcal{H}$. See also \cite{Alon_connCodes,BGMW_Codes} for related results.
As in coding theory, it is also natural to study {\em linear} $\mathcal{F}$-good graph families, where a graph family is called linear if it is closed under symmetric difference. As a notable example, Alon \cite{Alon_Codes} showed that if $\mathcal{G}$ is a linear graph family such that $G_1 \oplus G_2$ is not a clique for any $G_1,G_2 \in \mathcal{G}$, then $|\mathcal{G}| \leq 2^{\binom{n}{2}-\floor{\frac{n}{2}}}$, and this is best possible.

For two (hyper-)graphs $G_1,G_2$ on $[n]$, we use $G_1 \setminus G_2$ to denote the (hyper-)graph on $[n]$ with edge-set $E(G_1) \setminus E(G_2)$. 
Some of the constructions of $\mathcal{F}$-good graph families $\mathcal{G}$ in \cite{AGKMS} have the property that $G_1 \setminus G_2$ is isomorphic to $G_2 \setminus G_1$ for every $G_1,G_2 \in \mathcal{G}$. 
This led Alon et al.~\cite{AGKMS} to ask the question of determining the largest possible size of a graph family with this property. We will call such families {\em difference-isomorphic}:
\begin{definition}
    A family $\mathcal{G}$ of $r$-graphs on $[n]$ is called {\em difference-isomorphic} if for every two graphs $G_1,G_2 \in \mathcal{G}$, it holds that $G_1 \setminus G_2$ is isomorphic to $G_2 \setminus G_1$. 
\end{definition}
% Let $f_r(n)$ denote the largest possible size of a difference-isomorphic family of $r$-graphs on $[n]$. 
Observe that the family of all perfect matchings is difference-isomorphic and has size $n^{(\frac{1}{2}+o(1))n}$. 
Alon et al.~\cite{AGKMS} constructed a slightly larger difference-isomorphic graph family by taking graphs consisting of vertex-disjoint stars. 
More precisely, fix a partition of $[n]$ into two sets $A,B$ of size $|A| = k, |B| = n-k$, and consider the family of all graphs which consist of $k$ disjoint stars each with $\frac{n}{k}-1$ leaves, where each star has its center in $A$ and leaves in $B$. 
This family is difference-isomorphic and has size $n^{(1+o(1))n}$ for the optimal choice of $k$, namely $k \approx \frac{n}{\log n}$.
Given these constructions, one might expect that difference-isomorphic graph families are quite small. Rather surprisingly, it turns out that this is very far from the truth. Let us now describe a much larger difference-isomorphic graph family. 
Fix a permutation $\psi$ that consists only of 2-cycles (assuming that $n$ is even); in particular, $\psi^2$ is the identity, i.e., $\psi$ is an involution.
It is easy to see that for every edge $e \in E(K_n)$, there is an edge $f \in E(K_n)$ such that $\psi$ maps $e$ to $f$ and $f$ to $e$ (possibly $e=f$). Now, take $\mathcal{G}$ to be the set of all graphs which contain exactly one of the edges $e,f$ for each such pair $(e,f)$. Then $\mathcal{G}$ is difference-isomorphic (in fact, $\psi$ is the isomorphism witnessing this for every $G_1,G_2 \in \mathcal{G}$), and has size $2^{\frac{1}{2}\big(\binom{n}{2} - \floor{\frac{n}{2}}\big)}$; see \cref{prop:extremal construction} for the full details. As our main result, we shall show that this construction is in fact the largest possible difference-isomorphic family of graphs on $[n]$. This completely resolves the aforementioned question of Alon et al.~\cite{AGKMS}. 
% Due to the isomorphism condition, one might expect the $r$-graphs in the extremal examples to be partially isomorphic and guess that the largest size is of order $n^{O(n)}$, given by the number permutations.
% However, this is not the case -- our main result completely resolves the question mentioned above of Alon et al.~\cite{AGKMS} by determining exactly the largest size of a difference-isomorphic family. 
Moreover, our techniques can also be used to establish an analogous result for $r$-uniform hypergraphs. 
% Somewhat surprisingly, difference-isomorphic families can be quite large: there exist difference-isomorphic $r$-graph families of size $2^{\Theta(n^r)}$.
 % the $r$-graphs in the extremal example are highly non-isomorphic, which is somewhat 
To state it, define
% \begin{equation}\label{eq: f_r(n)}
%     f_r(n) := 
%     \begin{cases}
%         2^{\frac{1}{2}\binom{n}{r}} & 
%         r \text{ is odd and } n \text{ is even,}
%         \\
%         2^{\frac{1}{2}\left(\binom{n}{r} - \binom{\floor{n/2}}{\floor{r/2}}\right)} & \text{otherwise}.
%     \end{cases}
% \end{equation}
\begin{equation}\label{eq: f_r(n)}
    f_r(n) := 
    \begin{cases}
        \frac{1}{2}\binom{n}{r} & 
        r \text{ is odd and } n \text{ is even,}
        \\
        \frac{1}{2}\big(\binom{n}{r} - \binom{\floor{n/2}}{\floor{r/2}}\big) & \text{otherwise}.
    \end{cases}
\end{equation}
\begin{theorem}\label{thm:tight}
   For every $r \geq 2$, there is $n_0 = n_0(r)$ such that for every $n \geq n_0$, the largest possible size of a difference-isomorphic family of $r$-graphs on $[n]$ equals $2^{f_r(n)}$. 
\end{theorem}
% \lg{Say that the theorem doesn't hold for small $n$?}
% \zj{Maybe say that it is not very intuitive because one might guess the graphs in the largest family are very structured or even similar? Because Alon et al. suggested considering stars/matchings.}

As we shall see, the extremal example for \cref{thm:tight} satisfies that all isomorphisms $G_1\setminus G_2 \rightarrow G_2 \setminus G_1$ (for $G_1,G_2 \in \mathcal{G}$) are given by the same permutation, which is an involution. 
We can further strengthen \cref{thm:tight} by proving a stability result (see \cref{theorem: stablility}), saying that if $\mathcal{G}$ does not have this structure, then it has smaller size.

\paragraph{Paper organization:} 
\cref{sec:basics} contains the key notions we will use in the proofs, as well as the construction giving the lower bound in \cref{thm:tight} and a simple proof of an asymptotic version of the upper bound in \cref{thm:tight}. The full proof of \cref{thm:tight} is then given in \cref{sec:proof}. \cref{sec:concluding} contains some concluding remarks. 

% In fact, we acquire a stability result, which shows $\cal G$ is not ``large'' unless $\cal G$ has the following structure.
% \begin{definition}
%     Let $\cal G$ be a difference-isomorphic family of $r$-graphs on $[n]$, and $\varphi\in S_n$ is a permutation. 
%     We say $\cal G$ is {\em given by} $\varphi$ if for every two graphs $G_1,G_2$, it holds that $G_1\setminus G_2$ is isomorphic $G_2\setminus G_1$ by isomorphism $\varphi$ (or $\varphi^{-1}$).
% \end{definition}
% \begin{theorem}\label{theorem: stablility}
%     For every $r \geq 2$, there is $n_0 = n_0(r)$ such that for every $n \geq n_0$ and every difference-isomophic family $\cal G$ of $r$-graphs on $[n]$, if $\cal G$ is not ``given'' by some involution $\varphi \in S_n$ ($\varphi^2$ is the identity), then $\abs{\mc{G}}<\left(1-n^{-20000\sqrt{n}}\right)f_2(n)$ (when $r=2$) while $\abs{\mc{G}}<f_r(n)\cdot e^{-\binom{n}{r-2}/1000}$ (when $r \ge 3$).
% \end{theorem}
% \cref{thm:tight} follows then from \cref{theorem: stablility} and the following simply property.
% \begin{proposition}
%     Let $2 \le r \le n$ and $\cal G$ be a difference-isomorphic family of $r$-graphs on $[n]$.
%     If $\cal G$ is given by some involution $\varphi\in S_n$ (this means $\varphi^2$ is the identity), then $|{\cal G}| \le f_r(n)$.
% \end{proposition}

% \section{Notations and preliminaries}\lg{Change section name}
\section{Definitions, the Construction, and an Approximate Bound}\label{sec:basics}
% For a graph $G$ and a permutation $\varphi \in S_n$, we denote by $\tilde{\varphi}(G)$ the image of $G$ under $\tilde{\varphi}$, namely, $\tilde{\varphi}(G) = \{\tilde{\varphi}(e) : e \in E(G)\}$. 
% We assume that $r$ is a fixed constant throughout the proof.
We begin by introduce the key definitions we will use. 
A permutation $\varphi \in S_n$ on the vertices $[n]$ induces a permutation $\tilde{\varphi} \in S_{\binom{[n]}{r}}$ on the edges in $\binom{[n]}{r}$, namely, $\tilde{\varphi}(e) = \{\varphi(x) : x \in e\}$ for $e \in \binom{[n]}{r}$. 
We will omit the tilde and simply write $\varphi(e)$ when it is clear from the context that we are considering the action of $\varphi$ on the edges. 
Also, note that the uniformity $r$ is not part of the notation, as it will be clear from the context.
For a graph $G$, we write $\varphi(G) = \{\tilde{\varphi}(e) : e \in E(G)\}$ to denote the image of $G$ under $\tilde{\varphi}$. 

For $r$-graphs $G,H$ and a permutation $\varphi \in S_n$, we write $G \overset{\varphi}{\rightarrow} H$ to mean that $\varphi(G \setminus H) = H \setminus G$.
Let $N_{\varphi}(G)$ denote the set of all $r$-graphs $H$ such that $G \overset{\varphi}{\rightarrow} H$.
% $\tilde{\varphi}(G \setminus G') = G' \setminus G$. 
Note that $G \in N_{\varphi}(G)$, and that $G \overset{\varphi}{\rightarrow} H$ if and only if $H \overset{\varphi^{-1}}{\rightarrow} G$. 
In the following definition and lemma, we describe the relation $G \overset{\varphi}{\rightarrow} H$ by giving an equivalent condition. 

\begin{definition}
[Choosable pairs] \label{def: choosable pairs}
    Let $G$ be an $r$-graph and $\varphi$ be a permutation.
    A {\em choosable pair} for $(G,\varphi)$ is an ordered pair $(e,f)\in \binom{[n]}{r}\times\binom{[n]}{r}$ such that $\varphi(e)=f$, $e\in G$ and $f \notin G$.
\end{definition}
We write $\{e,f\}$ for the unordered pair of the choosable pair.
Observe that every $e \in \binom{[n]}{r}$ can be contained in at most one choosable pair. Also, if $(e,f)$ is a choosable pair then $e \neq f$, so an edge $e$ satisfying $\varphi(e) = e$ cannot be part of any choosable pair. 

\begin{lemma}\label{lemma: properties of choosable pairs}
    Let $\varphi \in S_n$ and let $G,H$ be $r$-graphs. Then $G \overset{\varphi}{\rightarrow} H$ if and only if the following two conditions hold:
    \begin{itemize}
        \item[(i)] For every choosable pair $(e,f)$ for $(G,\varphi)$, the graph $H$ contains exactly one of $e$ and $f$.
        \item[(ii)] If $e\in\binom{[n]}{r}$ is not contained in any choosable pair for $(G,\varphi)$, then $G$ and $H$ agree on $e$.
        % it holds that $G(e)=G'(e)$, i.e. $G'$ contains $e$ if and only if $G$ contains $e$.
    \end{itemize}
    In particular, $|N_{\varphi}(G)| \leq 2^m$, where $m$ is the number of choosable pairs of $(G,\varphi)$. 
\end{lemma}
\begin{proof}
    We first assume that $G \overset{\varphi}{\rightarrow} H$ and prove (i)-(ii). For (i), let $(e,f)$ be a choosable pair for $(G,\varphi)$. 
    So $\varphi(e) = f$ and $e \in G$, $f \notin G$. 
    Suppose by contradiction that $H$ contains neither or both of $e,f$. If $e,f \notin H$, then $e \in G \setminus H$, and therefore $f = \varphi(e) \in H \setminus G$ (as $G \overset{\varphi}{\rightarrow} H$), in contradiction to $f \notin H$. Similarly, if $e,f \in H$ then $f \in H \setminus G$, so $e = \varphi^{-1}(f) \in G \setminus H$, in contradiction to $e \in H$.
    For (ii), let $e$ be an edge which does not belong to any choosable pair of $(G,\varphi)$. Suppose by contradiction that $e \in G \setminus H$ or $e \in H \setminus G$. If $e \in G \setminus H$, then $f := \varphi(e) \in H \setminus G$, meaning that $(e,f)$ is a choosable pair. And if $e \in H \setminus G$, then $f := \varphi^{-1}(e) \in G \setminus H$, so $(f,e)$ is a choosable pair. In both cases we got a contradiction. 

    In the other direction, suppose that (i)(ii) hold and let us show that $\varphi(G \setminus H) = H \setminus G$. Let $e \in G \setminus H$. By (ii), $e$ must belong to some choosable pair. This pair cannot have the form $(f,e)$ (for some edge $f$), because this would imply that $e \notin G$. So this pair must have the form $(e,f)$. Then $e \in G$, $f \notin G$, and $f = \varphi(e) \in H \setminus G$, because $H$ contains exactly one of the edges $e,f$ (by (i)) and $e \notin H$. This shows that $\varphi(G \setminus H) \subseteq H \setminus G$. A similar argument shows that $H \setminus G \subseteq \varphi(G \setminus H)$, and thus $\varphi(G \setminus H) = H \setminus G$.
\end{proof}

\subsection{Involutions, the extremal construction, and stability}

An {\em involution} is a permutation $\psi \in S_n$ such that $\psi^2$ is the identity. 
In other words, an involution is a permutation whose every cycle has length 1 or 2. 
As we shall see, involutions play an important role in the proof of \Cref{theorem: stablility}. 
We use the following notation:
\begin{definition}[$\mathcal{C}_1,\mathcal{C}_2$]\label{def: C1 C2}
For an involution $\psi\in S_n$, denote by $\mathcal{C}_1(\psi)$ the set of fixed points of $\tilde{\psi}$, i.e., the set of edges $e \in \binom{[n]}{r}$ satisfying $\psi(e) = e$.
Denote by $\mathcal{C}_2(\psi)$ the set of all $2$-cycles of $\tilde{\psi}$, namely, the set of all (unordered) pairs of distinct edges $\{e,f\}$, $e,f \in \binom{[n]}{r}$, with $\psi(e) = f$ and $\psi(f) = e$.
\end{definition}
Crucially, if $\psi \in S_n$ is an involution then so is $\tilde{\psi} \in S_{\binom{[n]}{r}}$. 
% Observe that if $\psi$ is an involution then so is $\tilde{\psi}$. 
Hence, every edge $e \in \binom{[n]}{r}$ belongs to $\mathcal{C}_1(\psi)$ or to some pair in $\mathcal{C}_2(\psi)$.
The following lemma gives a relation between involutions and the function \nolinebreak $f_r(n)$.

% \begin{definition}
%     Let $2 \leq r \leq n$. 
%     For a permutation $\varphi \in S_n$, let $p_r(\varphi)$ denote the number of $e \in \binom{[n]}{r}$ with $\tilde{\varphi}(e) = e$. Let $p_r(n)$ denote the minimum of $p_r(\varphi)$ over all involutions $\varphi \in S_n$.
% \end{definition}

% \begin{claim}\label{claim: fixed r-sets}
%     $p_r(n) = 0$ if $r$ is odd and $n$ is even, and $p_r(n) = \binom{\floor{n/2}}{\floor{r/2}}$ in all other cases. The minimum is attained by the involution having at most one fixed point. 
% \end{claim}
\begin{lemma}\label{claim: fixed r-sets}
For $2 \leq r \leq n$, $f_r(n)$ is the maximum of $|\mathcal{C}_2(\psi)|$ over all involutions $\psi \in S_n$. The maximum is attained by involutions having at most one fixed point.  
\end{lemma}
\begin{proof}
    Let $\psi$ be an involution. We have $|\mathcal{C}_1(\psi)| + 2|\mathcal{C}_2(\psi)| = \binom{n}{r}$. Suppose that $\psi$ has $a$ fixed points $x_1,\dots,x_a$ and $b$ two-cycles $y_1z_1,\dots,y_bz_b$ so that $a+2b = n$. Observe that $e \in \binom{[n]}{r}$ satisfies $\psi(e) = e$ if and only if for every $1 \leq i \leq b$ it holds that $y_i \in e$ if and only if $z_i \in e$. The number of such $e$ is 
    \begin{equation}\label{eq:fixed points of tilde-psi}
    |\mathcal{C}_1(\psi)| = \sum_{\substack{0 \leq i \leq a \\ i \equiv r \; \text{mod } 2}} 
    \binom{a}{i}\binom{b}{(r-i)/2}
    \end{equation}
    If $a \leq 1$ (i.e., $a=0$ if $n$ is even and $a=1$ if $n$ is odd), then \eqref{eq:fixed points of tilde-psi} equals 
    $\binom{n}{r} - 2f_r(n)$. 
    Indeed, if $r$ is odd and $n$ is even then the sum in \eqref{eq:fixed points of tilde-psi} is empty, because $a=0$. If $r$ is odd and $n$ is odd then the sum equals $\binom{b}{(r-1)/2} = \binom{\floor{n/2}}{\floor{r/2}}$. The cases where $r$ is even can be checked similarly. So we see that when $a \leq 1$, we have $|\mathcal{C}_2(\psi)| = \frac{1}{2}\binom{n}{r} - \frac{1}{2}|\mathcal{C}_1(\psi)| = f_r(n)$, as claimed. 

    Now suppose that $a \geq 2$, and let $\psi'$ be obtained from $\psi$ by making two fixed points of $\psi$ into a 2-cycle. Then $\mathcal{C}_1(\psi') \subseteq \mathcal{C}_1(\psi)$, so $|\mathcal{C}_1(\psi')| \leq |\mathcal{C}_1(\psi)|$. Therefore, $|\mathcal{C}_1(\psi)|$ is minimized when $\psi$ has at most one fixed point. Hence, such $\psi$ maximize $|\mathcal{C}_2(\psi)|$. 
\end{proof}

If $\psi$ is an involution, then $G \overset{\psi}{\rightarrow} H$ if and only if $H \overset{\psi}{\rightarrow} G$ (because $\psi^{-1} = \psi$), so this relation is symmetric. The following lemma gives a simple description of this relation.
\begin{lemma}\label{prop: involution relation}
Let $\psi \in S_n$ be an involution 
% Let $\mathcal{C}_1$ be the set of edges $e \in \binom{[n]}{r}$ with $\psi(e) = e$, and let $\mathcal{C}_2$ be the set of pairs $(e,f)$, $e,f \in \binom{[n]}{r}$, $e \neq f$, satisfying $\psi(e) = f$ and $\psi(f) = e$. 
and let $G_1,G_2$ be $r$-graphs. Then $G_1 \overset{\psi}{\rightarrow} G_2$ holds if and only if the following two conditions are satisfied:
\begin{itemize}
    \item $G_1,G_2$ agree on all edges in $\mathcal{C}_1(\psi)$.
    \item For every $(e,f) \in \mathcal{C}_2(\psi)$, $G_1,G_2$ contain the same number edges from $\{e,f\}$. 
\end{itemize}
In particular, the relation $G \overset{\psi}{\rightarrow} G'$ is an equivalence relation.
\end{lemma}
\begin{proof}
    This follows from \cref{lemma: properties of choosable pairs}. 
    For every choosable pair $(e,f)$ for $(G_1, \psi)$, we must have $\{e,f\} \in \mathcal{C}_2(\psi)$, because $\psi(e) = f$ and $e \neq f$ (see \cref{def: choosable pairs}). 
    Also, for $\{e,f\} \in \mathcal{C}_2(\psi)$, we have that $(e,f)$ or $(f,e)$ is a choosable pair for $(G_1,\psi)$ if and only if $G$ contains exactly one of the edges $e,f$. 
\end{proof}

\cref{prop: involution relation} implies that for an involution $\psi$, every {\em $\psi$-component} of $r$-graphs, i.e. a set of $r$-graphs $X$ connected in the relation $G \overset{\psi}{\rightarrow} G'$, is a clique, namely, $G \overset{\psi}{\rightarrow} G'$ for all $G,G' \in X$. We refer to such a set $X$ as a {\em $\psi$-clique}. When $\psi$ is not specified, we will call $X$ an {\em involution clique}. 

Using the above, we now provide the lower bound in \cref{thm:tight}.
The construction is an involution clique. We also show that every involution cliques has at most the extremal size $2^{f_r(n)}$.

\begin{proposition}\label{prop:extremal construction}
    Let $2 \leq r \leq n$.
    \begin{enumerate}
        \item There exists a difference-isomorphic family of $r$-graphs on $[n]$ of size $2^{f_r(n)}$.
        \item For every involution $\psi$, every $\psi$-clique has size at most $2^{f_r(n)}$.
    \end{enumerate}

\end{proposition}
\begin{proof}
For Item 1, let $\psi \in S_n$ be an involution with at most 1 fixed points. 
Let $\{e_1,f_1\},\dots,\{e_m,f_m\}$ be the 2-cycles of $\psi$. By \cref{claim: fixed r-sets}, we have $m = f_r(n)$.
Let $\cal G$ be the set of all graphs that contain exactly one of the edges $e_i,f_i$ for every $1 \leq i \leq m$, and do not contain any other edges (the remaining edges are the edges in $\mathcal{C}_1(\psi)$). 
Then $|\mathcal{G}| = 2^m = 2^{f_r(n)}$. 
Also, \cref{prop: involution relation} implies $G_1 \overset{\psi}{\rightarrow} G_2$ for every $G_1,G_2 \in \mathcal{G}$, meaning that $\mathcal{G}$ is difference-isomorphic. 

For Item 2, let $\mathcal{G}$ be a $\psi$-clique. Then $|\mathcal{G}| \leq 2^{|\mathcal{C}_2(\psi)|}$ by \cref{prop: involution relation} (and equality holds only if for every $(e,f) \in \mathcal{C}_2(\psi)$, all the graphs in $\mathcal{G}$ contain exactly one of the edges $e,f$). Indeed, the graphs in $\mathcal{G}$ have at most $2$ choices for each pair $(e,f) \in \mathcal{C}_2$, and no choice for the remaining edges. By \cref{claim: fixed r-sets}, $|\mathcal{C}_2(\psi)| \leq f_r(n)$, so indeed $|\mathcal{G}| \leq 2^{f_r(n)}$.
\end{proof}

The construction given by \cref{prop:extremal construction} is an involution clique. As mentioned in the introduction, we can prove the following stability version of \cref{thm:tight}, stating that any difference-isomorphic family which is not an involution clique has smaller size. 

\begin{theorem}\label{theorem: stablility}
    For every $r \geq 2$, there is $n_0 = n_0(r)$ such that for every $n \geq n_0$ and every difference-isomophic family $\cal G$ of $r$-graphs on $[n]$, if $\cal G$ is not an involution clique then 
    $$
    \abs{\mc{G}} \leq
    \begin{cases}
    \big(1-n^{-O(\sqrt{n})}\big) \cdot 2^{f_2(n)} & r = 2, \\
    e^{-\Omega(n^{r-2})} \cdot 2^{f_r(n)} & r \geq 3.
    \end{cases}
    $$
\end{theorem}
\noindent
Observe that \cref{thm:tight} follows by combining \cref{theorem: stablility} and Item 2 in \cref{prop:extremal construction}.
% We remark that $e^{-\Omega(n^{r-2})}$ is necessary as we shall give an example in \cref{prop: tightness of stability}.

% \cref{theorem: stablility} is essentially tight for $r \geq 3$, namely, there exists a difference-isomorphic $r$-graph family on $[n]$ which is not an involution clique, but has size $e^{-O(n^{r-2})}2^{f_r(n)}$. For $r=2$, there exists such a graph-family of size $\frac{1}{2}2^{f_2(n)}$. We do not know if such a family can have size $(1-o(1))2^{f_2(n)}$. These constructions appear in the appendix.
In the appendix, we will construct a difference-isomorphic $r$-graph family on $[n]$ which is not an involution clique, but has size $e^{-O(n^{r-2})}2^{f_r(n)}$. 
This shows that when $r\ge 3$, \cref{theorem: stablility} is tight up to the hidden constant on the exponent. 
We do not know if it is also tight for $r=2$, i.e. if there is a difference-isomorphic graph-family on $[n]$ which is not an involution clique and has size $(1-o(1))2^{f_2(n)}$.

\subsection{An approximate upper bound}

Here we give an easy proof of an asymptotic version of \Cref{thm:tight}, namely, we show that a difference-isomorphic family of $r$-graphs on $[n]$ has size at most $2^{\left(\frac{1}{2}+o(1)\right)\binom{n}{r}}$. 
For an $r$-graph $G$, let $G^c$ denote the complement of $G$. 
We start with the following observation.
\begin{lemma}\label{claim:complement}
    If $\mathcal{G}$ is difference-isomorphic, then so is $\{G^c : G \in \mathcal{G}\}$.
\end{lemma}
\begin{proof}
    For every two $r$-graphs $G_1,G_2 \in \mathcal{G}$, it holds that $G_1^c \setminus G_2^c = G_2 \setminus G_1$ and $G_2^c\setminus G_1^c = G_1 \setminus G_2$. Hence, $G_1^c \setminus G_2^c$ is isomorphic to $G_2^c\setminus G_1^c$. 
\end{proof}
\noindent
In addition, if $\mathcal{G}$ is difference-isomorphic, then all $r$-graphs in $\mathcal{G}$ have the same amount of edges.
\begin{lemma}\label{claim:approximate bound}
    Let $\mathcal{G}$ be a difference-isomorphic family where all $r$-graphs have $m$ edges. Then $|\mathcal{G}| \leq \nolinebreak n! \cdot \nolinebreak 2^m$. 
\end{lemma}
\begin{proof}
    Fix any $r$-graph $G \in \mathcal{G}$. For each subgraph $F$ of $G$, let $\mathcal{G}_F$ be the set of all $G' \in \mathcal{G}$ with $G \setminus G' = F$. Then $G' \setminus G$ is isomorphic to $F$. It follows that there are at most $n!$ choices for $G' \setminus G$ (because each isomorphism class has size at most $n!$). As $G' \setminus G$ determines $G' = (G' \setminus G) \cup (G \setminus F)$, we get that
    $|\mathcal{G}_F| \leq n!$. 
    The lemma follows by summing over all $2^{e(G)} = 2^m$ choices for $F$.
\end{proof}

Now let $\mathcal{G}$ be a difference-isomorphic family of $r$-graphs on $[n]$. By \Cref{claim:complement}, we may assume that all $r$-graphs in $\mathcal{G}$ have size at most $\floor{\frac{1}{2}\binom{n}{r}}$. Then \Cref{claim:approximate bound} implies that 
$|\mathcal{G}| \leq 
n! \cdot 2^{\frac{1}{2}\binom{n}{r}} = 
2^{\left(\frac{1}{2}+o(1)\right)\binom{n}{r}}$, as claimed.

% \vspace{0.5cm}

% \noindent
% Later on we will need the Erd\H{o}s-Ko-Rado theorem.
% \begin{theorem}\label{theorem: erdos-ko-rado}
%     Let $n \ge 2r$.
%     If ${\mathcal {A}}$ is a family of distinct $r$-element subsets of $[n]$ such that each two subsets intersect, then $\abs{\cal A} \le \binom{n-1}{r-1}$.
% \end{theorem}

\section{Proof of \cref{thm:tight}}\label{sec:proof}

In this section we prove \cref{theorem: stablility}, which also implies \cref{thm:tight}.
Throughout the section, we will assume that $n$ is large enough as a function of $r$, wherever needed. 

\subsection{Proof overview}
For a set of $r$-graphs $X$ and $\psi \in S_n$, let $e_{\psi}(X)$ denote the number of (ordered) pairs $(G_1,G_2) \in X\times X$ such that $G_1 \overset{\psi}{\rightarrow} G_2$.
Note that the pairs $(G_1,G_1)$ are always counted because $G_1 \overset{\psi}{\rightarrow} G_1$ holds trivially. 
Also, if $X$ is a difference-isomorphic family of graphs, then $\sum_{\psi \in S_n} e_{\psi}(X) \geq |X|^2$, because for every pair $G_1,G_2 \in X$ there is $\psi \in S_n$ such that $G_1 \overset{\psi}{\rightarrow} G_2$.

Let $\cal G$ be a difference-isomorphic family of $r$-graphs.
For every $G \in \mc{G}$ and every $\varphi,\psi\in S_n$, we will consider $e_\psi(N_\varphi(G))$, i.e. the number of pairs $G_1,G_2 \in N_\varphi(G)$ such that $G_1 \overset{\psi}{\rightarrow} G_2$.
One of the main ingredients in the proof of \cref{theorem: stablility} is upper bounds for $e_\psi(N_\varphi(G))$, and this is the main focus of \cref{subsec:e_{psi}}.
Roughly speaking, we will show that $e_\psi(N_\varphi(G))$ is much smaller than $2^{2f_r(n)}$, which is the number of pairs in the extremal example, unless $\psi$ is an involution and $\varphi$ is very close to $\psi$. More precisely, we will show that unless $\varphi,\psi$ have this specific structure, it holds that
$e_\psi(N_\varphi(G)) \leq 2^{2f_r(n)} \cdot e^{-\Omega(n^{r-1})}$; see \cref{lemma: general counts for exceptional r-graph}. 
We note that due to this bound, the proof for the case $r \geq 3$ is simpler than for the case $r=2$. This is because for $r \geq 3$ we have $e^{\Omega(n^{r-1})} \gg n!$, allowing us to sum the above bound over all permutations $\psi \in S_n$. 
Thus, let us assume for now that $r \geq 3$.  
We use our bounds on $e_\psi(N_\varphi(G))$ to bound $|N_{\varphi}(G) \cap \mathcal{G}|$, using the fact that $|N_\varphi(G)\cap \mc{G}|^2 \leq \sum_{\psi \in S_n} e_{\psi}(N_\varphi(G) \cap \mathcal{G})$ (using that $\mathcal{G}$ is difference-isomorphic).
Due to the above, for a fixed $\varphi$, we have a very strong bound on $e_\psi(N_\varphi(G))$ for almost all permutations $\psi$ (i.e., for all $\psi$ which are not very close to $\varphi$). For the remaining $\psi$, we can show that $e_\psi(N_\varphi(G)) \leq 
2^{f_r(n)} \cdot e^{-\Omega(n)}$ unless $\varphi = \psi$ is an involution (see \cref{lemma: counts for exceptions r-graph}). 
This bound is weaker, but it suffices because the number of the remaining $\psi$s is only $e^{o(n)}$. 
Thus, combining the above, we conclude that unless $\varphi$ is an involution, $|N_{\varphi}(G) \cap \mathcal{G}|$ is much smaller than the extremal size $2^{f_r(n)}$, i.e. $|N_{\varphi}(G) \cap \mathcal{G}| \leq 2^{f_r(n)} \cdot e^{-\Omega(n)}$.
We note that if $r=2$ then this is no longer true; $N_\varphi(G) \cap \mathcal{G}$ can be of size $\Omega(2^{f_2(n)})$ even if $\varphi$ is not an involution (see the construction in the proof of \cref{prop: tightness of stability}). Consequently, the proof for the case $r=2$ is considerably more involved than for $r \geq 3$. The proof for $r \geq 3$ is given in \cref{subsec:main proof r>=3} and the proof for $r=2$ is given in \cref{subsec:main proof r=2}. 

We now continue with the overview of the proof for the case $r \geq 3$. A key step in the proof is to show that if $|\mathcal{G}|$ is close to the extremal size $2^{f_r(n)}$, then $\mathcal{G}$ must contain a large involution clique (this is done in \cref{theorem without large involution clique}). To this end, we use the aforementioned fact that $|N_{\varphi}(G) \cap \mathcal{G}|$ is small whenever $\varphi$ is not an involution. Note also that for an involution $\varphi$, the set $N_{\varphi}(G) \cap \mathcal{G}$ is a $\varphi$-clique (see \cref{prop: involution relation}). Thus, if $\mathcal{G}$ does not contain a large involution clique, then $|N_{\varphi}(G) \cap \mathcal{G}|$ is small for every permutation $\varphi \in S_n$. 
If this is the case, then we can upper-bound $|\mathcal{G}|$, as follows. Fix an arbitrary $G \in \mathcal{G}$, and fix a permutation $\varphi_1 \in S_n$ for which 
$|N_{\varphi_1}(G) \cap \mathcal{G}| \geq |\mathcal{G}|/n!$ (such a permutation must exist by averaging). Letting $\mathcal{G}_1 = N_{\varphi}(G) \cap \mathcal{G}$ and $\mathcal{G}_2 = \mathcal{G} \setminus \mathcal{G}_1$, we consider, for each permutation $\varphi_2$, the number of pairs $(G_1,G_2) \in \mathcal{G}_1 \times \mathcal{G}_2$ with $G_1 \overset{\varphi_2}{\rightarrow} G_2$. It turns out that the number of such pairs is very small unless $\varphi_2$ is close to $\varphi_1$ (this is shown in \cref{lemma: codegree is usually small}). Hence, almost all such pairs $(G_1,G_2)$ correspond to only few (i.e. only $e^{o(n)}$) permutations $\varphi_2 \in S_n$. Thus, by averaging, we can find a permutation $\varphi_2$ and a graph $G_1 \in \mathcal{G}_1$ for which $|N_{\varphi_2}(G_1) \cap \mathcal{G}_2| \geq |\mathcal{G}_2| \cdot e^{-o(n)}$. On the other hand, we know that $|N_{\varphi_2}(G_1) \cap \mathcal{G}| \leq 2^{f_r(n)} \cdot e^{-\Omega(n)}$, which allows us to deduce that $\mathcal{G}_2$ is small. As $\mathcal{G}_1$ is also small (for the same reason, as $\mathcal{G}_1 = N_{\varphi_1}(G) \cap \mathcal{G}$), we conclude that $|\mathcal{G}| = |\mathcal{G}_1| + |\mathcal{G}_2| \ll 2^{f_r(n)}$. 
% we upper-bound, for each permutation $\varphi_2$ the number of pairs $(G_1,G_2) \in \mathcal{G}_1 \times \mathcal{G}_2$ 

Summarizing the above discussion, we can show that if $|\mathcal{G}|$ is close to the extremal size $2^{f_r(n)}$, then $\mathcal{G}$ must contain a large involution clique. We will use this to conclude the proof of \cref{theorem: stablility} (for $r \geq 3$), as follows. 
Let $\mathcal{G}'$ be a largest involution clique in $\mathcal{G}$ and let $\psi$ be the corresponding involution. 
If $\mathcal{G}' = \mathcal{G}$ then $\mathcal{G}$ is an involution clique, so the statement of \cref{theorem: stablility} is vacuous. 
Otherwise, fix an arbitrary $G \in \mathcal{G} \setminus \mathcal{G}'$. Crucially, we have $N_{\psi}(G) \cap \mathcal{G}' = \emptyset$, because otherwise $\mathcal{G}' \cup \{G\}$ would be a $\psi$-clique by \cref{prop: involution relation}, contradicting the maximality of $\mathcal{G}'$.
Next, for every $\varphi \neq \psi$, we can bound $|N_{\varphi}(G) \cap \mathcal{G}'|$ by using the fact that $|N_{\varphi}(G) \cap \mathcal{G}'|^2 \leq e_{\psi}(N_{\varphi}(G))$, as $\mathcal{G}'$ is a $\psi$-clique. 
We then use this to upper bound $|\mathcal{G}'| \leq \sum_{\varphi \in S_n} |N_{\varphi}(G) \cap \mathcal{G}'|$. 
Here, too, we separate the permutations $\varphi$ into those which are not very close to $\psi$, for which we have a very strong bound (this covers almost all permutations), and the remaining few permutations, for which we have a weaker but sufficient bound. This allows us to show that $|\mathcal{G}'|$ is indeed small, giving a contradiction and completing the proof.

\subsection{Edge counts inside $N_\varphi(G)$}\label{subsec:e_{psi}}

We start with the following important definition.
% for bounding $e_\psi(N_\varphi(G))$.
\begin{definition}[Good choosable pairs]\label{def: good choosable pair}
Let $\varphi,\psi \in S_n$ and let $G$ be an $r$-graph.
Let $(e,f)$ be a choosable pair for $(G,\varphi)$. We say that $(e,f)$ is a {\em good choosable pair} for $(G,\varphi,\psi)$ if $\psi(e)=f$ and $\psi(f)=e$.
\end{definition}
% We first consider all the choosable pairs $(e_i,f_i)_{i=1}^m$ for $(G,\varphi)$.
% Clearly, $m \le \frac{1}{2}\binom{n}{2}$.
% According to \cref{lemma: properties of choosable pairs}, for every $i \in [m]$ we have $e_i \in G_1, f_i \notin G_1$ or $f_i \in G_1, e_i \notin G_1$, and choosing one of the two options for each $i \in [m]$ determines $G_1$. 
% The same holds for $G_2$. 
% For each $i\in\{1,\dots,m\}$, we say that $(e_i,f_i)$ is a {\em good choosable pair} if $\tilde{\psi}(e_i)=f_i$ and $\tilde{\psi}(f_i)=e_i$.
% Suppose there are $m_g$ good choosable pairs.
% We may assume they are $(e_1,f_1),\dots,(e_{m_g},f_{m_g})$.

% We have the following bound for $e_{\psi}(N_\varphi(G))$ in terms of $m$ and $m_g$ and certain criteria for $m_g$ to be ``large''.
% We will delay the respective proofs at the moment.

The following lemma gives an upper bound on $e_{\psi}(N_\varphi(G))$ in terms of the number of choosable pairs of $(G,\varphi)$ and the number of good choosable pairs of $(G,\varphi,\psi)$.

\begin{lemma}\label{claim: count in terms of good pairs}
    Let $n \ge r \ge 2$.
    Let $\varphi,\psi \in S_n$ and let $G$ be an $r$-graph. Let $m$ be the number of choosable pairs of $(G,\varphi)$ and let $m_g$ be the number of good choosable pairs of $(G,\varphi,\psi)$. Then
    \begin{equation}\label{eq: m m_g bound}
        e_{\psi}(N_\varphi(G)) \le 4^{m_g} \cdot 3.9^{m-m_g}.
    \end{equation}
    In particular, if $m \leq \frac{1}{2}\binom{n}{r}-x$ and $m_g\leq\frac{1}{2}\binom{n}{r}-y$ for some $y\ge x\ge 0$, then 
    \begin{equation}\label{eq: practical m m_g bound}
        e_{\psi}(N_\varphi(G)) \le 2^{\binom{n}{r}-2x}e^{-(y-x)/40}.
    \end{equation}
\end{lemma} 
It is worth comparing the bound in \cref{eq: m m_g bound} with the trivial upper bound $e_{\psi}(N_\varphi(G)) \leq |N_{\varphi}(G)|^2 \leq 4^m$, where the last inequality follows from \cref{lemma: properties of choosable pairs}.
\begin{proof}[Proof of \cref{claim: count in terms of good pairs}]
    The second part of the lemma follows from 
    \cref{eq: m m_g bound} as follows:
    \begin{equation} \nonumber
        e_{\psi}(N_\varphi(G)) 
        \le \left(\frac{4}{3.9}\right)^{m_g} 3.9^m
        \le \left(\frac{4}{3.9}\right)^{\frac{1}{2}\binom{n}{r}-y} 3.9^{\frac{1}{2}\binom{n}{r}-x}
        = 4^{\frac{1}{2}\binom{n}{r}-x} \left(1-\frac{1}{40}\right)^{y-x}
        \le 2^{\binom{n}{r}-2x} e^{-(y-x)/40},
    \end{equation}
    where in the last inequality, we used $y \ge x$.

    We now prove the first part of the lemma. Let $(e_i,f_i)$, $i = 1,\dots,m$ be the choosable pairs for $(G,\varphi)$. Suppose without loss of generality that $(e_i,f_i)$ is a good choosable pair for $(G,\varphi,\psi)$ if and only if $i \leq m_g$. 
    % For each $i \in \{m_g+1,\dots,m\}$, let $e_i' = \tilde{\psi}(e_i)$ and $f_i'=\tilde{\psi}(f_i)$.
    % There are three possibilities making $(e_i,f_i)$ not a good choosable pair: $e_i'$ or $f_i'$ is not contained in a choosable pair for $(G,\varphi)$; $e_i'=e_i$ or $f_i'=f_i$; $e_i'$ or $f_i'$ is contained in some choosable pair other than $(e_i,f_i)$.
    For each $i \in \{m_g+1,\dots,m\}$, $(e_i,f_i)$ is a {\em bad} (i.e., not good) choosable pair, so $\psi(e_i) \neq f_i$ or $\psi(f_i) \neq e_i$. 
    If $\psi(e_i) \neq f_i$, then there are three options for $\psi(e_i)$: either $\psi(e_i) = e_i$; or $\psi(e_i)$ does not belong to any of the choosable pairs, i.e. $\psi(e_i) \notin \{e_1,f_1,\dots,e_m,f_m\}$; or $\psi(e_i) \in \{e_j,f_j\}$ for some $j \neq i$. 
    If $\psi(f_i) \neq e_i$ then the analogous statements hold for $f_i$. 
    We will refer to these three cases as types (i), (ii) and (iii), respectively.
    
    Recall that our goal is to count pairs $(G_1,G_2)\in N_\varphi(G)\times N_\varphi(G)$ with $G_1 \overset{\psi}{\rightarrow} G_2$.
    We shall upper-bound the number of choices for $(G_1,G_2)$ on each of the choosable pairs $(e_i,f_i)$. By \cref{lemma: properties of choosable pairs}, each of $G_1,G_2$ contains exactly one of the edges $e_i,f_i$. Hence, trivially, $(G_1,G_2)$ has at most 4 options on $(e_i,f_i)$ (i.e., two options for $G_1$ and two options for $G_2$). 
    To prove the lemma, we will show that on $\Omega(m - m_g)$ of the bad choosable pairs, there are less than $4$ options.
    % , the pair $(G_1,G_2)$
    % Recall that both $G_1$ and $G_2$ contain exactly one edge between $e_i$ and $f_i$ for all $i\in\{1,\dots,m\}$ (see ).

    First, suppose there are at least $(m-m_g)/6$ bad choosable pairs of type (i) or (ii).
    Fix such a pair $(e_i,f_i)$.
    % By definition, there is $g_i \in \{e_i,f_i\}$ such that $\psi(g_i) = g_i$ or $\psi(g_i) \notin \{e_1,f_1,\dots,e_m,f_m\}$. 
    We assume that $\psi(e_i) = e_i$ or $\psi(e_i) \notin \{e_1,f_1,\dots,e_m,f_m\}$. The other case, namely that $\psi(f_i) = f_i$ or $\psi(f_i) \notin \{e_1,f_1,\dots,e_m,f_m\}$, is similar. 
    % We may assume that $e_i'$ is not contained in a choosable pair or $e_i'=e_i$ (the other cases are symmetric). 
    We claim that it is impossible to have $e_i \in G_1, f_i \notin G_1$ and $e_i \notin G_2, f_i \in G_2$. 
    This would show that one of the 4 possibilities for $(G_1,G_2)$ on $(e_i,f_i)$ is impossible. 
    Indeed, suppose by contradiction that the above holds. 
    Then $e_i \in G_1 \setminus G_2$. Hence, $\psi(e_i) \in G_2 \setminus G_1$, as $G_1 \overset{\psi}{\rightarrow} G_2$. This rules out the case $\psi(e_i) = e_i$, so $\psi(e_i) \notin \{e_1,f_1,\dots,e_m,f_m\}$. But then, by \cref{lemma: properties of choosable pairs}, we have that $\psi(e_i) \in G_1$ if and only if $\psi(e_i) \in G_2$, so this again contradicts $\psi(e_i) \in G_2 \setminus G_1$. Summarizing, we see that on at least $(m-m_g)/6$ of the bad choosable pairs, $(G_1,G_2)$ has at most $3$ options. As explained before, on every other choosable pair $(G_1,G_2)$ has at most $4$ options. Hence, \cref{eq: m m_g bound} follows as 
    % If $e_i \in G_1$ and $e_i \notin G_2$ (namely, $e_i \in G_1 \setminus G_2$), then $e_i'=\tilde{\psi}(e_i) \in G_2 \setminus G_1$.
    % However, in the case $e_i'$ is not contained in a choosable pair, \cref{lemma: properties of choosable pairs} indicates $e_i'\in G_1 \Leftrightarrow e_i'\in G\Leftrightarrow e_i' \in G_2$ while in the case $e_i'=e_i$, we have $e_i'=e_i\in G_1\setminus G_2$.
    % In both cases, we get a contradiction, so there are at most three choices for $(G_1,G_2)$ on choosable pairs of the first two cases.
    % For all other choosable pairs, $(G_1,G_2)$ has at most 4 choices trivially. 
    % The desired inequality then follows:
    % $$
    % e_\psi(N_\varphi(G))
    % \le 4^{m-m'}3^{m'}
    % =4^m\!\left(\frac{3}{4}\right)^{m'}
    % \le 4^m\!\left(\frac{3}{4}\right)^{\frac{m-m_g}{6}} 
    % = 4^{m_g} \cdot 4^{m-m_g}\!\left(\frac{3}{4}\right)^{\frac{m-m_g}{6}} 
    % \le 4^{m_g} \cdot 3.9^{m-m_g}.$$
    $$
    e_\psi(N_\varphi(G))
    \le 4^{m - \frac{m-m_g}{6}} \cdot 3^{\frac{m-m_g}{6}}
    % =4^m\!\left(\frac{3}{4}\right)^{\frac{m-m_g}{6}} 
    = 4^{m_g} \cdot 4^{m-m_g}\!\left(\frac{3}{4}\right)^{\frac{m-m_g}{6}} 
    \le 4^{m_g} \cdot 3.9^{m-m_g}.$$
    % as required. 

    Now suppose that there are at least $5(m-m_g)/6$ bad choosable pairs of type (iii). 
    Let $S$ be the set of indices $i \in \{m_g+1,\dots,m\}$ for which $(e_i,f_i)$ is of type (iii).
    By definition, for each $i \in S$, there exists some $j=j(i)\neq i$ such that $\psi(e_i)\in\{e_j,f_j\}$ or $\psi(f_i)\in\{e_j,f_j\}$.
    Construct an auxiliary digraph $H$ with vertex set $[m]$ where we put a directed edge $(i,j(i))$ for every $i \in S$.
    It is easy to see that $H$ has $\abs{S}$ edges (there can be two edges with the opposite direction), out-degree at most 1, and in-degree at most 2 (corresponds to $\psi^{-1}(e_i),\psi^{-1}(f_i)$).
    Notice that every edge, say $(i,j)$, is incident to at most four other edges (two in-neighbors of $i$, one in-neighbor of $j$ and one out-neighbor of $j$).
    $H$ contains a matching $M$ of size $\abs{M}\ge |S|/5\geq (m-m_g)/6$ (we can iteratively take an edge and delete all incident edges).

    Consider any $ij \in M$. Without loss of generality, let us assume that $\psi(e_i) \in \{e_j,f_j\}$ or $\psi(f_i) \in \{e_j,f_j\}$; the other case (where the roles of $i,j$ are switched) is symmetric. 
    Recall that a priori, $(G_1,G_2)$ has at most $4 \cdot 4 = 16$ options on $\{e_i,f_i,e_j,f_j\}$. 
    We will show that in fact there are at most $13$ options. 
    Suppose without loss of generality that $\psi(e_i) = f_j$; the remaining 3 cases (i.e. $\psi(e_i) = e_j$, $\psi(f_i) = e_j$, and $\psi(f_i) = f_j$) are similar. 
    Observe that out of the four choices for $(G_1,G_2)$ on $(e_i,f_i)$, if $e_i\in G_1,e_i\notin G_2$, namely $e_i\in G_1\setminus G_2$, then then $f_j = \psi(e_i) \in G_2 \setminus G_1$, which forces the choice $e_j \in G_1, f_j \notin G_1, e_j \notin G_2, f_j \in G_2$ of $(G_1,G_2)$ on $(e_j,f_j)$. Hence, 3 of the 16 choices for $(G_1,G_2)$ on $\{e_i,f_i,e_j,f_j\}$ are impossible, leaving 13 choices. 
    % Observe that in one of the four choices for $(G_1,G_2)$ on $(e_i,f_i)$, it holds that $e_i \in G_1, e_i \notin G_2$, namely, $e_i \in G_1 \setminus G_2$. If this happens, then $f_j = \psi(e_i) \in G_2 \setminus G_1$, which forces the choice $e_j \in G_1, f_j \notin G_1, e_j \notin G_2, f_j \in G_2$ of $(G_1,G_2)$ on $(e_j,f_j)$. Hence, 3 of the 16 choices for $(G_1,G_2)$ on $\{e_i,f_i,e_j,f_j\}$ are impossible, leaving 13 choices. 
    
    Recalling that $(G_1,G_2)$ has at most $4$ choices for every choosable pair $(e_i,f_i)$, we get the bound
    % Say that $((e_{i_k},f_{i_k}),(e_{j_k},f_{j_k}))_{k=1}^\ell$ is a matching in $H$ with $\ell \ge \abs{S}/5$.
    % WLOG, we may assume $\tilde{\psi}(e_{i_k})\in \{e_{j_k},f_{j_k}\}$ or $\tilde{\psi}(f_{i_k}) \in \{e_{j_k},f_{j_k}\}$ for all $k$.
    % For each $k \in \{1,\dots,\ell\}$, we claim that $(G_1,G_2)$ has at most 13 choices on the two choosable pairs $(e_{i_k},f_{i_k})$ and $(e_{j_k},f_{j_k})$.
    % Say $\tilde{\psi}(e_{i_k}) = f_{j_k}$ (the other cases are similar).
    % We know that if $e_{i_k} \in G_1$ and $e_{i_k} \notin G_2$ ($e_{i_k}\in G_1\setminus G_2$), then $f_{j_k}=\tilde{\psi}(e_{i_k}) \in G_2\setminus G_1$, i.e. $f_{j_k}\in G_2$ (so $e_{j_k}\notin G_2$) and $f_{j_k}\notin G_1$ (so $e_{j_k} \in G_1$).
    % Then, there is only one choice on these two pairs. 
    % But if $e_{i_k} \notin G_1$ or $e_{i_k} \in G_2$, then $(G_1,G_2)$ trivially has at most $3\times 4=12$ choices.
    % In total, $(G_1,G_2)$ has at most 13 choices on choosable pairs $(e_{i_k},f_{i_k})$ and $(e_{j_k},f_{j_k})$.
    % Also, for all choosable pairs not contained in $((e_{i_k},f_{i_k}),(e_{j_k},f_{j_k}))_{k=1}^\ell$, $(G_1,G_2)$ trivially has at most 4 choices.
    % The desired inequality then follows:
    % $$
    %     e_{\psi}(N_\varphi(G)) 
    %     \le 13^{|M|}\cdot 4^{m-2|M|}
    %     \le 4^m\!\left(\frac{13}{16}\right)^{\frac{m-m_g}{6}}
    %     \le 4^{m_g}\cdot 4^{m-m_g}\!\left(\frac{13}{16}\right)^{\frac{m-m_g}{6}}
    %     \le 4^{m_g}\cdot 3.9^{m-m_g}.
    % $$
    \begin{align*}
        e_{\psi}(N_\varphi(G)) 
        \le 13^{|M|}\cdot 4^{m-2|M|} 
        = 4^{m_g}\cdot 4^{m-m_g} \left(\frac{13}{16}\right)^{|M|}
        \le 4^{m_g}\cdot 4^{m-m_g}\left(\frac{13}{16}\right)^{\frac{m-m_g}{6}}
        % = 4^{m_g}\cdot 4^{m-m_g}\!\left(\frac{13}{16}\right)^{\frac{m-m_g}{6}}
        \le 4^{m_g}\cdot 3.9^{m-m_g}.
    \end{align*}
    This completes the proof of \cref{claim: count in terms of good pairs}.
\end{proof}
In the next lemma we prove the intuitive fact that if a permutation $\psi$ is an ``almost-involution" on the edges, in the sense that almost all edges belong to a $2$-cycle of $\tilde{\psi}$, then $\psi$ is also an almost-involution on the vertices, in the sense that almost all vertices belong to a $2$-cycle of $\psi$. For the proof, let us recall the Erd\H{o}s-Ko-Rado theorem \cite{EKR}:
\begin{theorem}[Erd\H{o}s-Ko-Rado, \cite{EKR}]\label{theorem: erdos-ko-rado}
    Let $n \ge 2r$.
    If ${\mathcal {A}}$ is a family of distinct $r$-element subsets of $[n]$ such that each two subsets intersect, then $\abs{\cal A} \le \binom{n-1}{r-1}$.
\end{theorem}

\begin{lemma} \label{lemma: 2-cycles in psi r-graph}
    Let $r \ge 2$ and suppose that $n$ is sufficiently large in terms of $r$. 
    Let $\delta\in\mathbb{R}$ satisfy $1 \leq \delta=o(n^r)$.
    For every $\psi \in S_n$,
    if $\tilde{\psi} \in S_{\binom{[n]}{r}}$ has at least $\left(\frac{1}{2}\binom{n}{r}-\delta\right)$ 2-cycles, then $\psi$ has at least $\left(\frac{n}{2}-r\delta^{1/r}\right)$ 2-cycles.
\end{lemma}
\begin{proof}
    Let $\mc{E}_2$ denote the set of edges $e \in \binom{[n]}{r}$ which belong to some $2$-cycle of $\tilde{\psi}$. 
    Fix any vertex $v \in [n]$.
    First, we claim that if $v$ is incident to more than $\binom{n-2}{r-2}$ distinct edges in $\mc{E}_2$, then $\psi^2(v) = v$. 
    Indeed, if $\mc{E}_2(v) :=\{e\setminus\{v\}: e \in \mc{E}_2, v\in e\}$ has size larger than $\binom{n-2}{r-2}$, then by the Erd\H{o}s-Ko-Rado theorem (\cref{theorem: erdos-ko-rado}) with $[n] \setminus \{v\}$ in place of $[n]$ and $r-1$ in place of $r$, there exist distinct $e_1,e_2 \in \mc{E}_2$ such that $v \in e_1,v\in e_2$ and $(e_1\setminus\{v\})\cap(e_2\setminus\{v\})=\emptyset$, i.e. $e_1 \cap e_2 = \{v\}$.
    Putting $f_1 = \psi(e_1)$ and $f_2 = \psi(e_2)$, we know that $\psi(f_1)=e_1$ and $\psi(f_2)=e_2$, because $(e_1,f_1),(e_2,f_2)$ are $2$-cycles of $\tilde{\psi}$ (as $e_1,e_2 \in \mc{E}_2$).
    Since $\psi$ is a bijection, 
    it holds that $\psi(e_1\cap e_2)=f_1 \cap f_2$ and $\psi(f_1 \cap f_2)=e_1 \cap e_2$.
    Since $e \cap e' = \{v\}$, we get that $\psi^2(v) = v$, as claimed. This means that $v$ lies in a 1-cycle or a 2-cycle of $\psi$.

    Let $X$ be the set of vertices $v \in [n]$ which are incident to {\bf at most $\binom{n-2}{r-2}$} edge from $\mathcal{E}_2$. 
    We showed above that every vertex not in $X$ lies in a 1-cycle or a 2-cycle of $\psi$. 
    Let $x := |X|$.
    For each vertex $v \in X$, there are at least $\binom{n-1}{r-1}-\binom{n-2}{r-2}=\binom{n-2}{r-1}$ edges $e \in \binom{[n]}{r}$ with $v \in e$ and $e \notin \mathcal{E}_2$.
    % (here the inequality holds if $n$ is large enough in terms of $r$).
    It follows that $\binom{n}{r} - |\mathcal{E}_2| \geq \frac{1}{r} x \binom{n-2}{r-1}$,
    % edges incident to some vertex in $\abs{X}$ that are not in $\cal E$ 
    because each edge is counted at most $r$ times.
    By assumption, $\tilde{\psi}$ has at least $\left(\frac{1}{2}\binom{n}{r}-\delta\right)$ 2-cycles, so $|\mathcal{E}_2| \geq \binom{n}{r} - 2\delta$. 
    Thus, 
    $\frac{1}{r} x\binom{n-2}{r-1} \leq 2\delta$, giving
    % This gives $\binom{n}{r}-2\delta+\frac{x}{2r}\binom{n-1}{r-1} \le \binom{n}{r}$ and thus, 
    $x \le 2r\delta/\binom{n-2}{r-1}\le 0.5\delta^{1/r}$ using that $\delta=o(n^{r})$.

    Let $Y$ be the set of fixed points of $\psi$, and set $y:=\abs{Y}$.
    % As any edge with all incident vertices in $Y$ form a 1-cycle in $\tilde{\psi}$, it holds that $\binom{n}{r}-2\delta + \binom{y}{r} \le \binom{n}{r}$.
    If $e \in \binom{[n]}{r}$ is contained in $Y$ then $\psi(e) = e$, so $e$ cannot be in a $2$-cycle of $\tilde{\psi}$, i.e. $e \notin \mathcal{E}_2$. 
    It follows that $|\mathcal{E}_2| \leq \binom{n}{r} - \binom{y}{r}$, and hence $2\delta \geq \binom{y}{r}$. So either $y < r$, or $y \geq r$ and then $2\delta\geq\binom{y}{r}\geq(y/r)^r$.
    In any case, $y\le r(2\delta)^{1/r} \leq 1.5r\delta^{1/r}$, using $r \geq 2$.
    
    By definition of $X$ and $Y$, every vertex in $[n] \setminus (X \cup Y)$
    lies in a 2-cycle of $\psi$. There are  $n - x - y\ge n - 2r\delta^{1/r}$ such vertices. 
    Hence, the number of 2-cycles in $\psi$ is at least $\frac{n}{2}-r\delta^{1/r}$.
\end{proof}

Note that good choosable pairs are closely related to the 2-cycles of $\tilde{\psi}$ (and thus $\psi$).
We further prove that there will be few good choosable pairs unless $\varphi$ contains most 2-cycles of $\psi$.
\begin{lemma}\label{lemma: few half-two-cycles r-graph}
    Let $G$ be an $r$-graph and $\varphi,\psi \in S_n$. Let $t$ be the number of 2-cycles of $\psi$ which are not 2-cycles of $\varphi$. Then 
    $m_g \le \frac{1}{2}\binom{n}{r}-\binom{t}{r}$, where $m_g$ is the number of good choosable pairs of $(G,\varphi,\psi)$.
\end{lemma}
\begin{proof}
    Let $u_1v_1,u_2v_2,\dots,u_rv_r$ be any $r$ distinct 2-cycles of $\psi$ which are not 2-cycles of $\varphi$.
    An edge $e \in \binom{[n]}{r}$ is called a {\em transversal} if it contains exactly one of $u_i,v_i$ for every $i\in\{1,\dots,r\}$.
    It suffices to show that there are at least two transversal edges $e,f$ which are not contained in any good choosable pair for $(G,\varphi,\psi)$.
    Indeed, we can then sum over all $\binom{t}{r}$ choices for $u_1v_1,u_2v_2,\dots,u_rv_r$ to get at least
    $2\binom{t}{r}$ edges which are not contained in any good choosable pair. This would imply that $m_g \leq \frac{1}{2}\binom{n}{r}-\binom{t}{r}$.

    Suppose for contradiction that all transversal edges (of $u_1v_1,u_2v_2,\dots,u_rv_r$) but at most one are contained in good choosable pairs.
    Observe that if edge $e$ is transversal, then $\psi(e)$ is also transversal, and $e,\psi(e)$ form a 2-cycle of $\tilde{\psi}$.
    By the definition of good choosable pairs (see \cref{def: good choosable pair}), if $e$ lies in some good choosable pair then this pair must be 
    $\{e,\psi(e)\}$. 
    % if and only if $\psi(e)$ lies in some good choosable pair.
    %  In fact, if so, they must lie in the same good choosable pair.
    Thus, our assumption implies that every transversal edge $e$ belongs to a good choosable pair, namely the pair $\{e,\psi(e)\}$.
    
    We first claim that for every distinct $i,j\in \{1,\dots,r\}$, $\varphi(u_i)\neq v_j$ and $\varphi(v_i)\neq u_j$.
    Indeed, suppose that $\varphi(u_i)=v_j$ for some $i \neq j$ (the case $\varphi(v_i)=u_j$ is similar).
    Let $e=\{u_k:k\neq i,j\}\cup\{u_i,v_j\}$ and $f=\psi(e)=\{v_k:k\neq i,j\}\cup\{v_i,u_j\}$.
    As $e,f$ form a good choosable pair, we have $\varphi(e)=f$ or $\varphi(f)=e$ (see \cref{def: choosable pairs}).
    But $u_i \in e$, $v_j = \varphi(u_i) \notin f$, meaning that $\varphi(e) \neq f$. Similarly, $v_j \in e$ but $u_i = \varphi^{-1}(v_j) \notin f$, so $e \neq \varphi(f)$. In either case, we got a contradiction. 
    % Clearly, $e\cap f=\emptyset$, so $\varphi(e)\cap e=\emptyset$ or $\varphi^{-1}(e)\cap e=\emptyset$.
    % But $v_j=\varphi(u_i)$ lies in both $e,\varphi(e)$, and $u_i=\varphi^{-1}(v_j)$ lies in both $e,\varphi^{-1}(e)$, giving the contradiction in either case.
    
    Now, write $U=\{u_1,\dots,u_r\}$ and $V=\{v_1,\dots,v_r\}$.
    As the edges $u_1\dots u_r$ and $v_1\dots v_r$ form a good choosable pair, it holds that $\varphi(u_1\dots u_r)=v_1\dots v_r$ or $\varphi(v_1\dots v_r)=u_1\dots u_r$, i.e. $\varphi(U)=V$ or $\varphi(V)=U$.
    Without loss of generality, assume that $\varphi(U)=V$.
    For every $i \in [r]$, we showed above that $\varphi(u_i) \neq v_j$ for every $j \neq i$. As $\varphi(u_i) \in V$, we must have
    $\varphi(u_i)=v_i$.
    By assumption, $\varphi(v_i)\neq u_i$ (otherwise $u_iv_i$ would form a common 2-cycle in $\varphi$ and $\psi$, in contradiction to the choice of $u_iv_i$). 
    We also showed that $\varphi(v_i) \neq u_j$ for every $j \neq i$, so $\varphi(v_j) \notin U$.
    It follows that $\varphi(V)\cap U=\emptyset$.
    Now, consider the edges $e=\{u_1,v_2,\dots,v_r\}$ and $f=\psi(e)=\{v_1,u_2,\dots,u_r\}$.
    By assumption $\{e,f\}$ forms a good choosable pair, which implies that $\varphi(e)=f$ or $\varphi(f)=e$.
    In the former case, as $\varphi(u_1)=v_1$, we must have $\varphi(v_2,\dots,v_r)=u_2\dots u_r$; in the latter case, as $\varphi(u_2\dots u_r)=v_2\dots v_r$, we must have $\varphi(v_1)=u_1$.
    In either case, $\varphi(V) \cap U\neq\emptyset$, giving a contradiction.
    We, therefore, conclude at least two transversal edges do not lie in any good choosable pair for $(G,\varphi,\psi)$, as required. 
\end{proof}

We now prove the main lemmas of this section, namely \cref{lemma: counts of normal phi r-graph}, \cref{lemma: general counts for exceptional r-graph} and \cref{lemma: counts for exceptions r-graph}. 
These lemmas establish upper bounds on $e_\psi(N_\varphi(G))$ in various cases. Our goal is to show that $e_\psi(N_\varphi(G))$ is significantly smaller than $2^{2f_r(n)}$, which is the number pairs in the extremal example (see \cref{prop:extremal construction}).
% because $e_\psi( N_\varphi(G))$ counts pairs of graphs in our family $\mathcal{G}$ and we are aiming to show that $|\mathcal{G}| \leq 2^{f_r(n)}$.
In the following lemma we show that $e_\psi(N_\varphi(G)))$ is small unless $\varphi,\psi$ are close to each other and consist mostly of 2-cycles.

\begin{lemma}\label{lemma: counts of normal phi r-graph}
    Let $r \ge 2$ and $n$ be sufficiently large in terms of $r$.
    Let $G$ be a graph and $\varphi,\psi \in S_n$.
    For every $\delta \ge n^{r/2}$ with $\delta=o(n^{r})$, it holds that $e_{\psi}(N_{\varphi}(G)) \leq 2^{2f_r(n)} \cdot e^{-\delta}$ unless $\psi$ and $\varphi$ share at least $\left(\frac{n}{2}-18r{\delta}^{1/r}\right)$ 2-cycles.
\end{lemma}
\begin{proof}
    Suppose that $e_\psi(N_\varphi(G)) > 2^{2f_r(n)}\cdot e^{-\delta}$.
    Denote by $m$ the number of choosable pairs for $(G,\varphi)$ and by $m_g$ the number of good choosable pairs for $(G,\varphi,\psi)$.
    We claim that $m_g > \frac{1}{2}\binom{n}{r}-80\delta$.
    Suppose otherwise. Note that $f_r(n) \geq \frac{1}{2}(\binom{n}{r} - n^{r/2})$. 
    % \begin{equation} \nonumber
    %     f_r(n)\ge2^{\left(\binom{n}{r}-\binom{\floor{n/2}}{\floor{r/2}}\right)/2}\ge 2^{\left(\binom{n}{r}-n^{r/2}\right)/2}\quad
    %     \Longrightarrow \quad 
    %     f_r(n)^2 \cdot 2^{n^{r/2}} \ge 2^{\binom{n}{r}}.
    % \end{equation}
    Now, as $m_g \le \frac{1}{2}\binom{n}{r} - 80\delta$ (by assumption) and $m \le \frac{1}{2}\binom{n}{r}$ (trivially), we can apply \cref{eq: practical m m_g bound} with $x=0$ and $y = 80\delta$ to get
    % \begin{equation} \label{eq: edge count is small}
    %     \begin{aligned}
    %         e_\psi(N_\varphi(G))
    %         &\le \left(\frac{4}{3.9}\right)^{m_g} 3.9^m
    %         = \left(\frac{4}{3.9}\right)^{\frac{1}{2}\binom{n}{r} - 80\delta} 3.9^{\frac{1}{2}\binom{n}{r}}
    %         \leq 2^{\binom{n}{r}} \left(1-\frac{1}{40}\right)^{80\delta} \\
    %         &\leq f_2(n)^2 \cdot 2^{n^{r/2}} \cdot e^{-80\delta/40} 
    %         \leq 
    %         f_2(n)^2 \cdot 2^{n^{r/2}} \cdot e^{-2\delta} 
    %         \le f_2(n)^2 \cdot e^{-\delta},
    %     \end{aligned}
    % \end{equation}
    \begin{equation*} \label{eq: edge count is small}
        \begin{aligned}
            e_\psi(N_\varphi(G))
            \leq 
            2^{\binom{n}{r}} e^{-2\delta} \leq 
            2^{2f_r(n) + n^{r/2}} \cdot e^{-2\delta} \leq  
            2^{2f_r(n)} \cdot e^{-\delta},
        \end{aligned}
    \end{equation*}
    using that $\delta \geq n^{r/2}$. This contradicts our assumption.

    So indeed $m_g > \frac{1}{2}\binom{n}{r}-80\delta$.
    Recall that the two edges of every good choosable pair form a 2-cycle of $\tilde{\psi}$ (see \cref{def: good choosable pair}).
    This means $\tilde{\psi}$ has at least $\left(\frac{1}{2}\binom{n}{r}-80\delta\right)$ 2-cycles. 
    Therefore, by \cref{lemma: 2-cycles in psi r-graph} (with $80\delta$ in place of $\delta$), the number of 2-cycles of $\psi$ is at least $\frac{n}{2}-r(80\delta)^{1/r}\ge \frac{n}{2}-9r\delta^{1/r}$ (using that $r \geq 2$).
    Let $t$ be the number of 2-cycles of $\psi$ that are not 2-cycles of $\varphi$.
    By \cref{lemma: few half-two-cycles r-graph}, it holds that $\frac{1}{2}\binom{n}{r}-\binom{t}{r} \ge m_g > \frac{1}{2}\binom{n}{r}-80\delta$.
    This means $t \le r$ or $80\delta>\binom{t}{r}\ge(t/r)^r$, indicating $t<r(80\delta)^{1/r}<9r\delta^{1/r}$.
    % This gives $80\delta>\binom{t}{r}\ge\max\left( 0,(t/r)^r \right)$, so $t < r(80\delta)^{1/r} < 9r\delta^{1/r}$.
    Therefore, the number of common 2-cycles of $\psi$ and $\varphi$ is at least $\frac{n}{2}-9r\delta^{1/r}-t > \frac{n}{2}-18r\delta^{1/r}$, as desired.
\end{proof}
\cref{lemma: counts of normal phi r-graph} allows us to count, for given $G,\varphi,\psi$, the number of pairs of graphs $G_1,G_2$ with $G \overset{\varphi}{\rightarrow} G_1,G_2$ and $G_1 \overset{\psi}{\rightarrow} G_2$. 
In some situations, we want to count the number of graphs $G$ such that $G_1 \overset{\varphi_1}{\rightarrow} G$, $G_2 \overset{\varphi_2}{\rightarrow} G$, for given $G_1,G_2$ and $\varphi_1,\varphi_2$. 
When working inside a difference-isomorphic family $\mathcal{G}$, this can be done by counting pairs of such graphs $G$, say $G,G'$, using the fact that for every such $G,G'$ there exists some $\psi \in S_n$ with $G \overset{\psi}{\rightarrow} G'$ ($\mathcal{G}$ is difference-isomorphic), so we can use \cref{lemma: counts of normal phi r-graph}, and then union-bounding over all $\psi$. This is done in the following lemma.

% \begin{lemma}\label{lemma: codegree is usually small}
%     Let $r \ge 2$ and $n$ be sufficiently large in terms of $r$.
%     Let $\cal G$ be a difference-isomorphic family of $r$-graphs on $[n]$, and let $G_1,G_2 \in \cal G$ and $\varphi_1,\varphi_2 \in S_n$. Then $\abs{N_{\varphi_1}(G_1)\cap N_{\varphi_2}(G_2)\cap \mathcal{G}} \leq f_r(n)\cdot e^{-n^{r-1/2}}$, or $\varphi_1$ and $\varphi_2$ share at least $\left(\frac{n}{2}-n^{1-1/3r}\right)$ 2-cycles.
% \end{lemma}
% \begin{proof}
%     Let $\mathcal{G}':=N_{\varphi_1}(G_1)\cap N_{\varphi_2}(G_2)\cap \mathcal{G}$.
%     Since $\mathcal{G}$ is difference-isomorphic, it holds that $\sum_{\psi\in S_n}e_\psi(\mathcal{G}')\ge \abs{\mathcal{G}'}^2$.
%     By averaging, there exists $\psi \in S_n$ such that $e_\psi(\mathcal{G}') \ge \abs{\mathcal{G}'}^2/n!$.
%     When $n$ is sufficiently large, $e_\psi(N_{\varphi_1}(G_1)) \ge e_\psi(\mathcal{G}')\ge \abs{\mathcal{G}'}^2/n! \ge f_r(n)^2\cdot e^{-3n^{r-1/2}}$.
%     Now, \cref{lemma: counts of normal phi r-graph} with $G=G_1$,$\varphi=\varphi_1$ and $\delta=3n^{r-1/2}$ indicates $\varphi_1$ and $\psi$ share at least $A:=\frac{n}{2}-30r(3n^{r-1/2})^{1/r}$ 2-cycles.
%     In particular, $\psi$ has at least $A$ 2-cycles.
%     Similarly, $\varphi_2$ and $\psi$ also share at least $A$ 2-cycles.
%     Therefore, the number of 2-cycles contained by both $\varphi_1$ and $\varphi_2$ is at least $A-(\frac{n}{2}-A)\ge\frac{n}{2}-60r(3n^{r-1/2})^{1/r}> \frac{n}{2}-n^{1-1/3r}$ when $n$ is sufficiently large in terms of $r$.
% \end{proof}

\begin{lemma}\label{lemma: codegree is usually small}
    Suppose that $n$ is large enough in terms of $r$.
    Let $\cal G$ be a difference-isomorphic family of $r$-graphs on $[n]$, and let $G_1,G_2 \in \cal G$ and $\varphi_1,\varphi_2 \in S_n$. For every $\delta \geq \max\{n \ln{n}, n^{r/2}\}$ with $\delta = o(n^r)$, we have $\abs{N_{\varphi_1}(G_1)\cap N_{\varphi_2}(G_2)\cap \mathcal{G}} \leq 2^{f_r(n)}\cdot e^{-\delta}$ unless $\varphi_1$ and $\varphi_2$ share at least $\left(\frac{n}{2} - 80r\delta^{1/r} \right)$ 2-cycles.
\end{lemma}
\begin{proof}
    Let $\mathcal{G}':=N_{\varphi_1}(G_1)\cap N_{\varphi_2}(G_2)\cap \mathcal{G}$. Suppose that $|\mathcal{G}'| > 2^{f_r(n)} \cdot e^{-\delta}$.
    Since $\mathcal{G}$ is difference-isomorphic, we have $\sum_{\psi\in S_n}e_\psi(\mathcal{G}')\ge \abs{\mathcal{G}'}^2$.
    By averaging, there exists $\psi \in S_n$ such that $e_\psi(\mathcal{G}') \ge \abs{\mathcal{G}'}^2/n!$.
    When $n$ is sufficiently large, $e_\psi(N_{\varphi_1}(G_1)) \ge e_\psi(\mathcal{G}')\ge \abs{\mathcal{G}'}^2/n! > 2^{2f_r(n)}\cdot e^{-2\delta}/n! \geq 2^{2f_r(n)}\cdot e^{-3\delta}$, using that $\delta \geq n\ln{n}$.
    Now, by \cref{lemma: counts of normal phi r-graph} with $G=G_1$, $\varphi=\varphi_1$ and with $3\delta$ in place of $\delta$, we conclude that $\varphi_1$ and $\psi$ share at least 
    $\left(\frac{n}{2} - 18r(3\delta)^{1/r} \right) \geq \left(\frac{n}{2} - 40r\delta^{1/r} \right)$  
    $2$-cycles. 
    % $A:=\frac{n}{2}-30r(3n^{r-1/2})^{1/r}$ 2-cycles.
    % In particular, $\psi$ has at least $A$ 2-cycles.
    By the same argument for $\varphi_2$, we get that $\varphi_2$ and $\psi$ also share at least 
    $\left(\frac{n}{2} - 40r\delta^{1/r} \right)$
    2-cycles.
    % Therefore, the number of 2-cycles contained by both $\varphi_1$ and $\varphi_2$ is at least $A-(\frac{n}{2}-A)\ge\frac{n}{2}-60r(3n^{r-1/2})^{1/r}> \frac{n}{2}-n^{1-1/3r}$ when $n$ is sufficiently large in terms of $r$.
    It follows that $\varphi_1$ and $\varphi_2$ share at least $\left(\frac{n}{2} - 80r\delta^{1/r} \right)$ $2$-cycles.
\end{proof}

In the following lemma, we show that $e_\psi(N_\varphi(G))$ is much smaller than $2^{2f_r(n)}$ unless $\varphi,\psi$ have some very restrictive structure.
\begin{lemma}\label{lemma: general counts for exceptional r-graph}
    There exists an absolute constant $c>0$ such that the following holds.
    Let $r \ge 2$ and $n$ be sufficiently large in terms of $r$.
    Let $G$ be an $r$-graph and let $\varphi,\psi \in S_n$.
    Then $e_{\psi}(N_{\varphi}(G)) \leq 2^{2f_r(n)} \cdot e^{-\binom{n}{r-1}/100}$ unless the following holds.
    \begin{enumerate}
        \item[(1)] $\psi$ is an involution with at most $\left(2c n^{1-\frac{1}{r}}\right)$ 1-cycles (fixed points).
        \item[(2)] $\varphi,\psi$ share at least $\left(\frac{n}{2}-c n^{1-\frac{1}{r}}\right)$ 2-cycles and all the 1-cycles.
        \item[(3)] For every 2-cycle $xy$ of $\psi$, it holds that $\varphi(x) = y$ or $\varphi(y) = x$.
    \end{enumerate}
    In particular, if $\varphi,\psi$ are both involutions, then $e_{\psi}(N_{\varphi}(G)) \leq 2^{2f_r(n)} \cdot e^{-\binom{n}{r-1}/100}$ unless $\varphi=\psi$.
\end{lemma}
\begin{proof}
    We assume that $e_{\psi}(N_{\varphi}(G)) > 2^{2f_r(n)} \cdot e^{-\binom{n}{r-1}/100}$ and show that (1)-(3) hold.
    %>f_r(n)^2 \cdot e^{-\binom{n}{r-1}}$.
    Let $u_1v_2,\dots,u_tv_t$ be the common 2-cycles of $\varphi$ and $\psi$.
    By \cref{lemma: counts of normal phi r-graph} with $\delta=\binom{n}{r-1}$ (note that $\delta\ge n^{r/2}$ for all $r \geq 2$ and sufficiently large $n$), it holds that $t\ge\frac{n}{2}-18r\binom{n}{r-1}^{1/r}$.
    % By the Stirling's formula, $((r-1)!)^{1/r}\ge c'r$ for some constant $c'>0$.
    We have $\binom{n}{r-1}^{1/r} \leq (\frac{en}{r-1})^{\frac{r-1}{r}} \leq c' n^{1 - \frac{1}{r}}/r$ for some absolute constant $c'$. Hence, picking $c=18c'$, we have $t \geq \frac{n}{2}-cn^{1-\frac{1}{r}}$, proving the first part of Item (2). 
    % \begin{equation}
    %     \begin{aligned}
    %         t 
    %         \ge\frac{n}{2}-18r\binom{n}{r-1}^{1/r}
    %         \ge \frac{n}{2} - 18r\left(\frac{n^{r-1}}{(r-1)!}\right)^{1/r}
    %         \ge \frac{n}{2}-\frac{18}{c'}n^{1-\frac{1}{r}}
    %         = \frac{n}{2}-cn^{1-\frac{1}{r}}
    %         = (1-o(1))\frac{n}{2},
    %     \end{aligned}
    % \end{equation}
    
    Write $C = \{u_1,\dots,u_t,v_1,\dots,v_t\}$.
    If $C=[n]$, then $\varphi$ and $\psi$ share $\frac{n}{2}$ two-cycles, so $\varphi=\psi$ is an involution, and (1)-(3) are all satisfied.
    From now on, we assume that $C \neq [n]$.
    Suppose first that there exists $w \in [n]\setminus C$ such that for every edge $e \in \binom{[n]}{r}$ with $w \in e$ and $e \setminus \{w\} \subseteq C$, 
    $e$ does not belong to any good choosable pair for $(G,\varphi,\psi)$. As there are $\binom{2t}{r-1}$ edges $e$ with $w \in e$ and $e\setminus\{w\}\subset C$, it follows that the number of good choosable pairs $m_g$ (for $(G,\varphi,\psi)$) satisfies 
    \begin{equation}\label{eq:m_g bound, exceptional r-graphs lemma}
    2m_g \le \binom{n}{r}-\binom{2t}{r-1}= \binom{n}{r}-\binom{(1-o(1))n}{r-1}= \binom{n}{r}-(1-o(1))\binom{n}{r-1}.
    \end{equation}
    Also, letting $m$ be the number of choosable pairs for $(G,\varphi)$, we have the trivial bound $m \le \frac{1}{2}\binom{n}{r}$. If $r \geq 3$, then we use \cref{eq: practical m m_g bound} with $x = 0$ and $y = (\frac{1}{2}-o(1))\binom{n}{r-1}$ to get
    \begin{equation}\label{eq:bound for e_{psi} r>=3}
        \begin{aligned}
            e_\psi(N_\varphi(G))
            \le 2^{\binom{n}{r}}\cdot e^{-(\frac{1}{80}-o(1))\binom{n}{r-1}}
            = 2^{2f_r(n)+O(n^{r/2})}\cdot e^{-(\frac{1}{80}-o(1))\binom{n}{r-1}}
            \le 2^{2f_r(n)} \cdot e^{-\binom{n}{r-1}/100},
        \end{aligned}
    \end{equation}
    where the equality uses that $f_r(n) = \frac{1}{2}\binom{n}{r}-O(n^{r/2})$ (see \cref{eq: f_r(n)}), and the last inequality uses that $r/2 < r-1$ ($r \geq 3$) and that $n$ is sufficiently large.
    For $r = 2$, we improve the bound by using that $\varphi(u_iv_i) = u_iv_i$ for every $1 \leq i \leq t$, and thus $u_iv_i$ does not belong to any choosable pair of $(G,\varphi)$. Therefore, we have the improved bounds 
    $2m \leq \binom{n}{2} - t$ and 
    $2m_g \leq \binom{n}{2} - t - (1-o(1))n$ (improving on \eqref{eq:m_g bound, exceptional r-graphs lemma}). So by using \cref{eq: practical m m_g bound} with $x = t/2$ and $y = t/2 + (1/2-o(1))n$ we get
    \begin{equation}\label{eq:bound for e_{psi} r=2}
        e_\psi(N_\varphi(G)) \le 
        2^{\binom{n}{2} - t}\cdot e^{-(\frac{1}{80}-o(1))n} = 
        2^{2f_2(n) + o(n)} \cdot e^{-(\frac{1}{80}-o(1))n} \leq
        2^{2f_2(n)} \cdot e^{-n/100},
    \end{equation}
    where the inequality uses that $2f_2(n) = \binom{n}{2} - \floor{\frac{n}{2}}$ and $t = \frac{n}{2} - o(n)$. 
    Now, by \eqref{eq:bound for e_{psi} r>=3} and \eqref{eq:bound for e_{psi} r=2} we have $e_\psi(N_\varphi(G)) \le 2^{2f_r(n)} \cdot e^{-\binom{n}{r-1}/100}$, in contradiction to our assumption.

    So from now on, we assume that for every $w \in [n] \setminus C$, there exists $e \in \binom{n}{r}$ such that $w \in e$, $e \setminus \{w\} \subset C$, and $e$ lies in some good choosable pair for $(G,\varphi,\psi)$.
    By \cref{def: good choosable pair}, the other edge in this pair must be $f:=\psi(e)$, and we must have $\psi(f)=e$.
    Notice that $\psi(C)=C$.
    Therefore, $\psi(e\setminus C)=f\setminus C$ and $\psi(f \setminus C)=e \setminus C$.
    As $e\setminus C=\{w\}$, we get that $f\setminus C=\{w'\}$ for some $w' \in [n]\setminus C$ with $\psi(w)=w'$ and $\psi(w')=w$.
    This means $\psi^2(w)=w$. As this holds for every $w \in [n] \setminus C$, we get that $\psi$ is an involution.
    Also, since $(e,f)$ or $(f,e)$ is a choosable pair for $(G,\varphi)$, we have $\varphi(e)=f$ or $\varphi(f)=e$ (see \cref{def: good choosable pair}).
    As $\varphi(C)=C$, we get that $\varphi(e\setminus C)=f\setminus C$ or $\varphi(f \setminus C)=e \setminus C$, so $\varphi(w)=w'$ or $\varphi(w')=w$.
    Moreover, if $w$ is a fixed point of $\psi$ then $w' = w$, so $w$ is also a fixed point of $\varphi$. 
    This proves Items 2 and 3 in the lemma.
    % We have shown that $w$ lies in a 1-cycle or a 2-cycle of $\psi$; if $w$ forms a 1-cycle (fixed point) of $\psi$, then $w$ also forms a 1-cycle of $\varphi$; if $w$ lies in a 2-cycle of $\psi$, then $\varphi(w)=\psi(w)$ or $\varphi(\psi(w))=w$.
    % This builds (2) and (3).
    % Moreover, 
    % Since $w$ is arbitrary, $\psi$ must be an involution with at least $t$ 2-cycles and at most $n-2t\le 2cn^{1-\frac{1}{r}}$ 1-cycles, proving (1).
    Finally, since every fixed point of $\psi$ belongs to $[n] \setminus C$, and $|C| = 2t \geq n - 2cn^{1-1/r}$, Item (1) holds as well. 

    In particular, if $\varphi,\psi$ are both involutions, suppose $e_{\psi}(N_{\varphi}(G)) > 2^{2f_r(n)} \cdot e^{-\binom{n}{r-1}/100}$.
    By (2), $\varphi,\psi$ share all the fixed points.
    By (3), for every 2-cycle $x,y$ of $\psi$, we know $\varphi(x)=y$ or $\varphi(y)=x$, either meaning $xy$ is also a 2-cycle of $\varphi$.
    Hence, $\varphi=\psi$, as desired.
\end{proof}

\noindent
If $r \ge 3$, then we can impose even stronger restrictions on $\varphi,\psi$. 
\begin{lemma}\label{lemma: counts for exceptions r-graph}
    Suppose that $r \ge 3$ and $n$ is sufficiently large in terms of $r$.
    Let $G$ be an $r$-graph and let $\varphi,\psi \in S_n$.
    Then $e_{\psi}(N_{\varphi}(G)) \leq 2^{2f_r(n)} \cdot e^{-\binom{n}{r-2}/100}$ unless $\varphi=\psi$ is an involution.
\end{lemma}
\begin{proof}
    Suppose $e_{\psi}(N_{\varphi}(G)) > 2^{2f_r(n)} \cdot e^{-\binom{n}{r-2}/100}$. Then $(\varphi,\psi)$ satisfy (1)-(3) in  \cref{lemma: general counts for exceptional r-graph}. 
    In particular, $\psi$ is an involution, and $\varphi$ and $\varphi$ share all 1-cycles and at least $(1-o(1))\frac{n}{2}$ $2$-cycles.
    If $\varphi = \psi$ then we are done, so let us assume by contradiction that $\varphi \neq \psi$.
    Let $u_1v_1,\dots,u_tv_t$ be the common 2-cycles of $\varphi$ and $\psi$.
    We have $t=(1-o(1)))\frac{n}{2}$.
    % Then, \cref{lemma: general counts for exceptional r-graph} indicates that $\psi$ is an involution; $t \ge \frac{n}{2}-50n^{1-\frac{1}{r}}$; ; for every 2-cycle $xy$ of $\psi$, it holds that $\varphi(x)=y$ or $\varphi(y)=x$.
    % Suppose that $\varphi \neq \psi$. 
    % We are done when $\varphi=\psi$, so we may assume that $\varphi \neq \psi$.
    For convenience, denote $C=\{u_1,\dots,u_t,v_1,\dots,v_t\}$.
    Let $m$ be the number of choosable pair for $(G,\varphi)$ and let $m_g$ be the number of good choosable pairs for $(G,\varphi,\psi)$.

    Define an edge-set $E$ by 
    \begin{equation}
        E=
        \begin{cases}
            \big\{e \in \binom{C}{r}: |e\cap\{u_i,v_i\}|=0,2 \text{ for every } i = 1,\dots,t\big\} & r \text{ is even}, \\
            \big\{e \in \binom{C\cup\{w_0\}}{r}: w_0\in e, |e\cap\{u_i,v_i\}|=0,2 \text{ for every } i = 1,\dots,t \big\} & r \text{ is odd}, n\text{ is odd}, \\
            \emptyset & r \text{ is odd}, n\text{ is even}.
        \end{cases}
    \end{equation}
    where, in the second case, $w_0\in [n]$ is a fixed point of $\psi$ (which must exist if $n$ is odd). Note that $w_0$ is also a fixed point of $\varphi$, because $\varphi,\psi$ share all fixed points.
    Since $\varphi(w_0) = w_0$ and $\varphi(\{u_i,v_i\}) = \{u_i,v_i\}$ for every $1 \leq i \leq t$, we have that $\varphi(e) = e$ for every $e \in E$. 
    % Clearly, all edges in $E$ form a 1-cycle of $\tilde{\varphi}$, thus not contained in any choosable pairs for $(G,\varphi)$ or any good choosable pairs for $(G,\varphi,\psi)$.
    Hence, no edge of $E$ is contained in any choosable pair for $(G,\varphi)$. Therefore, $2m \le \binom{n}{r}-|E|$. Note also that $|E| = \binom{t}{\floor{r/2}} = \binom{n/2 - o(n)}{\floor{r/2}} = \binom{\floor{n/2}}{\floor{r/2}} - o(n^{\floor{r/2}})$, unless $r$ is odd while $n$ is even, in which case $|E| = 0$. 
    Using the definition of $f_r(n)$ and that $r \ge 3$, we get that
    \begin{equation}
    \label{eq: size of E exceptional for hypergraphs}
        \binom{n}{r} - |E| 
        = 2f_r(n) + o(n^{\floor{r/2}})
        = 2f_r(n) + o(n^{r-2}).
    \end{equation}
    % , hence $m \leq f_r(n) + O(n^{r/2-1})$.

    Next we upper-bound $m_g$. Since we assumed that $\varphi\neq \psi$, there exist $u,v\in [n]$ that form a 2-cycle in $\psi$ but not in $\varphi$.
    By Item 3 in \cref{lemma: general counts for exceptional r-graph}, we have $\varphi(u)=v$ or $\varphi(v)=u$, and we cannot have both because otherwise $uv$ is also a 2-cycle in $\varphi$.
    Without loss of generality, assume that $\varphi(u)=v$ and $\varphi(v)\neq u$.
    Consider any $e \in \binom{[n]}{r}$ such that $e\setminus C=\{u,v\}$.
    We claim that $e$ does not lie in any good choosable pair for $(G,\varphi,\psi)$.
    Suppose it does. 
    The other edge in this pair must be $f:=\psi(e)$ (by the definition of good choosable pairs).
    As $\psi(C)=C$ and $e\setminus C=\{u,v\}$, we know that $f\setminus C=\{\psi(u),\psi(v)\}=\{u,v\}$.
    Since $(e,f)$ is a choosable pair,  we must have $\varphi(e)=f$ or $\varphi(f)=e$.
    As $\varphi(C)=C$, we have $\varphi(e\setminus C)=f\setminus C$ or $\varphi(f\setminus C)=e \setminus C$, either of which implies $\varphi(uv)=uv$.
    But $\varphi(u)=v$, so $\varphi(v)=u$, contradicting our assumption.
    Therefore, all edges $e \in \binom{[n]}{r}$ with $e \setminus C=\{u,v\}$ do not lie in any good choosable pair for $(G,\varphi,\psi)$.
    Note that there at least $\binom{2t}{r-2}$ such edges, and they are not in $E$.
    So, $2m_g \le \binom{n}{r}-|E|-\binom{2t}{r-2}$.

    Combining all of the above, we can apply \cref{eq: practical m m_g bound} with $x = \frac{1}{2}|E|$ and $y = \frac{1}{2}|E| +\frac{1}{2}\binom{2t}{r-2}$, to \nolinebreak get
    $$
    e_\psi(N_\varphi(G)) \leq 2^{\binom{n}{r} - |E|} \cdot e^{-\binom{2t}{r-2}/80} \leq 
    2^{2f_r(n) + o(n^{r-2})} \cdot e^{-\binom{2t}{r-2}/80} \leq 
    2^{2f_r(n)} \cdot e^{-\binom{n}{r-2}/100},
    $$
    where the second inequality uses \cref{eq: size of E exceptional for hypergraphs}
    % \cref{eq: size of E, exceptional for hypergraphs} 
    and the last uses $r \geq 3$ and $2t = (1-o(1))n$.
\end{proof}

\subsection{Proof of \cref{theorem: stablility} for {$r \geq 3$}}
\label{subsec:main proof r>=3}
\newcommand{\neighbor}[2]{{N_{{#2}}^{{\cal G}}({#1})}}

We start with the following simple lemma, bounding the number of permutations that are close to a given permutation $\varphi_0$.
\begin{lemma}\label{lem: number of close permutations}
    Let $\varphi_0 \in S_n$ and let $0 \leq A \leq \frac{n}{2}$. The number of permutations $\varphi \in S_n$ that share at least $\left( \frac{n}{2} - A \right)$ 2-cycles with $\varphi_0$ is at most $n^{2A}$.
\end{lemma}
\begin{proof}
    If $\varphi_0$ has less than $\left( \frac{n}{2} - A \right)$ 2-cycles then there is nothing to prove. Otherwise, to specify $\varphi$ as in the statement, it suffices to choose the $\left( \frac{n}{2} - A \right)$ common 2-cycles and then specify a permutation on the remaining $2A$ vertices. 
    Using that $A \le \frac{n}{2}$, the number of choices is at most $\binom{n/2}{n/2-A} \cdot (2A)!=\binom{n/2}{A} \cdot (2A)!\le\frac{n^A(2A)!}{A!}\le n^{A}(2A)^{A} \leq n^{2A}$. 
\end{proof}

We now show that if $\mathcal{G}$ is a difference-isomorphic $r$-graph family of size close to $2^{f_r(n)}$, then $\mathcal{G}$ contains a large involution clique. 
Note that the following lemma works also for $r=2$, and will be used in \cref{subsec:main proof r=2} as well.

\begin{lemma}\label{theorem without large involution clique}
    Let $r\ge 2$, $n$ be sufficiently large, and let $\delta$ satisfy
    % $n^{1/2}\log^{3/2} n \ll \delta \ll n^{r-1}$.
    $n^{1/2}\log^{3/2} n \ll \delta \leq \frac{1}{250}\binom{n}{r-1}$.
    % \leq \frac{1}{125}\binom{n}{r-1}$.
    Let $\cal G$ be a difference-isomorphic family of $r$-graphs on $[n]$. Then $\abs{\cal G} \le 2^{f_r(n)}\cdot e^{-\delta/4}$ or $\mathcal{G}$ contains an involution clique of size at least $2^{f_r(n)}\cdot e^{-\delta}$.
    % \zj{Put a constant in the range of $\delta$! Otherwise the theorem proof will not give $O(n^{r-2})$.}
\end{lemma}
\begin{proof}
    We assume $\mathcal{G}$ has no involution clique of size $2^{f_r(n)}\cdot e^{-\delta}$, and prove that $\abs{\cal G} \le 2^{f_r(n)}\cdot e^{-\delta/4}$.
    For convenience, put $\neighbor{G}{\varphi} := N_{\varphi}(G) \cap \mathcal{G}$.
    We first show $\neighbor{G}{\varphi}$ is small for all $G \in \mathcal{G}$ and $\varphi \in S_n$.
    The intuition is that $e_\psi(\neighbor{G}{\varphi})$ is ``large'' only if $\psi$ is an involution (see \cref{lemma: general counts for exceptional r-graph}), which further implies a ``large'' $\psi$-clique inside $\neighbor{G}{\varphi}$ (this is impossible by our assumption).
    \begin{claim}\label{claim: phi-neighbourhood is small}
        $\abs{\neighbor{G}{\varphi}} \le 2^{f_r(n)}\cdot e^{-\delta/2}$ for all $G \in \mathcal{G}$ and $\varphi \in S_n$.
    \end{claim}
    \begin{proof}
        Since $\cal G$ is difference-isomorphic, we have 
        \begin{equation}\label{eq: sum of edges in N_{phi} over all psi}
            \abs{\neighbor{G}{\varphi}}^2\le \sum_{\psi \in S_n} e_\psi(\neighbor{G}{\varphi})
        \end{equation}
        Set $\Delta := 2\delta + n\ln{n} + n^{r/2}$ and $A := 18r\Delta^{1/r}$. 
        Note that $\Delta = o(n^r)$ and $A = O(\delta^{1/r} + (n\ln{n})^{1/r} + n^{1/2}) = o(\delta/\ln n)$, using $\delta\gg n^{1/2}\log^{3/2} n$ and $r \ge 2$.
        Partition $S_n$ into parts $P_1\cup P_2\cup P_3$ as follows:
        % \begin{equation} \nonumber
        %     \begin{aligned}
        %         P_1 &:=\left\{\psi\in S_n: \varphi\text{ and }\psi\text{ share fewer than } \left(\frac{n}{2}-A\right) \text{ 2-cycles}\right\}, \\
        %         P_2 &:=\left\{\psi\in S_n\setminus P_1: \psi^2\neq\identity_{[n]}, \text{ or } \varphi\text{ and }\psi\text{ share fewer than } \left(\frac{n}{2}-B\right) \text{ 2-cycles}\right\}, \\
        %         P_3 &:=S_n \setminus (P_1\cup P_2).
        %     \end{aligned}
        % \end{equation}
        \begin{equation} \nonumber
            \begin{aligned}
                P_1 &:=\left\{\psi\in S_n: \varphi\text{ and }\psi\text{ share fewer than } \left(\frac{n}{2}-A\right) \text{ 2-cycles}\right\}, \\
                P_2 &:=\left\{\psi\in S_n\setminus P_1: \psi \text{ is not an involution}\right\}, \quad
                P_3 :=S_n \setminus (P_1\cup P_2).
            \end{aligned}
        \end{equation}
        Clearly, $\abs{P_1}\le |S_n| = n!$. By \cref{lemma: counts of normal phi r-graph} (with $\Delta$ in place of $\delta$) and by our choice of $A$, we have $e_\psi(\neighbor{G}{\varphi}) \le 2^{2f_r(n)} \cdot e^{-\Delta}$ for every $\psi \in P_1$. Therefore, the contribution of $P_1$ to the RHS of \cref{eq: sum of edges in N_{phi} over all psi} is 
        $\sum_{\psi \in P_1} e_\psi(\neighbor{G}{\varphi}) \leq n! \cdot 2^{2f_r(n)} \cdot e^{-\Delta} \leq 2^{2f_r(n)} \cdot e^{-2\delta}$, using that $\Delta \geq n \ln{n} + 2\delta$. 

        By \cref{lem: number of close permutations} with $\varphi_0=\varphi$ and $A$, we have $|P_2\cup P_3|\le n^{2A}\leq e^{o(\delta)}$ using that $A=o(\delta/\ln n)$.
        % Next, we bound the size of $|P_2 \cup P_3|$. By definition, every $\psi \in P_2 \cup P_3 = S_n \setminus P_1$ has at least $\frac{n}{2} - A$ common 2-cycles with $\varphi$. Hence, in order to choose such a $\psi$, it suffices to choose the common 2-cycles and then specify a permutation on the remaining $2A$ vertices. It follows that
        % In addition, by enumerating first the $\left(\frac{n}{2}-A\right)$ common 2-cycles (from the 2-cycles of $\varphi$) and then the permutation on the rest vertices, we know that 
        % $\abs{P_2\cup P_3}\le \binom{n/2}{n/2-A}(2A)! \leq n^{A} (2A)^{2A} \leq e^{O_r(\Delta^{1/r}\log n)} \leq e^{O_r(\log n)\cdot (\delta^{1/r} + (n\log n)^{1/r} + n^{1/2})} \leq e^{o(\delta)}$, where the third inequality uses the choice of $A$, the fourth uses the choice of $\Delta$, and the last uses $r \geq 2$ and the assumption that $\delta \gg n^{1/2}\log^{3/2} n$.
        Now, by \cref{lemma: general counts for exceptional r-graph}, we have $e_\psi(\neighbor{G}{\varphi})\le 2^{2f_r(n)}\cdot e^{-\binom{n}{r-1}/100}$ for every $\psi \in P_2$ (because $\psi$ is not an involution). Hence, the contribution of $P_2$ is  
        $\sum_{\psi \in P_2} e_\psi(\neighbor{G}{\varphi}) \leq e^{o(\delta)} \cdot 2^{2f_r(n)}\cdot e^{-\binom{n}{r-1}/100} \leq 2^{2f_r(n)} \cdot e^{-2\delta}$, using 
        the assumption that $\delta \le \frac{1}{250}\binom{n}{r-1}$.
        %$\delta\le\frac{1}{250} \binom{n}{r-1}$ by assumption.
        % that $\delta \ll n^{r-1}$ by assumption.

        Finally, we consider $P_3$. 
        Every $\psi \in P_3$ is an involution.
        Therefore, for every $G_1 \in \neighbor{G}{\varphi}$, the set $\neighbor{G_1}{\psi}$ is a $\psi$-clique by \cref{prop: involution relation}, and so $\abs{\neighbor{G_1}{\psi}} \le 2^{f_r(n)}\cdot e^{-\delta}$ by our assumption that all involution cliques are small.
        This means that 
        $$e_\psi(\neighbor{G}{\varphi})\le\sum_{G_1\in\neighbor{G}{\varphi}} \abs{\neighbor{G_1}{\psi}}\le \abs{\neighbor{G}{\varphi}}\cdot 2^{f_r(n)}\cdot e^{-\delta}.$$ 
        Hence, 
        $\sum_{\psi \in P_3} e_\psi(\neighbor{G}{\varphi}) \leq 
        |P_3| \cdot \abs{\neighbor{G}{\varphi}}\cdot 2^{f_r(n)}\cdot e^{-\delta} \leq \abs{\neighbor{G}{\varphi}}\cdot 2^{f_r(n)}\cdot e^{-2\delta/3}$, using $|P_3| \leq |P_2\cup P_3|\leq e^{o(\delta)}$.
        Plugging all of the above into \eqref{eq: sum of edges in N_{phi} over all psi}, we get
        % \begin{equation} \nonumber
        %     \begin{aligned}
        %         \abs{\neighbor{G}{\varphi}}^2
        %         &\le \sum_{\psi \in P_1\cup P_2 \cup P_3} e_\psi(\neighbor{G}{\varphi}) \\
        %         &\le n! \cdot f_r(n)^2 \cdot e^{-\Theta_r(n^{r-1/2})} + e^{n^{1-1/3r}} \cdot f_r(n)^2\cdot e^{-\binom{n}{r-1}/100} + e^{n^{1-1/3r}} \cdot \abs{\neighbor{G}{\varphi}}\cdot f_r(n)\cdot e^{-\delta} \\ 
        %         &\le 2e^{n^{1-1/3r}} \cdot f_r(n)^2\cdot e^{-\binom{n}{r-1}/100} + e^{n^{1-1/3r}} \cdot \abs{\neighbor{G}{\varphi}}\cdot f_r(n)\cdot e^{-\delta}
        %     \end{aligned}
        % \end{equation}
       \begin{equation} \nonumber
            \begin{aligned}
                \abs{\neighbor{G}{\varphi}}^2
                &\le \sum_{\psi \in P_1\cup P_2 \cup P_3} e_\psi(\neighbor{G}{\varphi}) \leq 
                2^{2f_r(n)}\cdot 2e^{-2\delta} + 
                \abs{\neighbor{G}{\varphi}}\cdot 2^{f_r(n)}\cdot e^{-2\delta/3}.
            \end{aligned}
        \end{equation}
        So $\abs{\neighbor{G}{\varphi}}^2$ is at most twice the first term or at most twice the second term on the RHS above. 
        Namely, 
        $\abs{\neighbor{G}{\varphi}}^2 \le 2 \cdot 2^{2f_r(n)} \cdot 2e^{-2\delta}$ or $\abs{\neighbor{G}{\varphi}} \leq 2 \cdot 2^{f_r(n)} \cdot e^{-2\delta/3}$. In either case, we get 
        $\abs{\neighbor{G}{\varphi}} \leq 2^{f_r(n)} \cdot e^{-\delta/2}$, as required. 
    \end{proof}

    We now continue with the proof of the lemma. 
    Fix an arbitrary $G \in \mathcal{G}$.
    Since $\cal G$ is difference-isomorphic, $\mathcal{G}=\bigcup_{\varphi\in S_n} \neighbor{G}{\varphi}$.
    By averaging, there exists $\varphi_1 \in S_n$ such that $\abs{\neighbor{G}{\varphi_1}}\ge\abs{\cal G}/n!$.
    Let $\mathcal{G}_1=\neighbor{G}{\varphi_1}$ and $\mathcal{G}_2=\mathcal{G}\setminus\mathcal{G}_1$.
    For $\varphi_2 \in S_n$, let $\mathcal{E}_{\varphi_2}=\{(G_1,G_2)\in\mathcal{G}_1\times\mathcal{G}_2: G_2 \overset{\varphi_2}{\rightarrow} G_1\}$.
    Again, using that $\cal G$ is difference-isomorphic, it holds that $\bigcup_{\varphi_2\in S_n}\mathcal{E}_{\varphi_2}=\mc{G}_1\times\mc{G}_2$. 
    % ; namely, for every $(G_1,G_2) \in \mc{G}_1\times\mc{G}_2$ there is $\varphi \in S_n$ with $G_2 \overset{\varphi}{\rightarrow} G_1$. 
    Therefore, $\abs{\mc{G}_1}\abs{\mc{G}_2}\le\sum_{\varphi_2\in S_n}\abs{\mc{E}_{\varphi_2}}$. We will use this to upper bound $|\mathcal{G}_1|  |\mathcal{G}_2|$.
    Now, set $\Delta := \delta + 2n\ln n + n^{r/2}$ and $A = 80r\Delta^{1/r}$, and note that $A=o(\delta/\ln n)$ because $\delta\gg n^{1/2}\log^{3/2} n$ and $r \geq 2$.
    Define 
    \begin{equation} \nonumber
        \begin{aligned}
            Q =\left\{\varphi_2\in S_n: \varphi_1\text{ and }\varphi_2\text{ share fewer than } \left(\frac{n}{2}-A\right) \text{ 2-cycles}\right\}.
        \end{aligned}
    \end{equation}
    
    We bound $|\mathcal{E}_{\varphi_2}|$ separately for $\varphi_2 \in Q$ and $\varphi_2 \in [n] \setminus Q$. 
    First, fix any $\varphi_2 \in Q$ and $G_2\in\mc{G}_2$.
    By \cref{lemma: codegree is usually small} (with $G_1 := G$ and with $\Delta$ in place of $\delta$), and using our choice of $A$ and the assumption $\varphi_2 \in Q$, we have $\abs{\neighbor{G}{\varphi_1}\cap\neighbor{G_2}{\varphi_2}}< 2^{f_r(n)}\cdot e^{-\Delta} \leq 2^{f_r(n)} \cdot n^{-2n} \cdot e^{-\delta}$, using $\Delta \geq 2n\ln n + \delta$. 
    By summing over $G_2\in\mc{G}_2$, we get $\abs{\mc{E}_{\varphi_2}}\le\abs{\mc{G}_2}\cdot 2^{f_r(n)}\cdot n^{-2n} \cdot e^{-\delta}$, and by summing over all at most $n!$ permutations $\varphi_2 \in Q$, we obtain $\sum_{\varphi_2\in S_n}\abs{\mc{E}_{\varphi_2}} \leq \abs{\mc{G}_2} \cdot 2^{f_r(n)} \cdot n^{-n} \cdot e^{-\delta}$.
    
    Next, fix any $\varphi_2 \in S_n\setminus Q$ and $G_1 \in \mc{G}_1$. 
    Note that $G_2 \overset{\varphi_2}{\rightarrow} G_1$  if and only if $G_1 \overset{\varphi_2^{-1}}{\rightarrow} G_2$. Thus, $\mathcal{E}_{\varphi_2}$ is the set of pairs $(G_1,G_2) \in \mathcal{G}_1 \times \mathcal{G}_2$ with $G_1 \overset{\varphi_2^{-1}}{\rightarrow} G_2$.
    By \cref{claim: phi-neighbourhood is small}, $\abs{N_{\varphi_2^{-1}}(G_1)\cap\mc{G}_2}\le\abs{\neighbor{G_1}{\varphi_2^{-1}}} \le 2^{f_r(n)}\cdot e^{-\delta/2}$.
    By summing over $G_1\in\mc{G}_1$, we get $\abs{\mc{E}_{\varphi_2}}\le\abs{\mc{G}_1}\cdot 2^{f_r(n)}\cdot e^{-\delta/2}$.
    By \cref{lem: number of close permutations} with $\varphi_0=\varphi_1$ and $A$ as above, it holds that $\abs{S_n\setminus Q}\le n^{2A} \leq e^{o(\delta)}$, using $A=o(\delta/\ln n)$.
    % In addition, similarly as in the proof of \cref{claim: phi-neighbourhood is small}, we can bound $\abs{S_n \setminus Q}$ by 
    % $\abs{S_n\setminus Q}\le \binom{n/2}{n/2-A}(2A)! \leq n^{A} (2A)^{2A} \leq e^{O_r(\Delta^{1/r}\log n)} \leq e^{O_r(\log n)\cdot (\delta^{1/r} + (n\log n)^{1/r} + n^{1/2})} \leq e^{o(\delta)}$, using the choice of $\Delta$ and $A$ and the assumption $\delta \gg n^{1/2}\log^{3/2}n$.
    Combining all of the above, we get 
    % \begin{equation} \nonumber
    %     \begin{aligned}
    %         \abs{\mc{G}_1}\abs{\mc{G}_2}
    %         \le\sum_{\varphi_2\in S_n}\abs{\mc{E}_{\varphi_2}}
    %         \le n!\cdot \abs{\mc{G}_2}f_r(n)\cdot e^{-n^{r-1/2}} + 2^{n^{1-1/4r}}\cdot \abs{\mc{G}_1}\cdot f_r(n)\cdot e^{-\delta/2}.
    %     \end{aligned}
    % \end{equation}
     \begin{equation} \nonumber
        \begin{aligned}
            \abs{\mc{G}_1}\abs{\mc{G}_2}
            \le\sum_{\varphi_2\in S_n}\abs{\mc{E}_{\varphi_2}}
            \le 
            \abs{\mc{G}_2} \cdot 2^{f_r(n)} \cdot n^{-n} \cdot e^{-\delta} +   \abs{\mc{G}_1}\cdot 2^{f_r(n)}\cdot e^{-\delta/2 + o(\delta)}.
        \end{aligned}
    \end{equation}
    So $\abs{\mc{G}_1}\abs{\mc{G}_2}$ is at most twice the first term or at most twice the second term on the RHS above. 
    Namely, $\abs{\mc{G}_1} \le 2 \cdot 2^{f_r(n)} \cdot n^{-n} \cdot  e^{-\delta}$ or $\abs{\mc{G}_2}\le 2\cdot 2^{f_r(n)}\cdot e^{-\delta/2 + o(\delta)}$.
    In the former case, using $\abs{\mc{G}_1}\ge\abs{\mc{G}}/n!$, we get 
    $|\mathcal{G}| \leq 2 n! \cdot 2^{f_r(n)} \cdot n^{-n} \cdot  e^{-\delta} \leq 2^{f_r(n)}\cdot e^{-\delta}$, as required.
    % \begin{equation} \nonumber
    %     \abs{\mc{G}} \le 2(n!)^2f_r(n)\cdot e^{-n^{r-1/2}}
    %     \le f_r(n)\cdot e^{-n^{r-1/2}/2}
    %     \le f_r(n)\cdot e^{-\delta/4};
    % \end{equation}
    And in the latter case, by combining the bound on $|\mathcal{G}_2|$ with the fact that 
    $\abs{\mc{G}_1}=\abs{\neighbor{G}{\varphi_1}}\le 2^{f_r(n)}\cdot e^{-\delta/2}$ (by \cref{claim: phi-neighbourhood is small}), we get that
    $\abs{\mc{G}}=\abs{\mc{G}_1}+\abs{\mc{G}_2}
        \le 3\cdot 2^{f_r(n)}\cdot e^{-\delta/2 + o(\delta)}
        \le 2^{f_r(n)}\cdot e^{-\delta/4},$ again giving the desired bound. 
    % \begin{equation} \nonumber
    %     \abs{\mc{G}}=\abs{\mc{G}_1}+\abs{\mc{G}_2}
    %     \le 3 f_r(n)\cdot e^{-\delta/2 + o(\delta)}
    %     \le f_r(n)\cdot e^{-\delta/4}.
    % \end{equation}
\end{proof}
\noindent
We are finally ready to prove \cref{theorem: stablility} in the case $r \ge 3$.

\begin{proof}[Proof of \cref{theorem: stablility} for $r\ge 3$]
\let\neighbor\relax
\newcommand{\neighbor}[2]{{N_{#2}(#1)}}
    Let $\cal G'\subseteq \cal G$ be a largest involution clique in $\mathcal{G}$.
    Say $\cal G'$ is a $\psi$-clique, where $\psi \in S_n$.
    % and let $\psi \in S_n$ be the involution corresponding to $\mathcal{G}'$.
    By assumption, $\mc{G}\neq \mc{G}'$.
    % If $\cal G$ is an involution clique, then $\cal G=\cal G'$, and $\abs{\cal G} \le f_r(n)$ by \cref{prop: property of involution clique}.
    % Now, assume that $\cal G$ is not an involution clique, i.e. ${\cal G}\setminus{\cal G'} \neq \emptyset$.
    It suffices to show that $\abs{\cal G'} < 2^{f_r(n)}\cdot n^{-\binom{n}{r-2}/250}$ because then, \cref{theorem without large involution clique} with $\delta=\frac{1}{250}\binom{n}{r-2}\ll n^{r-1}$ implies that $\abs{\mc{G}} \le 2^{f_r(n)}\cdot e^{-\binom{n}{r-2}/1000}$, as \nolinebreak desired.

    So our goal is to show that $\abs{\cal G'}\le 2^{f_r(n)}\cdot n^{-\binom{n}{r-2}/250}$.
    Let $G \in \mathcal{G}\setminus \mathcal{G}'$.
    As $\cal G$ is difference-isomorphic, $\mathcal{G}'=\bigcup_{\varphi\in S_n}(\neighbor{G}{\varphi}\cap\mathcal{G}')$.
    Note that $\neighbor{G}{\psi}\cap\mathcal{G'}=\emptyset$, because otherwise $\mathcal{G}' \cup \{G\}$ would be a $\psi$-clique (by \cref{prop: involution relation}), contradicting the maximality of $\mathcal{G}'$. Hence, $\abs{\mathcal{G'}} \le \sum_{\varphi \in S_n \setminus \{\psi\}} \abs{\neighbor{G}{\varphi}\cap\mathcal{G}'}$.
    Set $A:=cn^{1-\frac{1}{r}}$, where $c$ is the constant given by \cref{lemma: general counts for exceptional r-graph}, and let 
    $$
        P_1=\left\{\varphi\in S_n: \varphi\text{ and }\psi\text{ share fewer than } \left(\frac{n}{2}-A\right) \text{ 2-cycles} \right\},
        \quad P_2=S_n \setminus (P_1 \cup \{\psi\}).
    $$
    By \cref{lemma: general counts for exceptional r-graph}, $e_{\psi}(\neighbor{G}{\varphi})\le 2^{2f_r(n)}\cdot e^{-\binom{n}{r-1}/100}$ for all $\varphi \in P_1$. By \cref{lemma: counts for exceptions r-graph}, $e_{\psi}(\neighbor{G}{\varphi})\le 2^{2 f_r(n)}\cdot e^{-\binom{n}{r-2}/100}$ for all $\varphi \in S_n \setminus \{\psi\}$, and in particular for all $\varphi \in P_2$. 
    Clearly, $\neighbor{G}{\varphi}\cap\mathcal{G}'\subseteq\mc{G}'$ is also a ${\psi}$-clique, so $e_\psi(\neighbor{G}{\varphi})\ge\abs{\neighbor{G}{\varphi}\cap\mathcal{G}'}^2$ for all $\varphi\in S_n$.
    Then, $\abs{\neighbor{G}{\varphi}\cap\mathcal{G}'}\le 2^{f_r(n)}\cdot e^{-\binom{n}{r-1}/200}$ for $\varphi \in P_1$ and $\abs{\neighbor{G}{\varphi}\cap\mathcal{G}'}\le 2^{f_r(n)}\cdot e^{-\binom{n}{r-2}/200}$ for $\varphi \in P_2$.
    Note that $|P_2|\le n^{2A} \leq e^{o(n)}$ by \cref{lem: number of close permutations} with $\varphi_0=\psi$.
    % In addition, by enumerating $\varphi \in P_2$ via first picking $A$ common 2-cycles (that are in $\varphi$) and then putting a permutation on the rest vertices, we acquire 
    % $$
    %     \abs{P_2} 
    %     \le \binom{n/2}{n/2-A}(2A)!
    %     \le 2^n\cdot (100n^{1-1/r})^{100n^{1-1/r}}
    %     = 2^{o_r(n)},
    % $$
    % Now, the claim follows: when $n$ is sufficiently large in terms of $r$, 
    Combining all of the above and using that $r \geq 3$ and $n$ is sufficiently large, we get:
    \begin{equation} \nonumber
        \begin{aligned}
            \abs{\mathcal{G'}} 
            \le&\, \sum_{\varphi\neq\psi} \abs{\neighbor{G}{\varphi}}
            = \sum_{\varphi\in P_1} \abs{\neighbor{G}{\varphi}}+\sum_{\varphi\in P_2} \abs{\neighbor{G}{\varphi}} \\
            \le&\, n!\cdot 2^{f_r(n)}\cdot e^{-\binom{n}{r-1}/200} + e^{o(n)}\cdot 2^{f_r(n)}\cdot e^{-\binom{n}{r-2}/200}
            \le 2^{f_r(n)}\cdot e^{-\binom{n}{r-2}/250},
        \end{aligned}
    \end{equation}
    completing the proof. 
\end{proof}

\subsection{Proof of \cref{theorem: stablility} for {$r = 2$}}
\label{subsec:main proof r=2}

% It turns out that the case $r =2$ is more intricate than the case $r\ge 3$.
% For $r=2$, based on \cref{lemma: general counts for exceptional r-graph}, we make the following definition. 
Throughout this section, we assume that $r=2$. 
We begin with some lemmas. 
First, it will be convenient to introduce the following definition, giving a name to pairs $(\varphi,\psi)$ which satisfy (1)-(3) in
\cref{lemma: general counts for exceptional r-graph}.
\begin{definition}[Exceptional]\label{def:exceptional}
A pair of permutations $(\varphi,\psi)$ is {\em exceptional} if the following holds, where $c$ is the constant given by \cref{lemma: general counts for exceptional r-graph}.
\begin{enumerate}
    \item[(1)] $\psi$ is an involution with at most $2c\sqrt{n}$ 1-cycles (fixed points).
    \item[(2)] $\varphi,\psi$ share at least $\left(\frac{n}{2}-c\sqrt{n}\right)$ 2-cycles and all the 1-cycles.
    \item[(3)] For every two-cycle $xy$ in $\psi$, it holds that $\varphi(x) = y$ or $\varphi(y) = x$.
\end{enumerate}
\end{definition}
Note that $(\varphi,\psi)$ is exceptional if and only if $(\varphi^{-1},\psi)$ is. 
\cref{lemma: general counts for exceptional r-graph} states that $e_{\psi}(N_{\varphi}(G)) \leq 2^{2f_2(n)} \cdot e^{-n/100}$ unless $(\varphi,\psi)$ is exceptional. We strengthen this statement by adding a restriction on the graph $G$:
\begin{lemma}\label{lemma: counts for exceptions}
    Suppose that $r=2$ and $n$ is sufficiently large.
    Let $G$ be a graph and let $\varphi,\psi \in S_n$.
    Then $e_{\psi}(N_{\varphi}(G)) \leq 2^{2f_2(n)} \cdot e^{-n/100}$, unless the following holds:
    \begin{enumerate}
       \item $(\varphi,\psi)$ is exceptional.
       \item For every 2-cycle $xy$ of $\psi$, if $\varphi(x) = y$ but $\varphi(y)\neq x$, then $d_G(x) > \frac{n}{2}$ and $d_G(y) < \frac{n}{2}$.
    \end{enumerate}
\end{lemma}
\begin{proof}
    Suppose $e_\psi(N_\varphi(G)) > 2^{2f_2(n)} \cdot e^{-n/100}$. That $(\varphi,\psi)$ is exceptional follows from \cref{lemma: general counts for exceptional r-graph}. So it remains to establish Item 2 of the lemma.
    Let $u_1v_1,\dots,u_tv_t$ be the common 2-cycles of $\varphi$ and $\psi$, where $t=\frac{n}{2}-O(\sqrt{n})$.
    Write $C=\{u_1,\dots,u_t,v_1,\dots,v_t\}$.
    Let $m$ be the number of choosable pairs of $(G,\varphi)$ and let $m_g$ be the number of good choosable pairs of $(G,\varphi,\psi)$. Since $\varphi(u_iv_i) = u_iv_i$, none of the edges $u_iv_i$ belongs to any choosable pair, hence $2m \leq \binom{n}{2} - t$.

    Now, let $xy$ by any 2-cycle of $\psi$ such that $\varphi(x)=y$ and $\varphi(y)\neq x$.
    % If $\varphi=\psi$, then Item 2 holds trivially, so suppose that $\varphi \neq \psi$. Then by (2)-(3) in \cref{def:exceptional}, there must be a 2-cycle $xy$ of $\psi$ such that $\varphi(x)=y$ and $\varphi(y)\neq x$.
    Clearly, $x,y \notin C$.
    Denote $C'=\{a \in C: ax \in G \text{ and } \psi(a)y \notin G\}$ and $C'' = C \setminus C'$. 
    We claim that for every $a \in C''$, $ax$ and $\psi(a)y$ do not belong to any good choosable pair of $(G,\varphi,\psi)$. 
    Indeed, let $a \in C''$, and put $b = \psi(a)$ (so $ab = u_iv_i$ for some $i \in [t]$).
    By \cref{def: good choosable pair}, a good choosable pair $(e,f)$ satisfies $\varphi(e) = f$, $e \in G$, $f \notin G$, $\psi(e) = f$ and $\psi(f) = e$. Since $\psi(x) =y$, $\psi(y) = x$, $\psi(a) = b$ and $\psi(b) = a$, a good choosable pair containing $ax$ or $by$ must be $(ax,by)$ or $(by,ax)$. 
    In fact, this pair must be $(ax,by)$, because $\varphi(by) \neq ax$ (as $\varphi(b)=a$ and $\varphi(y) \neq x$). 
    However, as $a \notin C'$, we have $ax \notin G$ or $by \in G$, meaning that $(ax,by)$ is not a choosable pair, thus proving the claim. 
    This gives $2|C''|$ additional edges which do not belong to any good choosable pair (i.e., in addition to the edges $u_1v_1,\dots,u_tv_t$). Hence, $2m_g \leq \binom{n}{2} - t - 2|C''|$.

    Now we show that $|C'| > n/2$. 
    Otherwise, we would have $|C''|=2t-|C'| \geq \frac{n}{2} - O(\sqrt{n})$, and hence 
    $2m_g \leq \binom{n}{2}-t-n+O(\sqrt{n})$.
    % \begin{equation} \nonumber
    %     m_g \le \frac{1}{2}\left(\binom{n}{2}-t-2|C''|\right) \leq 
    %     \frac{1}{2}\left(\binom{n}{2}-\frac{n}{2}-n+O(\sqrt{n})\right)
    % \end{equation}
    Now, by \cref{eq: practical m m_g bound} with $x = \frac{t}{2}$ and $y = \frac{t}{2} + \frac{n}{2} - O(\sqrt{n})$, we get 
    \begin{equation} \nonumber
        \begin{aligned}
            e_\psi(N_\varphi(G))
            \le 2^{\binom{n}{2}-t} \cdot e^{-\frac{n}{80} + O(\sqrt{n})} 
            \le 2^{2f_2(n)+O(\sqrt{n})}\cdot e^{-\frac{n}{80} + O(\sqrt{n})} 
            < 2^{2f_2(n)}\cdot e^{-n/100},
        \end{aligned}
    \end{equation}
    where we used that $2f_2(n)=\binom{n}{2}-\floor{\frac{n}{2}}$ (see \cref{eq: f_r(n)}) and $t = \frac{n}{2} - O(\sqrt{n})$.
    This contradicts our assumption in the beginning of the proof.
    So indeed $|C'| > n/2$. Now, by the definition of $C'$, we have that $\deg_G(x) \geq |C'|$ and $\deg_G(y) \leq n-1 - |C'|$. Hence, $\deg_G(x)>\frac{n}{2}>\deg_G(y)$, as \nolinebreak required.  
\end{proof}

% We assume that $r=2$ throughout this section.
% By combining \cref{lemma: counts of normal phi r-graph} and \cref{lemma: counts for exceptions}, we will be able to show that a large difference-isomorphic graph family (i.e., of size comparable to $f_2(n)$) must have a large $\psi$-clique for some involution $\psi$. We will then use the following lemma to show that in fact, the number of graphs not belonging to this $\psi$-clique is very small, i.e. at most $f_2(n) \cdot 2^{-\Omega(n)}$.

\cref{theorem without large involution clique} implies that a large difference-isomorphic graph family must have a large $\psi$-clique for some involution $\psi$. In \cref{lemma: any non-small involution clique is large} we will show that this $\psi$-clique must in fact span almost the entire family, meaning that the number of graphs not belonging to the $\psi$-clique is exponentially smaller than $2^{f_2(n)}$. The proof of \cref{lemma: any non-small involution clique is large} relies on the following lemma.

\begin{lemma}\label{lem: leftover of psi clique is small}
    Let $(\varphi,\psi)$ be an exceptional pair with $\varphi \neq \psi$, and let $\mathcal{G}$ be a $\psi$-clique.
    % of size at least $f_2(n) \cdot \dots$. 
    Then the number of pairs of graphs $(G_1,G_2)$ with $G_1 \in \mathcal{G}$ and $G_1 \overset{\varphi}{\rightarrow} G_2$ is at most $2^{2f_2(n)} \cdot 2^{-\Omega(n)}$. 
\end{lemma}
\begin{proof}
    % By Item 2 of \Cref{lem:exceptional}, we know that $\varphi$ is not an involution.
    % So fix a cycle $C$ of $\varphi$ of length at least $3$. By Item 1 of \Cref{lem:exceptional}, we may write
    % $C = x_1,\dots,x_{2\ell}$ (where $\ell \geq 2$) and assume without loss of generality that $x_{2i-1}x_{2i}$ is a $2$-cycle of $\psi$ for every $1 \leq i \leq \ell$. Let $\mathcal{C}_2$ be the set of 2-cycles shared by $\varphi$ and $\psi$. By Item 2 of \Cref{def:exceptional}, we have $|\mathcal{C}_2| \geq \frac{n}{2} - O(\sqrt{n}) \geq \frac{n}{3}$. Fix $yz \in \mathcal{C}_2$. Then $\varphi(x_{2i-1}y) = x_{2i}z$, $\varphi(x_{2i-1}z) = x_{2i}y$, $\varphi(x_{2i}z) = x_{2i}y$ and $\varphi(x_{2i}y) = \varphi(x_{2i+1}z)$, with indices taken modulo $2\ell$. Therefore, $C_{y,z} := (x_1y,x_2z,x_3y,x_4z,\dots,x_{2\ell-1}y,x_{2\ell}z)$ and 
    % $C_{z,y} := (x_1z,x_2y,x_3z,x_4y,\dots,x_{2\ell-1}z,x_{2\ell}y)$ are two cycles of edges under $\varphi$.  
    Recall that $\psi$ is an involution by \cref{def:exceptional}.
    Fix $y \in [n]$ such that $\varphi(y) \neq \psi(y)$ (we assume that $\varphi \neq \psi$). Put $x = \psi(y)$ and $z = \varphi(y)$, so $x \neq z$. Note that $y,z$ are not 1-cycles in $\varphi$ or $\psi$, because $\varphi,\psi$ share all 1-cycles by Item 2 of \Cref{def:exceptional}. Hence, $xy$ is a $2$-cycle of $\psi$. By Item 3 in \Cref{def:exceptional}, we have $\varphi(x) = y$; indeed, $\varphi(y) = x$ is impossible because $\varphi(y) = z \neq x$.
    Also, setting $w = \psi(z)$, we have $\psi(w) = z$ and $w \neq z$ (because $z$ is not a 1-cycle). 
    Let $u_1v_1,\dots,u_tv_t$ be the 2-cycles shared by $\varphi$ and $\psi$ (clearly $u_i,v_i \notin \{x,y,z,w\}$ for every $i$). 
    By Item 2 of \Cref{def:exceptional}, we have 
    $t = \frac{n}{2} - O(\sqrt{n})$.
    % For a 2-cycle $uv \in \mathcal{C}_2$ and a graph $G \in X$, we say that $G$ is {\em good with respect to $uv$} if
    For a graph $G \in \mathcal{G}$, let $I(G)$ be the set of all indices $1 \leq i \leq t$ such that $v_iy,u_iz \in G$. 
    Let $\mathcal{G}_{0}$ be the set of graphs $G \in \mathcal{G}$ such that $|I(G)| \geq \frac{n}{10}$.  
    To prove the lemma, we will show that $\mathcal{G} \setminus \mathcal{G}_0$ is small and that $N_{\varphi}(G)$ is small for every $G \in \mathcal{G}_0$. 
    % Then, using these two facts, we will derive that there are ``not many'' pairs of graphs $(G_1,G_2)$ with $G_1\in X$ and $G_1 \overset{\varphi}{\rightarrow} G_2$.
    \begin{claim}\label{claim: X-X0}
        $|\mathcal{G} \setminus \mathcal{G}_0| \leq 2^{f_2(n)} \cdot 2^{-\Omega(n)}$.
    \end{claim}
    \begin{proof}
        By \Cref{prop: involution relation}, the graphs in $\mathcal{G}$ agree on all edges in $\mathcal{C}_1(\psi)$, and for every pair $(e,f) \in \mathcal{C}_2(\psi)$, there is $a_{e,f} \in \{0,1,2\}$ such that every graph in $\mathcal{G}$ contains exactly $a_{e,f}$ of the edges $e,f$. 
        Hence, letting $S$ be the set of pairs $e,f$ with $a_{e,f} = 1$, we see that a graph $G$ in $\mathcal{G}$ is determined by deciding if $G$ contains $e$ or $f$ for every pair $(e,f) \in S$. 
        In particular, $|\mathcal{G}| \leq 2^{|S|}$. 
        Since $u_iv_i \in \mathcal{C}_1(\psi)$ for every $1 \leq i \leq t$, we have 
        $|S| \leq \frac{1}{2}\left( \binom{n}{2} - t \right)$.
        % $|S| \leq \frac{1}{2}\left( \binom{n}{2} - t \right) \leq \frac{1}{2}\left( \binom{n}{2} - \frac{n}{2} + O(\sqrt{n})\right)=f_2(n)+O(\sqrt{n})$.

        Let us fix a set $I \subseteq [t]$ and count all $G \in \mathcal{G} \setminus \mathcal{G}_0$ satisfying $I(G) = I$. 
        Note that there are 
        $\binom{t}{|I|}$ choices for $I$ of a given size. 
        Fix any $i \in [t]$, and note that $(u_ix,v_iy),(u_iz,v_iw)$ are both 2-cycles of $\tilde{\psi}$. 
        Thus, a priori, a graph $G \in X$ has (at most) 2 choices for each of these pairs, so 4 choices altogether.
        If $i \in I(G)$, then by the definition of $I(G)$ we have $v_iy,u_iz \in G$, so $G$ has at most one choice on the pairs $(u_ix,v_iy),(u_iz,v_iw)$. 
        Indeed, either $a_{u_ix,v_iy} = 2$, in which case there is  only one choice, or $a_{u_ix,v_iy} = 1$ and the choice $u_ix \notin G, v_iy \in G$ is forced; and the same applies to $(u_iz,v_iw)$.
        % the choice $ux \notin G$, $vy \in G$, $uz \in G$, $vw \notin G$ is forced. 
        On the other hand, if $i \notin I(G)$ then $G$ has at most 3 choices on $(u_ix,v_iy),(u_iz,v_iw)$, because either one of these two pairs is not in $S$, in which case $G$ has at most 2 choices on these pairs, or both are in $S$ and the choice  $u_ix \notin G$, $v_iy \in G$, $u_iz \in G$, $v_iw \notin G$ is not allowed (because $i \notin I(G)$). 
        It follows that when restricted to the pairs in the set $P := \{(u_ix,v_iy),(u_iz,v_iw) : i = 1,\dots,t\}$, the number of choices for $G \in \mathcal{G}$ with $I(G) = I$ is $3^{t-|I|}$. 
        Recall that if $G \in \mathcal{G} \setminus \mathcal{G}_0$ then $|I(G)| \leq \frac{n}{10}$. 
        Hence, when restricted to the pairs in $P$, the number of choices for $G \in \mathcal{G} \setminus \mathcal{G}_0$ is $\sum_{k \leq n/10} \binom{t}{k} 3^{t - k} = 4^{t} \cdot \mathbb{P}\left[ \Bin(t, \frac{1}{4}) \leq \frac{n}{10} \right] \leq 
        4^{t - \Omega(n)}$, using the Chernoff bound and the fact that $t = (\frac{1}{2} - o(1))n$. 
        Since all edges in the pairs of $P$ and those of form $u_iv_i$, $1\le i\le t$, are not contained in the pairs of $S\setminus P$, it holds that $2\abs{S\setminus P} \le \binom{n}{2}-t-4t$, meaning $\abs{S\setminus P} \le \frac{1}{2}\binom{n}{2}-\frac{5t}{2}$.
        Now, using that each $G \in X$ has at most $2$ choices for each remaining pair $(e,f) \in S \setminus P$, we get that
        $|\mathcal{G} \setminus \mathcal{G}_0| \leq 2^{\frac{1}{2}\binom{n}{2}-\frac{5t}{2}} \cdot 4^{t - \Omega(n)} = 2^{|S| - \Omega(n)} \leq
        2^{\frac{1}{2}(\binom{n}{2} - t)} \cdot 2^{-\Omega(n)} = 2^{f_2(n)} \cdot 2^{-\Omega(n)}$,
        where the last equality uses $f_2(n) = \frac{1}{2}\left(\binom{n}{2} - \floor{\frac{n}{2}}\right)$ and $t = \frac{n}{2} - o(n)$.
        % $|\mathcal{G} \setminus \mathcal{G}_0| \leq 2^{|S| - |P|} \cdot 4^{t - \Omega(n)} = 2^{|S| - \Omega(n)} \leq
        % 2^{\frac{1}{2}(\binom{n}{2} - t)} \cdot 2^{-\Omega(n)} = 2^{f_2(n)} \cdot 2^{-\Omega(n)}$, where the first equality uses $|P| = 2t$ and the last equality uses $f_2(n) = \frac{1}{2}\left(\binom{n}{2} - \floor{\frac{n}{2}}\right)$ and $t = \frac{n}{2} - o(n)$.
    \end{proof}
    Next, we bound $|N_{\varphi}(G)|$ for every $G \in \mathcal{G}$. We begin with a simple bound that holds for every $G \in \mathcal{G}$, and then prove a stronger bound for $G \in \mathcal{G}_0$. 
    \begin{claim}\label{claim:degree bound on X}
        For every $G \in \mathcal{G}$, it holds that $|N_{\varphi}(G)| \leq 2^{f_2(n)} \cdot 2^{O(\sqrt{n})}$.
    \end{claim}
    \begin{proof}
        Let $m$ be the number of choosable pairs of $(G,\varphi)$.
        As $u_iv_i \in \mc{C}_1(\varphi)$ ($\varphi(u_iv_i)=u_iv_i$) for all $1 \le i \le t$, all these edges do not belong to any choosaable pairs.
        % By \cref{def: choosable pairs}, if $\varphi(e)=e$, then $e$ cannot belong to any choosable pair.
         % Recall that if $(e,f)$ is a choosable pair then $\varphi(e) \neq e$ and $\varphi(f) \neq f$ (see \cref{def: choosable pairs}). Hence, if $\varphi(e) = e$ then $\varphi$ cannot belong to any choosable pair. 
        % In particular, as $\varphi(u_iv_i) = u_iv_i$, the edge $u_iv_i$ does not belong to any choosable pair (for every $i=1,\dots,t$). 
        Hence,  $m \leq \frac{1}{2}\left( \binom{n}{2} - t \right) = \frac{1}{2}\left( \binom{n}{2} - \frac{n}{2} + O(\sqrt{n}) \right) = f_2(n) + O(\sqrt{n})$. 
        Now, \Cref{lemma: properties of choosable pairs} implies
        $|N_{\varphi}(G)| \leq 2^m \leq 2^{f_2(n)} \cdot 2^{O(\sqrt{n})}$. 
    \end{proof}
    \begin{claim}\label{claim:degree bound on X0}
        For every $G \in \mathcal{G}_0$, it holds that $|N_{\varphi}(G)| \leq 2^{f_2(n)} \cdot 2^{-\Omega(n)}$. 
    \end{claim}
    \begin{proof}
        Again, let $m$ be the number of choosable pairs of $(G,\varphi)$. 
        % As in the previous proof, we know that the edges $xy$ satisfying $\varphi(xy) = xy$, of which there are at least $\frac{n}{2} - O(\sqrt{n})$, do not belong to any choosable pair. 
        As in the previous proof, $u_iv_i$ does not belong to any choosable pair (for every $i=1,\dots,t$). 
        To improve further, we show that for every $i \in I(G)$, the edge $v_iy$ does not belong to any choosable pair either. 
        Indeed, recall that a choosable pair $(e,f)$ satisfies $\varphi(e) = f$, $e \in G$, $f \notin G$. We have $\varphi(u_ix) = v_iy$ and $\varphi(v_iy) = u_iz$. 
        Hence, if $v_iy$ belongs to a choosable pair $(e,f)$, then either $e=u_ix,f=v_iy$ or $e=v_iy,f=u_iz$. But $v_iy,u_iz \in G$ (because $i \in I(G)$), and this rules out both options. Using that $|I(G)| \geq \frac{n}{10}$ (as $G \in \mathcal{G}_0$), we get
        $m \leq \frac{1}{2}\left( \binom{n}{2} - t - |I(G)| \right) \leq \frac{1}{2}\left( \binom{n}{2} - t - \frac{n}{10} \right) = 
        \frac{1}{2}\left( \binom{n}{2} - \frac{n}{2} \right) - \Omega(n) = f_2(n) - \Omega(n)$, where we used that $f_2(n) = \frac{1}{2}\left(\binom{n}{2} - \floor{\frac{n}{2}}\right)$ and $t = \frac{n}{2} - o(n)$.
        Now the claim follows from $|N_{\varphi}(G)| \leq 2^m$ (see \Cref{lemma: properties of choosable pairs}). 
    \end{proof}
    We are now ready to complete the proof of \Cref{lem: leftover of psi clique is small}. 
    The number of pairs of graphs $(G_1,G_2)$ with $G_1 \in \mathcal{G}$ and $G_1 \overset{\varphi}{\rightarrow} G_2$ is 
    $
    \sum_{G_1 \in \mathcal{G}}|N_{\varphi}(G_1)| = 
    \sum_{G_1 \in \mathcal{G} \setminus \mathcal{G}_0}|N_{\varphi}(G_1)| + 
    \sum_{G_1 \in \mathcal{G}_0}|N_{\varphi}(G_1)|.$ 
    For the first sum, using \Cref{claim: X-X0} and 
    \Cref{claim:degree bound on X}, it holds that 
    $$
    \sum_{G_1 \in \mathcal{G} \setminus \mathcal{G}_0}|N_{\varphi}(G_1)|
    \leq 
    |\mathcal{G} \setminus \mathcal{G}_0| \cdot 2^{f_2(n)} \cdot 2^{O(\sqrt{n})} \leq 2^{2f_2(n)} \cdot 2^{-\Omega(n)}.
    $$
    % using \Cref{claim: X-X0} and    \Cref{claim:degree bound on X}. 
    And for the second sum, using \Cref{claim:degree bound on X0} and $|\mathcal{G}| \leq 2^{f_2(n)}$ (see Item 2 of \cref{prop:extremal construction}), we get
    $$
    \sum_{G_1 \in \mathcal{G}_0}|N_{\varphi}(G_1)| \leq |\mathcal{G}| \cdot 2^{f_2(n)} \cdot 2^{-\Omega(n)}.
    $$
    % using \Cref{claim:degree bound on X0} and that $|\mathcal{G}| \leq 2^{f_2(n)}$ (see Item 2 of \cref{prop:extremal construction}).
    This completes the proof of the lemma.
\end{proof}
Now, we will show that a large involution clique must span almost the whole difference-isomorphic family by simply counting over all $\varphi\in S_n$.
\begin{lemma} \label{lemma: any non-small involution clique is large}
    Let $\cal G$ be a difference-isomorphic family on $[n]$, and let $\cal G' \subseteq \cal G$ be a maximal involution clique in $\mathcal{G}$.
    Then $\abs{\cal G'} < 2^{f_2(n)}\cdot e^{-cn}$ or
    $\abs{{\cal G}\setminus{\cal G'}} < 2^{f_2(n)}\cdot e^{-cn}$ for some absolute constant $c > 0$.
\end{lemma}
\begin{proof}
    Say that $\cal G'$ is a $\psi$-clique for some involution $\psi$.
    % Let $\psi$ be the involution corresponding to $\mathcal{G}'$, i.e., $\mathcal{G}'$ is a $\psi$-clique. 
    For $\varphi \in S_n$, let $\mathcal{E}_{\varphi}$ be the set of pairs $(G_1,G_2) \in \mathcal{G}' \times (\mathcal{G} \setminus \mathcal{G}')$ such that $G_2 \overset{\varphi}{\rightarrow} G_1$. Since $\mathcal{G}$ is difference-isomorphic, we have $|\mathcal{G}'||\mathcal{G} \setminus \mathcal{G}'| \leq \sum_{\varphi \in S_n} |\mathcal{E}_{\varphi}|$. 
    Note that $E_{\psi} = \emptyset$, because if $G_2 \overset{\psi}{\rightarrow} G_1$ for some $G_1 \in \mathcal{G}', G_2 \in \mathcal{G} \setminus \mathcal{G}'$, then $\mathcal{G'} \cup \{G_2\}$ is a larger $\psi$-clique (see \cref{prop: involution relation} and the following paragraph), contradicting the maximality of $\mathcal{G}'$. 
    As in the proof for $r\ge 3$, we partition $S_n$ into several classes and bound the contribution of each class. Setting $\delta = 2n\ln{n}$ and $A = 36\sqrt{\delta}$, 
    define:
    \begin{equation} \nonumber
        \begin{aligned}
            P_1 &:=\left\{\varphi\in S_n: \varphi\text{ and }\psi\text{ share fewer than } \left(\frac{n}{2}-A\right) \text{ 2-cycles}\right\}, \\
            P_2 &:=\left\{\varphi\in S_n\setminus P_1: (\varphi,\psi) \text{ is not exceptional}\right\}, \quad
            P_3 :=S_n \setminus (P_1\cup P_2 \cup \{\psi\}).
        \end{aligned}
    \end{equation}
    Note that as $N_{\varphi}(G_2) \cap \mathcal{G}' \subseteq \mathcal{G}'$ is a $\psi$-clique, we have 
    $|N_{\varphi}(G_2) \cap \mathcal{G}'|^2 \leq e_{\psi}(N_{\varphi}(G_2))$.
    
    First we consider $P_1$. Fix $\varphi \in P_1$ and $G_2 \in \mathcal{G} \setminus \mathcal{G}'$.  By \cref{lemma: counts of normal phi r-graph} and by our choice of $A$, we have $e_{\psi}(N_{\varphi}(G_2)) \leq 2^{2f_2(n)} \cdot e^{-\delta}$, so 
    $|N_{\varphi}(G_2) \cap \mathcal{G}'| \leq 2^{f_2(n)} \cdot e^{-\delta/2} = 2^{f_2(n)} \cdot n^{-n}$. Summing over all $G_2 \in \mathcal{G}\setminus \mathcal{G}'$ and all (at most $n!$) permutations $\varphi \in P_1$, we get 
    \begin{equation}\label{eq:P1}
    \sum_{\varphi \in P_1} |\mathcal{E}_{\varphi}| \leq 
    \sum_{\varphi \in P_1} \sum_{G_2 \in \mathcal{G}\setminus \mathcal{G}'} |N_{\varphi}(G_2) \cap \mathcal{G}'| \leq 
    n! \cdot |\mathcal{G} \setminus \mathcal{G}'| \cdot 2^{f_2(n)} \cdot n^{-n} \leq 
    |\mathcal{G} \setminus \mathcal{G}'| \cdot 2^{f_2(n)} \cdot e^{-\Omega(n)}.
    \end{equation}
    
    Next, by \cref{lem: number of close permutations} with $\varphi_0 = \psi$ and $A$, we have $|P_2 \cup P_3| \leq n^{2A} \leq e^{o(n)}$. Also, for each $\varphi \in P_2$ we have $e_{\psi}(N_{\varphi}(G_2)) \leq 2^{2f_2(n)} \cdot e^{-\Omega(n)}$ by \cref{lemma: counts for exceptions}, because $(\varphi,\psi)$ is not exceptional. Therefore, $|N_{\varphi}(G_2) \cap \mathcal{G}'| \leq 2^{f_2(n)} e^{-\Omega(n)}$. Summing over all $G_2 \in \mathcal{G} \setminus \mathcal{G}'$, we get $|\mathcal{E}_{\varphi}| \leq |\mathcal{G} \setminus \mathcal{G}'| \cdot 2^{f_2(n)} \cdot e^{-\Omega(n)}$. Moreover, for each $\varphi \in P_3$, we have that $|\mathcal{E}_{\varphi}| \leq 2^{2f_2(n)} \cdot 2^{-\Omega(n)}$ by \cref{lem: leftover of psi clique is small}. Summing over all $e^{o(n)}$ permutations $\varphi \in P_2 \cup P_3$, we get
    \begin{equation}\label{eq: P2,P3}
     \sum_{\varphi \in P_2 \cup P_3} |\mathcal{E}_{\varphi}| \leq 
     |\mathcal{G} \setminus \mathcal{G}'| \cdot 2^{f_2(n)} \cdot e^{-\Omega(n)} + 2^{2f_2(n)} \cdot 2^{-\Omega(n)}.  
    \end{equation}
    By combining \eqref{eq:P1} and \eqref{eq: P2,P3}, we get that 
    $$
    |\mathcal{G}'| \cdot |\mathcal{G} \setminus \mathcal{G}'| \leq 
    \sum_{\varphi \in S_n} |\mathcal{E}_{\varphi}| \leq  
    |\mathcal{G} \setminus \mathcal{G}'| \cdot 2^{f_2(n)} \cdot e^{-\Omega(n)} + 2^{2f_2(n)} \cdot 2^{-\Omega(n)}.
    $$
    Therefore, there exists some absolute constant $c > 0$ such that either
    $|\mathcal{G}'| < 2^{f_2(n)} \cdot e^{-cn}$ or 
    $|\mathcal{G}'| \geq 2^{f_2(n)} \cdot e^{-cn}$ and
    $|\mathcal{G}'| \cdot |\mathcal{G} \setminus \mathcal{G}'| < 2^{2f_2(n)} \cdot e^{-2cn}$. In both cases, the assertion of the lemma holds. 
\end{proof}

% Let $\cal G$ be a difference-isomorphic family of $r$-graphs on $n$ vertices.
% For a graph $G \in \cal G$ and a permutation $\varphi\in S_n$, define $N_{\varphi}^{\cal G}(G):=N_{\varphi}(G)\cap \cal{G}$ \lg{Do we need this notation?}.
% Note that trivially $G \in \neighbor{G}{\varphi}$ for all $\varphi\in S_n$.

The following lemma is the last step towards \cref{theorem: stablility}. It shows that if a difference-isomorphic graph family contains an involution clique of size almost $2^{f_2(n)}$, then the family itself is an involution clique.

\begin{lemma}\label{lemma: no very large involution clique}
    There exists a constant $c>0$ such that the following holds for sufficiently large $n$.
    Let $\cal G$ be a difference-isomorphic family of graphs on $n$ vertices.
    Suppose $\cal G$ contains a maximal $\psi$-clique $\cal G'$ of size at least $\left(1-n^{-c\sqrt{n}}\right)2^{f_2(n)}$ for some involution $\psi \in S_n$.
    Then, $\cal G = \mathcal{G}'$.
\end{lemma}
\begin{proof}
    Let $c'$ be the constant given by \cref{lemma: general counts for exceptional r-graph}. We will show that $c=500c'$ suffices.
    % We may assume that $\cal G'$ is a maximal $\psi$-clique in $\cal G$.
    Suppose for contradiction that there exists $G_0 \in \cal G \setminus \cal G'$.
    % As in \cref{prop: involution relation}, let $\mathcal{C}_1$ be the set of $e \in \binom{[n]}{2}$ with $\psi(e) = e$, and let $\mathcal{C}_2$ be the set of pairs $(e,f)$, $e,f \in \binom{[n]}{2}$, $e \neq f$, with $\psi(e) = f$ and $\psi(f) = e$.
    Recall the definition of $\mathcal{C}_1 := \mathcal{C}_1(\psi)$ and $\mathcal{C}_2 := \mathcal{C}_2(\psi)$ (see \cref{def: C1 C2}). 
    According to \cref{prop: involution relation}, 
    all graphs in $\mathcal{G}'$ agree on the edges in $\mathcal{C}_1$, and for every $(e,f) \in \mathcal{C}_2$ there is $a_{e,f} \in \{0,1,2\}$ such that all graphs in $\mathcal{G}'$ contain exactly $a_{e,f}$ of the edges $e,f$. 
    Hence, letting 
    $S = \{\{e,f\} : a_{e,f} = 1\}\subseteq \mc{C}_2$, we have
    $|\mathcal{G}'| \leq 2^{|S|}$. 
    Recall that \cref{claim: fixed r-sets} implies $|S|\le |\mc{C}_2|\le f_2(n)$.
    It must hold that $|S| = |\mc{C}_2|=f_2(n)$, as otherwise,
    $\abs{\cal G'}\le 2^{f_2(n)-1}$, contradicting our assumption on the size of $\mathcal{G}'$ (when $n$ is sufficiently large). 
    Namely, for every $(e,f) \in \mathcal{C}_2$, 
    each graph in $\mathcal{G}'$ contains exactly one of $e,f$. 
    % for every two cycles $e,f$ of $\tilde{\psi}$, all the graphs in $\cal G'$ must contain exactly one of them (otherwise $\abs{\cal G'}\le f_2(n)/2$).
    
    Since $G_0 \in \cal G \setminus \cal G'$, the relation $G_0 \overset{\psi}{\rightarrow} G$ does {\bf not} hold for any $G \in \mathcal{G}'$, because otherwise $\mathcal{G}' \cup \{G_0\}$ is a $\psi$-clique by \cref{prop: involution relation}, in contradiction to the maximality of $\mathcal{G}'$.
    Hence, by \cref{prop: involution relation}, $G_0$ disagrees with the graphs in $\mathcal{G}'$ on one of the edges in $\mathcal{C}_1$, or there is $(e,f) \in \mathcal{C}_2$ such that $G_0$ contains both $e,f$ or none of them. 
    % there $\vedge{G}{e}\neq \vedge{G'}{e}$ for some fixed point $e \in \binom{[n]}{2}$ of $\tilde{\psi}$ and all $G'\in\mc{G}'$, or that $\vedge{G}{e}+\vedge{G}{f}\neq 1$ for some pair of distinct edges $e,f$ with $\tilde{\psi}(e)=f$ and $\tilde{\psi}(f)=e$.
    To reduce the number of possibilities to consider, we use \cref{claim:complement}, which states that $\{G^c : G \in \mathcal{G}\}$ is also difference-isomorphic. 
    So by taking the complements of the graphs in $\cal G$ (if necessary), we may assume 
    %as $G_1 \overset{\varphi}{\rightarrow} G_2$ if and only if $G_2 \overset{\varphi^{-1}}{\rightarrow} G_1$.
    without loss of generality that one of the following holds:
    \begin{itemize}
        \item[(a)] There is $e \in \mathcal{C}_1$ such that $e \in G_0$ and $e \notin G$ for all $G \in \cal G'$;
        \item[(b)] $G_0$ agrees with all $G\in \mathcal{G}$ on all edges in $\mathcal{C}_1$, and there is a pair $(e,f) \in \mathcal{C}_2$ such that $e,f \in G_0$. 
    \end{itemize}

    % To get a contradiction, our plan is to show that 
    % $\bigcup_{\varphi\in S_n}(N_{\varphi}(G_0) \cap \cal G') \neq \mc{G}'$. This is indeed a contradiction because every $G \in \mathcal{G}'$ belongs $N_{\varphi}(G_0)$ for some $\varphi \in S_n$, as $\mathcal{G}$ is difference-isomorphic.
    Since $\mathcal{G}$ is difference isomorphic, we have $\mathcal{G}'\subseteq \bigcup_{\varphi \in S_n}N_{\varphi}(G_0) \cap \cal G'$. 
    We will obtain a contradiction by upper-bounding $\abs{\bigcup_{\varphi \in S_n}N_{\varphi}(G_0) \cap \cal G'}$ (recall that $|\mc{G}'|\ge (1-n^{-c\sqrt{n}})2^{f_2(n)}$).
    % to this inequality by bounding the sets $|N_{\varphi}(G_0) \cap \cal G'|$.
    As explained above, we have $N_{\psi}(G_0) \cap \cal G' = \emptyset$. 
    % because otherwise ${\cal G}' \cup \{G_0\}$ is a $\psi$-clique (by \cref{claim:involution clique}), contradicting the maximality of $\mathcal{G}'$.
    Setting $A:=36\sqrt{2n\ln n}$, partition $S_n\setminus\{\psi\}$ into three sets $P_1\cup P_2\cup P_3$, as follows:
    \begin{equation}\nonumber
        \begin{aligned}
            P_1 &=\{\varphi \in S_n\setminus\{\psi\}: \varphi\text{ and }\psi\text{ share fewer than } \left(\frac{n}{2}-A\right)\text {2-cycles}\}, \\
            P_2 &= \{ \varphi\in S_n\setminus P_1: e_\psi(N_{\varphi}(G_0))\le 2^{2f_2(n)}\cdot e^{-n/100} \}, \quad 
            P_3 = S_n\setminus \left(P_1 \cup P_2 \cup \{\psi\} \right).
        \end{aligned}
    \end{equation}

    First, for the contribution of permutations $\varphi \in P_1 \cup P_2$, we will use the fact that 
    $\abs{N_{\varphi}(G_0)\cap \mc{G}'}^2 \leq e_{\psi}(N_{\varphi}(G_0))$, which holds because $\neighbor{G}{\varphi}\cap \mc{G}'\subseteq\mc{G}'$ is a $\psi$-clique. 
    For $\varphi \in P_1$, \cref{lemma: counts of normal phi r-graph} with $\delta=2n\ln n$ gives $e_\psi(N_{\varphi}(G_0)) \le 2^{2 f_2(n)}\cdot n^{-2n}$, and so $\abs{N_{\varphi}(G_0)\cap \mc{G}'} \leq  2^{f_2(n)}\cdot n^{-n}$. Therefore,
    % $e_{\psi}(N_\varphi(G)\cap G')=\abs{N_\varphi(G)\cap G'}^2$.
    % So, $\abs{N_\varphi(G)\cap G'} \le f_2(n)\cdot n^{-2n}$, and then 
    \begin{equation} \label{eq: contribution from normal phis}
        \sum_{\varphi\in P_1}\abs{N_{\varphi}(G_0)\cap \cal G'} 
        \le n! \cdot 2^{f_2(n)}\cdot n^{-n}
        \le 2^{f_2(n)}\cdot e^{-\Omega(n)}.
    \end{equation}
    For $\varphi \in P_2$, we have $e_\psi(N_{\varphi}(G_0)) \le 2^{2 f_2(n)}\cdot e^{-n/100}$ by the definition of $P_2$.
    Therefore, $\abs{N_{\varphi}(G_0)\cap \mathcal{G}'} \le 2^{f_2(n)}\cdot e^{-n/200}$. 
    Note that $|P_2|\le n^{2A}=2^{o(n)}$ by \cref{lem: number of close permutations} with $\varphi_0=\psi$.
    % Let us now bound the size of $P_2$. 
    % By definition, every $\varphi \in P_2$ shares at least $\left(\frac{n}{2}-A\right)$ $2$-cycles with $\psi$. 
    % To determine $\varphi$, it suffices to choose the $\left(\frac{n}{2}-A\right)$ common 2-cycles of $\varphi$ and $\psi$ and then to specify a permutation on the remaining $n-2A$ vertices. Thus, $\abs{P_2}\le\binom{n/2}{n/2-A} \cdot (2A)!=2^{o(n)}$.
    % Now, observe that the number of $\varphi \in S_n$ sharing at least $\left(\frac{n}{2}-A\right)$ 2-cycles is at most $\binom{n/2}{n/2-140\sqrt{n\ln n}}(280\sqrt{n\ln n})!=n^{O(\sqrt{n\ln n})}$ because we can first choose the common 2-cycles and then the permutation on the rest vertices.
    % In other words, $\abs{P_2}=n^{O(\sqrt{n\ln n})}$.
    Therefore,
    \begin{equation} \nonumber
        \sum_{\varphi\in P_2}\abs{N_{\varphi}(G_0)\cap \cal G'} 
        \le 2^{o(n)} \cdot 2^{f_2(n)}\cdot e^{-n/200}
        \le 2^{f_2(n)}\cdot e^{-\Omega(n)}.
    \end{equation}
    By combining this with \cref{eq: contribution from normal phis}, when $n$ is sufficiently large, we get
    \begin{equation} \label{eq: P3 contribution}
        \begin{aligned}
        \abs{\bigcup_{\varphi\in P_3}\left(N_{\varphi}(G_0)\cap \cal G'\right)}
        &\ge \abs{\cal G'} - \sum_{\varphi\in P_1\cup P_2}\abs{N_{\varphi}(G_0)\cap \cal G'} \\
        &\ge \left(1-n^{-c\sqrt{n}}\right)2^{f_2(n)} - 2^{f_2(n)}\cdot e^{-\Omega(n)} 
        > \left(1-n^{-c\sqrt{n}/2}\right)2^{f_2(n)}.
        \end{aligned}
    \end{equation}
    For the rest of the proof, we will get a contradiction to \eqref{eq: P3 contribution} by upper bounding $\bigcup_{\varphi\in P_3}\left(N_{\varphi}(G_0)\cap \cal G'\right)$.
    However, we can no longer achieve this by giving a uniform upper bound of $\abs{N_\varphi(G_0)\cap \mc{G}'}$ for all $\varphi \in P_3$, because in the example of \cref{prop: tightness of stability}, it holds that $\abs{N_\varphi(G_0)\cap \mc{G}'}=\Omega(2^{f_2(n)})$.
    Instead, we will classify a superset $\mc{F}$ of $\bigcup_{\varphi\in P_3} (N_\varphi(G_0)\cap \mc{G}')$ which has size at most $\left(1-n^{-c\sqrt{n}/2}\right)2^{f_2(n)}$. 
    
    By \cref{lemma: counts for exceptions}, if $\varphi\in P_3$ then $(\varphi,\psi)$ is exceptional, and for every 2-cycle $xy$ of $\psi$, if $\varphi(x)=y$ and $\varphi(y)\neq x$, then $d_{G_0}(x)>d_{G_0}(y)$.
    In particular, $\varphi$ shares at least $\left(\frac{n}{2}-c'\sqrt{n}\right)$ 2-cycles with $\psi$ (see \cref{def:exceptional}).
    Hence, $\abs{P_3}\le n^{2c'\sqrt{n}}$ by \cref{lem: number of close permutations} with $\varphi_0=\psi$ and $A=c'\sqrt{n}$.
    % A similar computation shows that $\abs{P_3}=n^{O(\sqrt{n})}$.
    
    Let $u_1v_1,\dots,u_{\ell}v_{\ell}$ be all the 2-cycles of $\psi$. Then $u_iv_i \in \mathcal{C}_1$ for every $i \in [\ell]$. 
    Write $U:=\{u_1,\dots,u_{\ell}\}$, $V:=\{v_1,\dots,v_{\ell}\}$, and note that every point in $[n] \setminus (U \cup V)$ is a fixed point of $\psi$. 
    Without loss of generality, assume $d_{G_0}(u_i)\ge d_{G_0}(v_i)$ for all $1 \le i \le \ell$.
    As $(\varphi,\psi)$ is exceptional, we know that $\varphi$ and $\psi$ share all 1-cycles. Observe also that $\varphi(u_i)=v_i$ for all $i\in\{1,\dots,\ell\}$, because by Item 3 in \cref{def:exceptional} we have $\varphi(u_i) = v_i$ or $\varphi(v_i) = u_i$, and it is impossible to have $\varphi(v_i) = u_i$ and $\varphi(u_i) \neq v_i$, since $d_{G_0}(u_i)\ge d_{G_0}(v_i)$.
    Define the following set, which depends only on the graph $G_0$ but not on any particular $\varphi\in P_3$, by
    $$I:=\left\{(i,j)\in\binom{[\ell]}{2}: u_iv_j,v_iu_j \in G_0 \text{ or } u_iv_j,v_iu_j \notin G_0 \right\}.$$
    \begin{claim}\label{claim: I is small}
        $\abs{I} < \frac{c}{2}\sqrt{n}\ln n$.
    \end{claim}
    \begin{proof}
        Fix any $\varphi \in P_3$.
        Let $m$ be the number of choosable pairs of $(G_0,\varphi)$ and $m_g$ be the number of good choosable pairs of $(G_0,\varphi,\psi)$.
        By Item 2 in \cref{def:exceptional}, there are at least $\left(\frac{n}{2}-c'\sqrt{n}\right)$ indices $i\in [\ell]$ such that $u_iv_i$ is a 2-cycle of $\varphi$.
        Note that for such $i$, we have $\varphi(u_iv_i) = u_iv_i$, and thus $u_iv_i$ is not contained in any choosable pair of $(G_0,\varphi)$ (see \cref{def: choosable pairs}).
        Hence, $2m \le \binom{n}{2}-\frac{n}{2}+c'\sqrt{n}$.
        % In addition, for every $(i,j)\in I$, $u_iv_j$ and $v_iu_j$ form a 2-cycle of $\tilde{\psi}$.
        In addition, for every $1 \leq i < j \leq \ell$, we have $\psi(u_iv_j) = v_iu_j$ and $\psi(v_iu_j) = u_iv_j$. Therefore, if one of the edges $u_iv_j, v_iu_j$ belongs to some good choosable pair of $(G_0,\varphi,\psi)$, then this pair must be $\{u_iv_j,v_iu_j\}$ (see \cref{def: good choosable pair}).
        Now, if $(i,j) \in I$, then $G_0$ contains both or none of the edges $u_iv_j,v_iu_j$, meaning that $\{u_iv_j,v_iu_j\}$ cannot be a choosable pair (see \cref{def: choosable pairs}).
        Hence, for every $(i,j)\in I$, the edges $u_iv_j$ and $v_iu_j$ are not contained in any good choosable pair.
        Therefore, $2m_g \le \binom{n}{2}-\frac{n}{2}+c'\sqrt{n}-2\abs{I}$.
        By \cref{eq: practical m m_g bound} with $x=(\frac{n}{2}-c'\sqrt{n})/2$ and $y=(\frac{n}{2}-c'\sqrt{n})/2+\abs{I}$, it holds that 
        \begin{equation} \nonumber
            e_\psi(N_\varphi(G)) 
            \le 2^{\binom{n}{2}-\frac{n}{2}+c'\sqrt{n}}\cdot e^{-\abs{I}/40}
            \le 2^{2f_2(n)+c'\sqrt{n}}\cdot e^{-\abs{I}/40},
        \end{equation}
        where we used $2f_2(n)\ge\binom{n}{2}-\frac{n}{2}$ (see \cref{eq: f_r(n)}).
        % Similar to the computation in \cref{eq: edge count is small},
        % \begin{equation} \nonumber
        %     e_\psi(N_\varphi(G)) \le 4^{\frac{1}{2}\binom{n}{2}-\frac{n}{4}+25\sqrt{n}-\abs{I}}\cdot 3.9^{\abs{I}}
        %     \le f_2(n)^2 \cdot 2^{50\sqrt{n}}\cdot\left(1-\frac{1}{40}\right)^{\abs{I}}
        %     \le f_2(n)^2 \cdot 2^{50\sqrt{n}}\cdot e^{-\abs{I}/40}.
        % \end{equation}
        Therefore, $\abs{N_{\varphi}(G_0)\cap \cal G'} \leq 2^{f_2(n) + c'\sqrt{n}/2} \cdot e^{-|I|/80}$. 
        Now, if $\abs{I} \geq \frac{c}{2}\sqrt{n}\ln n=250c'\sqrt{n}\ln n$, then, using $|P_3| \leq n^{2 c'\sqrt{n}}$, we get:
        % \begin{equation} \nonumber
        %     \abs{\bigcup_{\varphi\in P_3}\left(N_{\varphi}(G_0)\cap \cal G'\right)}
        %     \le \sum_{\varphi\in P_3}\abs{N_{\varphi}(G_0)\cap \cal G'} 
        %     \le n^{2c'\sqrt{n}} \sqrt{2^{2f_2(n)} \cdot 2^{c'\sqrt{n}}\cdot e^{-\abs{I}/40}}
        %     < 2^{f_2(n)}/2,
        % \end{equation}
        \begin{equation} \nonumber
            \abs{\bigcup_{\varphi\in P_3}\left(N_{\varphi}(G_0)\cap \cal G'\right)}
            \le \sum_{\varphi\in P_3}\abs{N_{\varphi}(G_0)\cap \cal G'} 
            \le n^{2c'\sqrt{n}} \cdot 
            2^{f_2(n) + c'\sqrt{n}/2} \cdot e^{-|I|/80}
            < 2^{f_2(n)}/2,
        \end{equation}
        when $n$ is sufficiently large.
        This contradicts \cref{eq: P3 contribution}, so $\abs{I} < 250c'\sqrt{n}\ln n$.
    \end{proof}
    % By \cref{lemma: G almost always takes one}, we have $e_\psi(N_\varphi(G)) \le f_2(n)^2\cdot e^{-\abs{I}/40}$ for all $\varphi \in S_n$.
    % This means $\abs{\neighbor{\varphi}{G}\cap \cal G'} \le f_2(n)\cdot e^{-\abs{I}/80}$, and thus
    % \begin{equation} \nonumber
    %     \left(1-o(1)\right)f_2(n) 
    %     \le \sum_{\varphi\in P_3}\abs{\neighbor{G}{\varphi}\cap \cal G'} 
    %     \le n^{O(\sqrt{n})} f_2(n)\cdot e^{-\abs{I}/80},
    % \end{equation}
    % implying $\abs{I}=O(\sqrt{n}\ln n)$.
    \noindent
    As explained in the beginning, we may assume that (a) or (b) above holds. We now consider these two cases. 
    \paragraph{Case 1: (a) holds.} Namely, there exists $e \in \mathcal{C}_1$ with $e \in G_0$ and $e \notin G$ for all $G \in \mc{G}'$.
        This means $\psi(e) = e$, so either $e$ connects two fixed points of $\psi$, or $e = u_iv_i$ for some $1 \leq i \leq \ell$. 
        We first rule out the former case. 
        Indeed, fix any $\varphi \in P_3$ with $N_{\varphi}(G_0) \cap \mathcal{G}' \neq \emptyset$; such a $\varphi$ exists by \cref{eq: P3 contribution}. 
        By \cref{def:exceptional}, $\varphi$ and $\psi$ share all fixed points. 
        Thus, if $e$ connects two fixed points of $\psi$, then $\varphi(e) = e$. However, taking any $G \in N_{\varphi}(G_0) \cap \mathcal{G}'$, we get that $e \in G_0 \setminus G$ (by the choice of $e$) and hence $e = \varphi(e) \in G \setminus G_0$ (as $G_0 \overset{\varphi}{\rightarrow} G$), a contradiction. 
        Therefore, $e=u_iv_i$ for some $1\le i\le \ell$. 
        Note that $e$ is not dependent on any particular choice of $\varphi \in P_3$.
        We may, without loss of generality, assume that $e=u_1v_1$.
        Now, let 
            \begin{equation} \nonumber
                J=\{1 < i \leq \ell : u_1v_i,v_1u_i \notin G_0\}, \quad 
                {\cal F}=\{G \in {\cal G'}: \exists j \in J.\;  v_1u_j\in G \}.
            \end{equation}
            % Let $J'=\{j \neq 1: u_1v_j,v_1u_j \in G \text{ or } u_1v_j,v_1u_j \in G\}$, and consider a subfamily of graphs ${\cal F}=\{G \in {\cal G'}: \exists j \in J, v_1u_j\in G \}$.
            By \cref{claim: I is small}, $\abs{J}\le\abs{I}\le \frac{c}{2}\sqrt{n}\ln n$, since $(1,j) \in I$ for every $j \in J$.
            Recall that for every pair of edges $(e,f) \in \mathcal{C}_2$, every graph in $\mathcal{G}'$ contains exactly one of $e,f$. 
            In particular, this holds for $(e,f) = (u_1v_j,v_1u_j) \in \mathcal{C}_2$ for every $j\in J$.
            Then, when restricting to the edges in $\{u_1v_j,v_1u_j : j \in J\}$, the graphs $G \in \mathcal{F}$ have at most $2^{|J|}-1$ choices, because the choice $u_1v_j \in G$ (and hence $v_1u_j \notin G$) for every $j \in J$ is impossible. 
            Hence, using that $|\mathcal{C}_2| \leq f_2(n)$ (see \cref{claim: fixed r-sets}), we get that $\abs{\cal F}\le \left(2^{\abs{J}}-1 \right) \cdot 2^{\abs{\mc{C}_2}-\abs{J}}\le \left( 1-2^{-\abs{J}} \right) \cdot 2^{f_2(n)}$.
            
            Next, it suffices to show $\bigcup_{\varphi\in P_3}(N_{\varphi}(G_0)\cap \cal G')\subseteq {\cal F}$.  
            % the graphs in $\cal F$ have two choices (must have exactly one of them), but for the pairs $(u_1v_j,v_1u_j)_{j\in J'}$ these graphs must contain $v_1u_j$ for at least one $j\in J'$.
            Indeed, using $|J| \leq \frac{c}{2}\sqrt{n} \ln{n}$, this implies 
            \begin{equation*} \label{eq: contradiction binary strings}
                \abs{\bigcup_{\varphi\in P_3}\left(N_{\varphi}(G_0)\cap \cal G'\right)}
                \le \abs{\cal F} \le \left( 1-2^{-\abs{J}} \right) \cdot 2^{f_2(n)}
                \leq \left(1-n^{-c\sqrt{n}/2}\right)2^{f_2(n)},
                % = \left(1-n^{-c\sqrt{n}/2}\right)2^{f_2(n)},
            \end{equation*}
            contradicting \cref{eq: P3 contribution}.
            
            To show that $\bigcup_{\varphi\in P_3}(N_{\varphi}(G_0) \cap \cal G')\subseteq {\cal F}$, fix any $\varphi \in P_3$ and suppose, for the sake of contradiction, that there exists some $G \in N_{\varphi}(G_0) \cap \cal G'$ with $G \notin \mathcal{F}$. 
            Let $S$ be the cycle of $\varphi$ containing $u_1$. 
            Recall that $\varphi(u_i)=v_i \in V$ for all $i \in \{1,\dots,\ell\}$ and that every $a\in[n]$ is a fixed point of $\psi$ (and hence also of $\varphi$ as they share all fixed points) if and only if $a\notin U\cup V$.
            Then, $\varphi(U)=V$ implies $\varphi(V)=U$.
            It follows that $S$ alternates between $U$ and $V$.
            To ease notation, let us assume without loss of generality that $S=(u_1,v_1,u_2,v_2,\dots,u_k,v_k)$ for some $k$, where $\varphi(v_i)=u_{i+1}$ for $i\in\{1,\dots,k-1\}$ and $\varphi(v_k)=u_1$.           
            % For every $i\in\{1,\dots,\ell\}$ that $\varphi(v_i) \in U$ because 
            % Recall that for every $i\in\{1,\dots,\ell\}$, we have $\varphi(u_i)=v_i$, so $S$ contains $v_1$. Also, $\varphi(v_i) \in U$, because $\varphi(v_i) \notin V$ (since the preimage of every element of $V$ is in $U$) and $\varphi(v_i) \notin [n] \setminus (U \cup V)$, since every element in $[n] \setminus (U \cup V)$ is a fixed point of $\psi$ and hence also a fixed point of $\varphi$ (as $\varphi,\psi$ share all fixed points). It follows that 
            % $S$ alternates between $U$ and $V$. To ease notation, let us assume without loss of generality that $S = (u_1,v_1,u_2,v_2,\dots,u_k,v_k)$ for some $k$. 
            We will show by induction that $u_1v_i \in G_0\setminus G$ for every $i \in \{1,\dots,k\}$.
            The base case $i=1$ holds by the assumption of Case 1.
            Suppose now that $u_1v_{i-1} \in G_0 \setminus G$ for some $2\le i \leq k$.
            Then, $v_1u_i=\varphi(u_1v_{i-1})\in G\setminus G_0$, as $G_0 \overset{\varphi}{\rightarrow} G$.
            Recall that $G$ contains exactly one of the edges $u_1v_i,v_1u_i$, as $G \in \mathcal{G}'$ and $(u_1v_i,v_1u_i) \in \mathcal{C}_2$.  
            Hence, $u_1v_i \notin G$.
            % $u_1v_{i+1}=\psi(v_1u_{i+1})\notin G'$ 
            % (all graphs in $\cal G'$ contains exactly one between $\tilde{u}_1\tilde{v}_{i+1}$ and $\tilde{v_1}\tilde{u}_{i+1}$). 
            % ($v_1u_{i+1} \in G'$.
            % Then, $G'' \in \cal G'$ does have $\psi(v_1u_{i+1})=u_1v_{i+1}$ ($v_1u_{i+1}$ and $u_1v_{i+1}$ form a 2-cycle of $\tilde{\psi}$ and all graphs in $\cal G'$ contain exactly one of them).
            Now, if $u_1v_i\notin G_0$, 
            then $i \in J$, implying that $G\in \mc{F}$ because $v_1u_i \in G$. This is a contradiction, as we assumed $G \notin \mathcal{F}$.
            This means $u_1v_i$ must be in $G_0$, so $u_1v_i\in G_0\setminus G$, establishing the induction step.
            Taking $i=k$, we have $u_1v_k \in G_0\setminus G$.
            So, $u_1v_1 = \varphi(u_1v_k)\in G \setminus G_0$, which is impossible as $u_1v_1=e\in G_0$ by the assumption of Case 1.
            This proves $\bigcup_{\varphi\in P_3}(N_{\varphi}(G_0) \cap \cal G')\subseteq {\cal F}$ and completes the proof of Case 1. 

        \paragraph{Case 2: (b) holds.} Namely, $G_0$ agrees with the graphs in $\mathcal{G}$ on all edges in $\mathcal{C}_1$, and there is a pair $(e,f) \in \mathcal{C}_2$ such that $e,f \in G_0$.
        % for some pair of distinct edges $e,f$ with $\tilde{\psi}(e)=f$ and $\tilde{\psi}(f)=e$.
        Recall that pairs $(e,f) \in \mathcal{C}_2$ satisfy $\psi(e) = f$, $\psi(f) = e$ and $e \neq f$. Hence, there are three possibilities for $e,f$: 
        % $\{e,f\}=\{au_i,av_i\}$ for some fixed point $a$ of $\psi$ and some $i\in\{1,\dots,\ell\}$ (assume $e=au_i$ and $f=av_i$); $\{e,f\}=\{u_iu_j,v_iv_j\}$ for distinct $i,j$ (assume $e=u_iu_j,f=v_iv_j$); $\{e,f\}=\{u_iv_j,v_iu_j\}$ for distinct $i,j$.
        (i) $\{e,f\}=\{au_i,av_i\}$ for some fixed point $a$ of $\psi$ and some $i\in\{1,\dots,\ell\}$; (ii) $\{e,f\}=\{u_iu_j,v_iv_j\}$ for some $1 \leq i < j \leq \ell$; and (iii) $\{e,f\}=\{u_iv_j,v_iu_j\}$ for some $1 \leq i < j \leq \ell$. We now consider two subcases. 
        \paragraph{Case 2.1. $\{e,f\}$ is of type (i) or (ii).} 
        Without loss of generality, let us assume that if $\{e,f\}$ is of type (i) then $e = au_i, f = av_i$, and if $\{e,f\}$ is of type (ii) then $e = u_iu_j, f=v_iv_j$. Observe that in both cases, we have $\varphi(e) = f$ for every $\varphi \in P_3$, because $\varphi(u_i) = v_i$ and $\varphi,\psi$ share all fixed points. 
        We claim that $e \in G$ for every $\varphi \in P_3$ and $G \in N_{\varphi}(G_0)$. Indeed, if $e \notin G$ then $e \in G_0 \setminus G$ and hence $f = \varphi(e) \in G \setminus G_0$ (as $G_0 \overset{\varphi}{\rightarrow} G$), contradicting $f \in G_0$. 
        % So $e \in G$ for every 
        % $G \in \bigcup_{\varphi\in P_3}\left(N_{\varphi}(G_0) \cap {\cal G'}\right)$. 
        In particular, this means that all graphs in $\bigcup_{\varphi\in P_3}\left(N_{\varphi}(G_0) \cap {\cal G'}\right)$ must contain $e$ but not $f$.
        % have no choice on the pair $(e,f)$. 
        Recall also that the graphs in $\mathcal{G}'$ have at most two choices for every pair in $\mathcal{C}_2$, and $|\mathcal{C}_2| \leq f_2(n)$ (by \cref{claim: fixed r-sets}). 
        It follows that $\abs{\bigcup_{\varphi\in P_3}\left(N_{\varphi}(G_0)\cap {\cal G'}\right)} \le 2^{f_2(n)-1}=2^{f_2(n)}/2$, contradicting \cref{eq: P3 contribution}.
        
            % For the first two cases, fix some $\varphi \in P_3$.
            % It holds that $\varphi(e)=f$. Suppose now that $\varphi(f)=e$.
            % Then, $e,f$ forms a 2-cycle in both $\tilde{\varphi}$ and $\tilde{\psi}$.
            % Clearly, any $G' \in \neighbor{G}{\varphi}\cap\mc{G}'$ must also contain both $e$ and $f$, causing a contradiction.
            % So, $\varphi(f)\neq e$, and the choosable good pair is $(e,f)$.
            % Then, for any $G' \in \neighbor{G}{\varphi} \cap \cal G'$, we have $f \notin G'\setminus G$, so $e={\varphi^{-1}}(f)\notin G\setminus G'$, i.e. $e \in G'$.
            % This means, all the graphs in $\bigcup_{\varphi\in P_3}\left(\neighbor{G}{\varphi}\cap {\cal G'}\right)$ contain edge $e$.
            % Since there are at most $\log_2f_2(n)$ 2-cycles of edges of $\tilde{\psi}$ (\cref{XX}) for these graphs to choose from,
            
        \paragraph{Case 2.2. $\{e,f\}$ is of type (iii).} 
            In this case $\{e,f\} = \{u_iv_j,v_iu_j\}$ for some $1 \leq i < j \leq \ell$. 
            Note that $e,f$ are not dependent on any particular choice of $\varphi \in P_3$.
            Without loss of generality, let us assume that $i=1,j=2$ and $e = u_1v_2, f = v_1u_2$. 
            Define 
            \begin{equation} \nonumber
                J=\{(i,j): 1 \le i \le 2 < j \le \ell, u_iv_j, v_iu_j \notin G_0\}, \quad 
                {\cal F}=\{G \in {\cal G'}: \exists (i,j)\in J. \; v_iu_j\in G \}.
            \end{equation}
            We proceed similarly to Case 1. First, note that $J \subseteq I$, so $|J| \leq |I| \leq \frac{c}{2}\sqrt{n}\ln n$ by \cref{claim: I is small}. As in Case 1, the graphs in $\mathcal{F}$ have at most $2^{|J|}-1$ choices on the edges in the set $\{u_iv_j,v_iu_j : (i,j) \in J\}$. Hence,
            $\abs{\cal F} \le \left(2^{\abs{J}}-1 \right) \cdot 2^{\abs{C_2}-\abs{J}}= \left( 1-2^{-\abs{J}} \right) \cdot 2^{f_2(n)}$.
            Now, it suffices to show that $\bigcup_{\varphi\in P_3}(N_{\varphi}(G_0) \cap \cal G')\subseteq \cal F$, as then, $\abs{\bigcup_{\varphi\in P_3}(N_{\varphi}(G_0) \cap \cal G')} \le \abs{\mc{F}}\le (1-n^{-c\sqrt{n}/2})2^{f_2(n)}$, contradicting \cref{eq: P3 contribution}.
            % We will now show that $\bigcup_{\varphi\in P_3}(N_{\varphi}(G_0) \cap \cal G')\subseteq \cal $, and this will contradict \cref{eq: P3 contribution} in the same way as in Case 1.

            So let us assume, for the sake of contradiction, that there exists $\varphi \in P_3$ and $G \in N_{\varphi}(G_0) \cap \mathcal{G}'$ such that $G \notin \mathcal{F}$. 
            We know that every graph in $\cal G'$ contains exactly one of the edges $e=u_1v_2$ and $f=v_1u_2$.
            Assume that $v_1u_2 \in G$ and $u_1v_2 \notin G$; the other case is similar by swapping the roles of $1,2$ and of $e,f$.
            As in Case 1, the cycle in $\varphi$ containing $u_2$ alternates between $U$ and $V$.
            % As $\varphi(u_2)=v_2$, there is a cycle in $\varphi$ containing $u_2$ and $v_2$.
            % As in Case 1, this cycle alternates between $U$ and $V$. 
            Let us denote this cycle by $(u_{i_1},v_{i_1},\dots,u_{i_k},v_{i_k})$, where $u_{i_1} = u_2$ and $v_{i_1} = \varphi(u_2)=v_2$ (that is, $i_1 = 2$). 
            Note that we do not make any assumption whether $u_1$ is in this cycle or not.
            % We may denote the vertices to be $\tilde{u}_1=u_2,\tilde{v}_1=v_2,\tilde{u}_2,\tilde{v}_2,\dots,\tilde{u}_\ell,\tilde{v}_\ell$ by order, where $\tilde{u}_i \in U$ and $\tilde{v}_i \in V$ for all $i$.
            We now show by induction that $u_1v_{i_j} \in G_0\setminus G$ for every $1 \leq j \leq k$.
            The base case $j=1$ holds by our assumptions that $u_1v_{i_1} = u_1v_2 = e \in G_0$ and $u_1v_2 \notin G$. 
            Suppose now that we proved that $u_1v_{i_{j-1}} \in G_0 \setminus G$ for some $2 \leq j \leq k$.
            Then, $v_1u_{i_{j}}={\varphi}(u_1v_{i_{j-1}}) \in G \setminus G_0$, because $G_0 \overset{\varphi}{\rightarrow} G$.
            % As edges $v_1\tilde{u}_{i+1}$ and $u_1\tilde{v}_{i+1}$ form a 2-cycle in $\tilde{\psi}$,
            As $(v_1u_{i_{j}}, u_1v_{i_{j}}) \in \mathcal{C}_2$,
            every graph in $\cal G'$ contains exactly one of the edges $v_1u_{i_{j}}, u_1v_{i_{j}}$.
            In particular, $u_1v_{i_{j}} \notin G$.
            Observe that $i_j \neq 2$ (because $i_1 = 2$ and $j \neq 1$). Also, $i_j \neq 1$, because otherwise we would have $v_1u_1 = v_1u_{i_j} \in G \setminus G_0$, in contradiction to the assumption that $G_0$ and $G$ agree on all edges in $\mathcal{C}_1$. So $i_j > 2$. 
            Now, if $u_1v_{i_{j}} \notin G_0$, 
            % assuming $\tilde{u}_{i+1}=u_k$ (clearly, $k\neq 1,2$), 
            then $(1,i_{j})\in J$ and hence $G \in \cal F$ (as $v_1u_{i_j} \in G$), and this contradicts our assumption $G \notin \mathcal{F}$.
            This means $u_1v_{i_{j}} \in G_0$ must hold.
            By the former discussion, we proved that $u_1v_{i_{j}} \in G_0 \setminus G$, completing the induction step.
            Now, taking $j = k$, we get $u_1v_{i_k} \in G_0 \setminus G$.
            Therefore, $f=v_1u_2=v_1u_{i_1}=\varphi(u_1v_{i_k})\in G \setminus G_0$, contradicting our assumption that $f \in G_0$.
            This proves that indeed $\bigcup_{\varphi\in P_3}(N_{\varphi}(G_0) \cap \cal G')\subseteq {\cal F}$, completing the proof in Case 2 and hence the entire proof of the lemma. 
\end{proof}

\begin{proof}[Proof of \cref{theorem: stablility} for $r=2$]
    Let $\cal G'\subseteq \cal G$ be a largest involution clique in $\mathcal{G}$.
    By assumption, $\mathcal{G}$ is not an involution clique, hence $\mathcal{G} \neq\mc{G}'$.
    We consider the following three cases depending on the size of $\cal G'$. 
    Let $c_1$ and $c_2$ be the constants from \cref{lemma: any non-small involution clique is large,lemma: no very large involution clique}, respectively.
    \begin{description}
        \item[Case 1:] $\abs{\cal G'}\ge\left(1-n^{-c_2\sqrt{n}}\right)2^{f_2(n)}$.
            Then, by \cref{lemma: no very large involution clique}, $\cal G$ itself is an involution clique, contradicting our assumption.
        \item[Case 2:] $2^{f_2(n)}\cdot e^{-c_1n} \leq \abs{\cal G'} \le \left(1-n^{-c_2\sqrt{n}}\right)2^{f_2(n)}$.
        By \cref{lemma: any non-small involution clique is large},  $\abs{{\cal G}\setminus{\cal G'}} \leq 2^{f_2(n)}\cdot e^{-c_1n}$.
        When $n$ is sufficiently large, the desired inequality holds as 
        $$
            \abs{\cal G}
            =\abs{\cal G'}+\abs{{\cal G}\setminus{\cal G'}}
            \le \left(1-n^{-c_2\sqrt{n}}+e^{-c_1n}\right)\cdot 2^{f_2(n)}
            \leq \left(1-n^{-2c_2\sqrt{n}}\right)\cdot 2^{f_2(n)}.
        $$
        \item[Case 3:] $\abs{\cal G'} < 2^{f_2(n)}\cdot e^{-c_1n}$.
        Then \cref{theorem without large involution clique} with $\delta=\min(c_1n,n/250)$ implies $\abs{\mc{G}} = 2^{f_2(n)}\cdot e^{-\Omega(n)}$, as desired.
    \end{description}
\end{proof}

\section{Concluding remarks}\label{sec:concluding}
We showed that if $n \geq n_0(r)$, then $2^{f_r(n)}$ is the largest size of a difference-isomorphic $r$-graph family on $[n]$. 
It would be interesting to determine how large $n_0(r)$ should be. 
Let us observe that for $n=r+1$, our result does not hold. Indeed, let $\mathcal{G}$ be the family consisting of all $r$-graphs on $[r+1]$ with exactly $\floor{(r+1)/2}$ edges. 
We claim that $\mathcal{G}$ is difference isomorphic. 
To this end, for every $G \in \mathcal{G}$, we consider the set $S(G) := \{i : [r+1]\setminus \{i\} \in E(G)\}$, i.e., $S(G)$ is the set of complements of edges in $G$, which are singletons. 
Then $S(G)$ is a subset of $[r+1]$ of size $\ceil{(r+1)/2}$. 
For every $G_1,G_2 \in \mathcal{G}$, set $S_1 = S(G_1), S_2 = S(G_2)$, and take a permutation $\varphi \in S_{[r+1]}$ with $\varphi(S_1\setminus S_2)=S_2\setminus S_1$ (we know $\abs{S_1\setminus S_2}=\abs{S_2\setminus S_1}$ because $\abs{S_1}=\abs{S_2}$). 
Now, let $e \in G_1 \setminus G_2$, so $e = [r+1] \setminus \{i\}$ where $i \in S_1 \setminus S_2$. 
Then $j := \varphi(i) \in S_2 \setminus S_1$, so $\varphi(e) = [r+1] \setminus \varphi(i) = [r+1] \setminus \{j\} \in G_2 \setminus G_1$. 
This shows that $\varphi(G_1\setminus G_2)\subseteq G_2\setminus G_1$.
Using that $|G_1\setminus G_2|=|G_2\setminus G_1|$, it holds that $\varphi(G_1\setminus G_2)=G_2\setminus G_1$. 
We see that $\cal G$ is a difference-isomorphic family of size $\binom{r+1}{\floor{(r+1)/2}}$.
It is easy to check that $\binom{r+1}{\floor{(r+1)/2}}> 2^{f_r(r+1)}$ for all $r \ge 2$, so \cref{thm:tight} does not hold when $n=r+1$. 

% However, for small $n$, $2^{f_r(n)}$ is no longer the right answer.
% For example, when $n=r+1$, the family $\cal G$ consisting of all $r$-graphs on $[r+1]$ with $\floor{(r+1)/2}$ edges is difference-isomorphic.
% To see this, for any $G_1,G_2 \in \cal G$, it holds that $S_1:=\{i:[r+1]\setminus i \in G_1\}$ and $S_2:=\{i:[r+1]\setminus i \in G_2\}$ are both subsets of $[r+1]$ of size $\ceil{(r+1)/2}$.
% So there exists a permutation $\varphi\in S_{[r+1]}$ such that $\varphi(S_1\setminus S_2)=S_2\setminus S_1$, which gives $\varphi(G_1\setminus G_2)=G_2\setminus G_1$.
% So, $\cal G$ is difference-isomorphic of size $\binom{r+1}{\floor{(r+1)/2}}$.
% It is easy to check that $\binom{r+1}{\floor{(r+1)/2}}> 2^{f_r(r+1)}$ for all $r \ge 2$.

The argument from the previous paragraph, of considering the complements of edges, is more general; it shows that if $\mathcal{G}$ is a difference-isomorphic family of $r$-graphs on $[n]$, then $\{\{[n] \setminus e : e \in G\} : G \in \mathcal{G}\}$ is a difference isomorphic family of $(n-r)$-uniform graphs on $[n]$. Hence, letting $F_r(n)$ denote the size of the largest difference-isomorphic family of $r$-graphs on $[n]$, we have that $F_r(n)=F_{n-r}(n)$. Note that $F_1(n)=\binom{n}{\floor{n/2}}$ (``1-uniform" hypergraphs are simply subsets, and two subsets and difference isomorphic if and only if they have the same size), and this recovers the fact that $F_r(r+1)=F_1(r+1)=\binom{r+1}{\floor{(r+1)/2}}$ from the previous paragraph. 
Moreover, one can check from \cref{eq: f_r(n)} that $f_r(n) = f_{n-r}(n)$. Hence, whenever we have $F_r(n) = 2^{f_r(n)}$, we also have $F_{n-r}(n) = 2^{f_{n-r}(n)}$.
Therefore, by \cref{thm:tight}, we have $F_r(r+2)=F_2(r+2)=2^{f_2(r+2)} = 2^{f_r(r+2)}$ for sufficiently large $r$. It would be interesting to understand in general for which values of $n$ (depending on $r$) it holds that $F_r(n) = 2^{f_r(n)}$.
% Maybe this is an indication that $F_r(n)=2^{f_r(n)}$ except for trivial cases.
% Hence, it might be interesting to determine the whole $F_r(n)$.

\bibliographystyle{plain} 
% \bibliography{library}

\begin{thebibliography}{10}
\bibitem{Alon_Codes}
Noga Alon. Graph-codes. arXiv preprint arXiv:2301.13305.

\bibitem{Alon_connCodes}
Noga Alon. Connectivity graph-codes. arXiv preprint arXiv:2308.07653.

\bibitem{AGKMS}
N. Alon, A. Gujgiczer, J. K\"orner, A. Milojevi\'c and G. Simonyi. Structured codes of graphs.  SIAM J. Discrete Math. 37(1), 379-403 (2023).

\bibitem{BF_Book}
L. Babai and P. Frankl. {\bf Linear algebra methods in combinatorics}, 2020. 

\bibitem{BGMW_Codes}
B. Bai, Y. Gao, J. Ma and Y. Wu.
Phase transitions of structured codes of graphs. arXiv preprint arXiv 2307.08266.

\bibitem{Berger_et_al}
A. Berger, R. Berkowitz, P. Devlin, M. Doppelt, S. Durham, T. Murthy and H. Vemuri. Connected-Intersecting Families of Graphs. arXiv preprint arXiv:1901.01616, 2019.

\bibitem{BZ_K4}
A. Berger and Y. Zhao. $K_4$-intersecting families of graphs. Journal of Combinatorial Theory, Series B, 163 (2023), 112-132.

\bibitem{CGFS}
F. R. Chung, R. L. Graham, P. Frankl and J. B. Shearer. Some intersection theorems for ordered sets and graphs. Journal of Combinatorial Theory, Series A, 43(1) (1986), 23-37.

\bibitem{Ellis}
D. Ellis. Intersection Problems in Extremal Combinatorics: Theorems, Techniques and Questions Old and New.  ``Surveys in combinatorics 2022'', London Math. Soc. Lecture Note
Ser., Vol. 481 (2022) 115–173.

\bibitem{EFF}
D. Ellis, Y. Filmus and E. Friedgut. Triangle-intersecting families of graphs. J. Eur. Math. Soc. (JEMS) 14 (2012), 841–885.

\bibitem{EKR}
P. Erd{\H{o}}s, H. Ko and R. Rado. Intersection theorems for systems of finite sets. Quart. J. Math. Oxford Ser. 12 (1961) 313–320.

\bibitem{FT_Book}
P. Frankl and N. Tokushige. {\bf Extremal problems for finite sets} (Vol. 86). American Mathematical Soc. (2018).

\bibitem{FW}
P. Frankl and R. M. Wilson. Intersection Theorems with Geometric Consequences. Combinatorica, 1981,
vol. 1, no. 4, pp. 357–368.

\bibitem{Gowers}
W. T. Gowers. Blog post: \url{https://gowers.wordpress.com/2009/11/14/the-first-unknown-case-of-polynomial-dhj/}

\bibitem{IRT}
I. Leader, \v{Z}. Ran\dj elovi\'c and T. S. Tan. A Note on Hamiltonian-Intersecting Families of Graphs. arXiv preprint arXiv:2309.00757, 2023.

\bibitem{RC_W}
D. K. Ray-Chaudhuri and R. M. Wilson. On t-designs. Osaka J. Math.
12 (1975), 737-744.

\bibitem{SS}
M. Simonovits and V. T. S\'os. Intersection theorems on structures. Annals of Discrete Mathematics, 6 (1980), 301-313.

\end{thebibliography}

\appendix
\section{Tightness for the stability result}\label{sec:stability tightness}

\begin{proposition}\label{prop: tightness of stability}
    Let $r\ge 2$ and $n$ be sufficiently large in terms of $r$.
    Then, there exists a difference-isomorphic family $\mc{G}$ of $r$-graphs on $[n]$ such that $\mc{G}$ is not an involution clique, and $\abs{G}=2^{f_r(n)-O(n^{r-2})}$.
\end{proposition}
\begin{proof}
    Let $\psi\in S_n$ be the involution such that 
    % $\{2i-1,2i\}\in \mc{C}_1(\psi)$ 
    $\psi(2i-1) = 2i, \psi(2i) = 2i-1$
    for all $1 \le i \le \floor{\frac{n-1}{2}}$, and that $\psi(i)=i$ for $2\floor{\frac{n-1}{2}}<i\le n$.
    Let $a$ be the number of fixed points of $\psi$. We know $a=1$ if $n$ is odd while $a=2$ if $n$ is even.
    By \cref{eq:fixed points of tilde-psi}, 
    $$|\mathcal{C}_1(\psi)| = \sum_{\substack{0 \leq i \leq a \\ i \equiv r \; \text{mod } 2}} 
    \binom{a}{i}\binom{(n-a)/2}{(r-i)/2}$$
    A similar computation as in \cref{claim: fixed r-sets} shows that $\abs{\mc{C}_1(\psi)}=\binom{\floor{n/2}}{\floor{r/2}}$ if $n$ is odd; $\abs{\mc{C}_1(\psi)}=2\binom{n/2-1}{\floor{r/2}}$ if $n$ is even and $r$ is odd; $\abs{\mc{C}_1(\psi)}=\binom{n/2-1}{r/2}+\binom{n/2-1}{r/2-1}$ if $n$ and $r$ are even.
    Hence, $\abs{\mc{C}_2(\psi)}=f_2(n)$ if $r = 2$ and $\abs{\mc{C}_2(\psi)}=f_r(n)-O(n^{\floor{r/2}})$ if $r \ge 3$.
    
    Let $\varphi \in S_n$ be a permutation such that 
    $\varphi(1)=2,\varphi(2)=3,\varphi(3)=4,\varphi(4)=1$ and $\varphi(i)=\psi(i)$ for all $5 \le i \le n$.
    % Recall that $\tilde{\psi},\tilde{\varphi}\in S_{\binom{[n]}{r}}$.
    Define $e_0=\{3,4,\dots,r+2\}$ if $r$ is even and $e_0=\{3,4,\dots,r+1,n\}$ if $r$ is odd.
    By definition, $e_0 \in \mc{C}_1(\psi)$.
    Let $f_0=\varphi(e_0)$.
    Then, $f_0\cap\{1,2,3,4\}=\{1,4\}$ and $f_0$ is contained in some 2-cycle of $\tilde{\psi}$.
    Pick any $r$-graph $G_0$ on $[n]$ with the following properties.
    \begin{itemize}
        \item For every $\{e,f\}\in\mc{C}_2({\psi})$ with $e \subseteq\{5,\dots,n\}$ (so $f\subseteq \{5,\dots,n\}$), $G_0$ contains one between them;
        \item for every $\{e,f\}\in\mc{C}_2({\psi})$ with $\abs{e\cap \{ 1,\dots,4\}}=1$ (so $\abs{f\cap \{ 1,\dots,4\}}=1$), $G_0$ contains $e$ if $e \cap\{1,3\}\neq\emptyset$ while $G_0$ contains $f$ if $f\cap \{1,3\}\neq\emptyset$ (clearly, $G_0$ contains one of them);
        \item Among all edges incident to at least 2 vertices in $\{1,2,3,4\}$, only $e_0\in G_0$;
        \item $G$ contains no edges in $\mc{C}_1({\psi})\setminus\{e_0\}$.
    \end{itemize}
    Consider a family $\cal G'$ of $r$-graphs on $[n]$, including all $r$-graphs $G$ such that 
    \begin{itemize}
        \item for every $\{e,f\}\in\mc{C}_2({\psi})$ with $e \subseteq\{5,\dots,n\}$ (so $f\subseteq \{5,\dots,n\}$), $G$ contains one between them;
        \item for every $\{e,f\}\in\mc{C}_2({\psi})$ with $\abs{e\cap \{ 1,\dots,4\}}=1$ (so $\abs{f\cap \{ 1,\dots,4\}}=1$), $G_0$ contains $e$ if $e \cap\{1,3\}\neq\emptyset$ while $G$ contains $f$ if $f\cap \{1,3\}\neq\emptyset$;
        \item Among all edges incident to at least 2 vertices in $\{1,2,3,4\}$, only $f_0\in G$;
        \item $G$ contains no edges in $\mc{C}_1({\psi})$
    \end{itemize}
    It is easy to see that every $G\in \mc{G}'$ contains the same number of edges in every 2-cycle $\{e,f\}\in\mc{C}_2({\psi})$ and contains none of the edges in $\mc{C}_1(\psi)$, so \cref{prop: involution relation} implies $\mc{G}'$ forms a $\psi$-clique.
    In addition, the number of 2-cycles that $G\in\mc{G}'$ has a choice is at least $\mc{C}_2(\psi)-\binom{4}{2}\binom{n}{r-2}$.
    If $r=2$, $\abs{\mc{G}'}\ge 2^{f_2(n)-6}$; if $r \ge 3$, $\abs{\mc{G}'}\ge 2^{\mc{C}_2(\psi)-\binom{4}{2}\binom{n}{r-2}}=2^{f_r(n)-O(n^{\floor{r/2}})-O(n^{r-2})}=2^{f_r(n)-O(n^{r-2})}$.
    So, $\abs{\mc{G}'}\ge 2^{f_r(n)-O(n^{r-2})}$ holds unconditionally.

    Now, consider $\mc{G}:=\mc{G}'\cup\{G_0\}$.
    We will complete the proof by showing that $\mc{G}:=\mc{G}'\cup\{G_0\}$ is a difference-isomorphic graph family of size $2^{f_r(n)-O(n^{r-2})}$ such that $\mc{G}$ is not an involution clique.
    First, we claim that $\mc{G}$ is difference-isomorphic.
    It suffices to show that $G_0 \overset{\varphi}{\rightarrow} G$ for all $G \in \mc{G}'$.
    % \mc{G}'\subseteq N_\varphi(G_0)$.
    Indeed, $(e,f)\in\binom{[n]}{r}\times\binom{[n]}{r}$ is a choosable pair for $(G,\varphi)$ if and only if one of the following holds: $\{e,f\}\in\mc{C}_2(\psi)$ and $e\subseteq \{5,\dots,n\}$; $\{e,f\}\in\mc{C}_2(\psi)$ and $e\cap\{1,2,3,4\}=\{1\}$ or $\{3\}$; $e=e_0$.
    It is easy to check that (i)(ii) (of \cref{lemma: properties of choosable pairs}) are satisfied for all $G\in\mc{G}'$.
    Hence, $G_0 \overset{\varphi}{\rightarrow} G$, meaning $\mc{G}$ is difference-isomorphic.

    Second, we show that $\mc{G}$ is not an involution clique.
    Suppose for contradiction that $\mc{G}$ is a $\psi'$-clique some involution $\psi'$.
    We have $e_{\psi}(N_{\psi'}(G_0))\ge e_\psi(\mc{G}') \ge 2^{2f_r(n)-O(n^{r-2})}\gg 2^{2f_r(n)}\cdot e^{-\binom{n}{r-1}/100}$, when $n$ is sufficiently large in terms of $r$.
    Then, \cref{lemma: general counts for exceptional r-graph} implies that $\psi=\psi'$.
    % $\psi,\psi'$ share all the 1-cycles and that for every 2-cycle $xy$ of $\psi$, $\psi'(x)=y$ or $\psi'(y)=x$, either of which means $xy$ is also a 2-cycle of $\psi'$.
    % Hence, $\psi=\psi'$.
    But since $e_0 \in \mc{C}_1(\psi)$, \cref{prop: involution relation} implies that all graphs in $\mc{G}'$ must also contain $e_0$, which is impossible.
    So, $\mc{G}$ is indeed not an involution clique.
\end{proof}

\end{document}